\documentclass[11pt]{article}
\usepackage[alphabetic,maxbibnames=9,isbn=false,url=false]{mypreambles}
\usepackage{stackengine}
\usetikzlibrary{intersections}
\crefname{appsec}{Appendix}{Appendices}
\usepackage{listings}
\usepackage{xcolor}
\definecolor{maccolor}{rgb}{0.3,0.3,0.8}
\lstdefinelanguage{Macaulay2}
{
basicstyle={\ttfamily},
keywordstyle={\color{maccolor!80!black}},
commentstyle={\color{gray}},
stringstyle={\color{red!40!black}},
rulecolor=\color{maccolor},
basewidth={1.2ex}, 
sensitive=false,
morecomment=[l]{--},
morecomment=[s]{-*}{*-},
morestring=[b]",
escapechar={`},
escapebegin={\rmfamily},
morekeywords={about,abs,AbstractToricVarieties,accumulate,Acknowledgement,acos,acosh,acot,addCancelTask,addDependencyTask,addEndFunction,addHook,AdditionalPaths,addStartFunction,addStartTask,Adjacent,adjoint,AdjointIdeal,AffineVariety,AfterEval,AfterNoPrint,AfterPrint,agm,AInfinity,alarm,AlgebraicSplines,Algorithm,Alignment,all,AllCodimensions,allowableThreads,ambient,analyticSpread,Analyzer,AnalyzeSheafOnP1,ancestor,ancestors,ANCHOR,and,andP,AngleBarList,ann,annihilator,antipode,any,append,applicationDirectory,applicationDirectorySuffix,apply,applyKeys,applyPairs,applyTable,applyValues,apropos,argument,Array,arXiv,Ascending,ascii,asin,asinh,ass,assert,associatedGradedRing,associatedPrimes,AssociativeAlgebras,AssociativeExpression,atan,atan2,atEndOfFile,Authors,autoload,AuxiliaryFiles,backtrace,Bag,Bareiss,baseFilename,BaseFunction,baseName,baseRing,baseRings,BaseRow,BasicList,basis,BasisElementLimit,Bayer,BeforePrint,beginDocumentation,BeginningMacaulay2,Benchmark,benchmark,Bertini,BesselJ,BesselY,betti,BettiCharacters,BettiTally,between,BGG,BIBasis,Binary,BinaryOperation,Binomial,binomial,BinomialEdgeIdeals,Binomials,BKZ,BlockMatrix,BLOCKQUOTE,BODY,Body,BoijSoederberg,BOLD,Book3264Examples,Boolean,BooleanGB,borel,Boxes,BR,break,Browse,Bruns,cache,CacheExampleOutput,CacheFunction,CacheTable,cacheValue,CallLimit,cancelTask,capture,catch,Caveat,CC,CDATA,ceiling,Center,centerString,Certification,ChainComplex,chainComplex,ChainComplexExtras,ChainComplexMap,ChainComplexOperations,ChangeMatrix,char,CharacteristicClasses,characters,charAnalyzer,check,CheckDocumentation,chi,Chordal,class,Classic,clean,clearAll,clearEcho,clearOutput,close,closeIn,closeOut,ClosestFit,CODE,code,codim,CodimensionLimit,coefficient,CoefficientRing,coefficientRing,coefficients,Cofactor,CohenEngine,CohenTopLevel,CoherentSheaf,CohomCalg,cohomology,coimage,CoincidentRootLoci,coker,cokernel,collectGarbage,columnAdd,columnate,columnMult,columnPermute,columnRankProfile,columnSwap,combine,Command,commandInterpreter,commandLine,COMMENT,commonest,commonRing,comodule,CompactMatrix,compactMatrixForm,CompiledFunction,CompiledFunctionBody,CompiledFunctionClosure,Complement,complement,complete,CompleteIntersection,CompleteIntersectionResolutions,Complexes,ComplexField,components,compose,compositions,compress,concatenate,conductor,ConductorElement,cone,Configuration,ConformalBlocks,conjugate,connectionCount,Consequences,Constant,Constants,constParser,content,continue,contract,Contributors,ConvexInterface,conwayPolynomial,ConwayPolynomials,copy,copyDirectory,copyFile,copyright,Core,CorrespondenceScrolls,cos,cosh,cot,CotangentSchubert,cotangentSheaf,coth,cover,coverMap,cpuTime,createTask,Cremona,csc,csch,current,currentColumnNumber,currentDirectory,currentFileDirectory,currentFileName,currentLayout,currentLineNumber,currentPackage,currentString,currentTime,Cyclotomic,Database,Date,DD,dd,deadParser,debug,debugError,DebuggingMode,debuggingMode,debugLevel,DecomposableSparseSystems,Decompose,decompose,deepSplice,Default,default,defaultPrecision,Degree,degree,degreeLength,DegreeLift,DegreeLimit,DegreeMap,DegreeOrder,DegreeRank,Degrees,degrees,degreesMonoid,degreesRing,delete,demark,denominator,Dense,Density,Depth,depth,Descending,Descent,Describe,describe,Description,det,determinant,DeterminantalRepresentations,DGAlgebras,diagonalMatrix,diameter,Dictionary,dictionary,dictionaryPath,diff,DiffAlg,difference,dim,directSum,disassemble,discriminant,dismiss,Dispatch,distinguished,DIV,Divide,divideByVariable,DivideConquer,DividedPowers,Divisor,DL,Dmodules,do,doc,docExample,docTemplate,document,DocumentTag,Down,drop,DT,dual,eagonNorthcott,EagonResolution,echoOff,echoOn,EdgeIdeals,edit,EigenSolver,eigenvalues,eigenvectors,eint,EisenbudHunekeVasconcelos,elapsedTime,elapsedTiming,elements,Eliminate,eliminate,Elimination,EliminationMatrices,EllipticCurves,EllipticIntegrals,else,EM,Email,End,end,endl,endPackage,Engine,engineDebugLevel,EngineRing,EngineTests,entries,EnumerationCurves,environment,Equation,EquivariantGB,erase,erf,erfc,error,errorDepth,euler,EulerConstant,eulers,even,EXAMPLE,ExampleFiles,ExampleItem,examples,ExampleSystems,Exclude,exec,exit,exp,expectedReesIdeal,expm1,exponents,export,exportFrom,exportMutable,Expression,expression,Ext,extend,ExteriorIdeals,ExteriorModules,exteriorPower,Factor,factor,false,Fano,FastMinors,FastNonminimal,FGLM,File,fileDictionaries,fileExecutable,fileExists,fileExitHooks,fileLength,fileMode,FileName,FilePosition,fileReadable,fileTime,fileWritable,fillMatrix,findFiles,findHeft,FindOne,findProgram,findSynonyms,FiniteFittingIdeals,First,first,firstkey,FirstPackage,fittingIdeal,flagLookup,FlatMonoid,flatten,flattenRing,Flexible,flip,floor,flush,fold,FollowLinks,for,forceGB,fork,FormalGroupLaws,Format,format,formation,FourierMotzkin,FourTiTwo,fpLLL,frac,fraction,FractionField,frames,FrobeniusThresholds,from,fromDividedPowers,fromDual,Function,FunctionApplication,FunctionBody,functionBody,FunctionClosure,FunctionFieldDesingularization,fusePairs,futureParser,GaloisField,gb,GBDegrees,gbRemove,gbSnapshot,gbTrace,gcd,gcdCoefficients,gcdLLL,GCstats,genera,GeneralOrderedMonoid,GenerateAssertions,generateAssertions,generator,generators,Generic,GenericInitialIdeal,genericMatrix,genericSkewMatrix,genericSymmetricMatrix,gens,genus,get,getc,getChangeMatrix,getenv,getGlobalSymbol,getNetFile,getNonUnit,getPrimeWithRootOfUnity,getSymbol,getWWW,GF,gfanInterface,Givens,GKMVarieties,GLex,Global,global,globalAssign,globalAssignFunction,GlobalAssignHook,globalAssignment,globalAssignmentHooks,GlobalDictionary,GlobalHookStore,globalReleaseFunction,GlobalReleaseHook,Gorenstein,GradedLieAlgebras,GradedModule,gradedModule,GradedModuleMap,gradedModuleMap,gramm,GraphicalModels,GraphicalModelsMLE,Graphics,graphIdeal,graphRing,Graphs,Grassmannian,GRevLex,GroebnerBasis,groebnerBasis,GroebnerBasisOptions,GroebnerStrata,GroebnerWalk,groupID,GroupLex,GroupRevLex,GTZ,Hadamard,handleInterrupts,HardDegreeLimit,hash,HashTable,hashTable,HEAD,HEADER1,HEADER2,HEADER3,HEADER4,HEADER5,HEADER6,HeaderType,Heading,Headline,Heft,heft,Height,height,help,Hermite,hermite,Hermitian,HH,hh,HigherCIOperators,HighestWeights,Hilbert,hilbertFunction,hilbertPolynomial,hilbertSeries,HodgeIntegrals,hold,Holder,Hom,homeDirectory,HomePage,Homogeneous,Homogeneous2,homogenize,homology,homomorphism,HomotopyLieAlgebra,hooks,horizontalJoin,HorizontalSpace,HR,HREF,HTML,html,httpHeaders,Hybrid,HyperplaneArrangements,Hypertext,hypertext,HypertextContainer,HypertextParagraph,icFracP,icFractions,icMap,icPIdeal,id,Ideal,ideal,idealizer,identity,if,IgnoreExampleErrors,ii,image,imaginaryPart,IMG,ImmutableType,importFrom,in,incomparable,Increment,independentSets,indeterminate,IndeterminateNumber,Index,index,indexComponents,IndexedVariable,IndexedVariableTable,indices,inducedMap,inducesWellDefinedMap,InexactField,InexactFieldFamily,InexactNumber,InfiniteNumber,infinity,info,InfoDirSection,infoHelp,Inhomogeneous,input,Inputs,insert,installAssignmentMethod,installedPackages,installHilbertFunction,installMethod,installMinprimes,installPackage,InstallPrefix,instance,instances,IntegralClosure,integralClosure,integrate,IntermediateMarkUpType,interpreterDepth,intersect,intersectInP,Intersection,intersection,interval,InvariantRing,inverse,InverseMethod,inversePermutation,Inverses,inverseSystem,InverseSystems,Invertible,InvolutiveBases,irreducibleCharacteristicSeries,irreducibleDecomposition,isAffineRing,isANumber,isBorel,isCanceled,isCommutative,isConstant,isDirectory,isDirectSum,isEmpty,isField,isFinite,isFinitePrimeField,isFreeModule,isGlobalSymbol,isHomogeneous,isIdeal,isInfinite,isInjective,isInputFile,isIsomorphism,isLinearType,isListener,isLLL,isMember,isModule,isMonomialIdeal,isNormal,isOpen,isOutputFile,isPolynomialRing,isPrimary,isPrime,isPrimitive,isPseudoprime,isQuotientModule,isQuotientOf,isQuotientRing,isReady,isReal,isReduction,isRegularFile,isRing,isSkewCommutative,isSorted,isSquareFree,isStandardGradedPolynomialRing,isSubmodule,isSubquotient,isSubset,isSupportedInZeroLocus,isSurjective,isTable,isUnit,isWellDefined,isWeylAlgebra,ITALIC,Iterate,Jacobian,jacobian,jacobianDual,Jets,Join,join,Jupyter,K3Carpets,K3Surfaces,Keep,KeepFiles,KeepZeroes,ker,kernel,kernelLLL,kernelOfLocalization,Key,keys,Keyword,Keywords,kill,koszul,Kronecker,KustinMiller,LABEL,last,lastMatch,LATER,LatticePolytopes,Layout,lcm,leadCoefficient,leadComponent,leadMonomial,leadTerm,Left,left,length,LengthLimit,letterParser,Lex,LexIdeals,LI,Licenses,LieTypes,lift,liftable,Limit,limitFiles,limitProcesses,Linear,LinearAlgebra,LinearTruncations,lineNumber,lines,LINK,linkFile,List,list,listForm,listLocalSymbols,listSymbols,listUserSymbols,LITERAL,LLL,LLLBases,lngamma,load,loadDepth,LoadDocumentation,loadedFiles,loadedPackages,loadPackage,Local,local,localDictionaries,LocalDictionary,localize,LocalRings,locate,log,log1p,LongPolynomial,lookup,lookupCount,LowerBound,LUdecomposition,M0nbar,M2CODE,Macaulay2Doc,makeDirectory,MakeDocumentation,makeDocumentTag,MakeHTML,MakeInfo,MakeLinks,makePackageIndex,MakePDF,makeS2,Manipulator,map,MapExpression,MapleInterface,markedGB,Markov,MarkUpType,match,mathML,Matrix,matrix,MatrixExpression,Matroids,max,maxAllowableThreads,maxExponent,MaximalRank,maxPosition,MaxReductionCount,MCMApproximations,member,memoize,memoizeClear,memoizeValues,MENU,merge,mergePairs,META,method,MethodFunction,MethodFunctionBinary,MethodFunctionSingle,MethodFunctionWithOptions,methodOptions,methods,midpoint,min,minExponent,mingens,mingle,minimalBetti,MinimalGenerators,MinimalMatrix,minimalPresentation,minimalPresentationMap,minimalPresentationMapInv,MinimalPrimes,minimalPrimes,minimalReduction,Minimize,minimizeFilename,MinimumVersion,minors,minPosition,minPres,minprimes,Minus,minus,Miura,MixedMultiplicity,mkdir,mod,Module,module,ModuleDeformations,modulo,MonodromySolver,Monoid,monoid,MonoidElement,Monomial,MonomialAlgebras,monomialCurveIdeal,MonomialIdeal,monomialIdeal,MonomialIntegerPrograms,MonomialOrbits,MonomialOrder,Monomials,monomials,MonomialSize,monomialSubideal,moveFile,multidegree,multidoc,multigraded,MultigradedBettiTally,MultiGradedRationalMap,multiplicity,MultiplicitySequence,MultiplierIdeals,MultiplierIdealsDim2,MultiprojectiveVarieties,mutable,MutableHashTable,mutableIdentity,MutableList,MutableMatrix,mutableMatrix,NAGtypes,Name,nanosleep,Nauty,NautyGraphs,NCAlgebra,NCLex,needs,needsPackage,Net,net,NetFile,netList,new,newClass,newCoordinateSystem,NewFromMethod,newline,NewMethod,newNetFile,NewOfFromMethod,NewOfMethod,newPackage,newRing,nextkey,nextPrime,nil,NNParser,NoetherianOperators,NoetherNormalization,NonminimalComplexes,nonspaceAnalyzer,NoPrint,norm,normalCone,Normaliz,NormalToricVarieties,not,Nothing,notify,notImplemented,NTL,null,nullaryMethods,nullhomotopy,nullParser,nullSpace,Number,number,NumberedVerticalList,numcols,numColumns,numerator,numeric,NumericalAlgebraicGeometry,NumericalCertification,NumericalImplicitization,NumericalLinearAlgebra,NumericalSchubertCalculus,numericInterval,NumericSolutions,numgens,numRows,numrows,odd,oeis,of,ofClass,OL,OldPolyhedra,OldToricVectorBundles,on,OneExpression,OnlineLookup,OO,oo,ooo,oooo,openDatabase,openDatabaseOut,openFiles,openIn,openInOut,openListener,OpenMath,openOut,openOutAppend,operatorAttributes,Option,OptionalComponentsPresent,optionalSignParser,Options,options,OptionTable,optP,or,Order,order,OrderedMonoid,orP,OutputDictionary,Outputs,override,pack,Package,package,PackageCitations,PackageDictionary,PackageExports,PackageImports,PackageTemplate,packageTemplate,pad,pager,PairLimit,pairs,PairsRemaining,PARA,Parametrization,parent,Parenthesize,Parser,Parsing,part,Partition,partition,partitions,parts,path,pdim,peek,PencilsOfQuadrics,Permanents,permanents,permutations,pfaffians,PHCpack,PhylogeneticTrees,pi,PieriMaps,pivots,PlaneCurveSingularities,plus,poincare,poincareN,Points,polarize,poly,Polyhedra,Polymake,PolynomialRing,Posets,Position,position,positions,PositivityToricBundles,POSIX,Postfix,Power,power,powermod,PRE,Precision,precision,Prefix,prefixDirectory,prefixPath,preimage,prepend,presentation,pretty,primaryComponent,PrimaryDecomposition,primaryDecomposition,PrimaryTag,PrimitiveElement,Print,print,printerr,printingAccuracy,printingLeadLimit,printingPrecision,printingSeparator,printingTimeLimit,printingTrailLimit,printString,printWidth,processID,Product,product,ProductOrder,profile,profileSummary,Program,programPaths,ProgramRun,Proj,Projective,ProjectiveHilbertPolynomial,projectiveHilbertPolynomial,ProjectiveVariety,promote,protect,Prune,prune,PruneComplex,pruningMap,Pseudocode,pseudocode,pseudoRemainder,Pullback,PushForward,pushForward,Python,QQ,QQParser,QRDecomposition,QthPower,Quasidegrees,QuaternaryQuartics,QuillenSuslin,quit,Quotient,quotient,quotientRemainder,QuotientRing,Radical,radical,RadicalCodim1,radicalContainment,RaiseError,random,RandomCanonicalCurves,RandomComplexes,RandomCurves,RandomCurvesOverVerySmallFiniteFields,RandomGenus14Curves,RandomIdeals,randomKRationalPoint,RandomMonomialIdeals,randomMutableMatrix,RandomObjects,RandomPlaneCurves,RandomPoints,RandomSpaceCurves,Range,rank,RationalMaps,RationalPoints,RationalPoints2,ReactionNetworks,read,readDirectory,readlink,readPackage,RealField,RealFP,realPart,realpath,RealQP,RealQP1,RealRoots,RealRR,RealXD,recursionDepth,recursionLimit,Reduce,reducedRowEchelonForm,reduceHilbert,reductionNumber,ReesAlgebra,reesAlgebra,reesAlgebraIdeal,reesIdeal,References,ReflexivePolytopesDB,regex,regexQuote,registerFinalizer,regSeqInIdeal,Regularity,regularity,relations,RelativeCanonicalResolution,relativizeFilename,Reload,remainder,RemakeAllDocumentation,remove,removeDirectory,removeFile,removeLowestDimension,reorganize,replace,RerunExamples,res,reshape,ResidualIntersections,ResLengthThree,Resolution,resolution,ResolutionsOfStanleyReisnerRings,restart,Result,resultant,Resultants,return,returnCode,Reverse,reverse,RevLex,Right,right,Ring,ring,RingElement,RingFamily,ringFromFractions,RingMap,rootPath,roots,rootURI,rotate,round,rowAdd,RowExpression,rowMult,rowPermute,rowRankProfile,rowSwap,RR,RRi,rsort,run,RunDirectory,RunExamples,RunExternalM2,runHooks,runLengthEncode,runProgram,same,saturate,Saturation,scan,scanKeys,scanLines,scanPairs,scanValues,schedule,schreyerOrder,Schubert,Schubert2,SchurComplexes,SchurFunctors,SchurRings,SCRIPT,scriptCommandLine,ScriptedFunctor,SCSCP,searchPath,sec,sech,SectionRing,SeeAlso,seeParsing,SegreClasses,select,selectInSubring,selectVariables,SelfInitializingType,SemidefiniteProgramming,Seminormalization,separate,SeparateExec,separateRegexp,Sequence,sequence,Serialization,serialNumber,Set,set,setEcho,setGroupID,setIOExclusive,setIOSynchronized,setIOUnSynchronized,setRandomSeed,setup,setupEmacs,sheaf,SheafExpression,sheafExt,sheafHom,SheafOfRings,shield,ShimoyamaYokoyama,short,show,showClassStructure,showHtml,showStructure,showTex,showUserStructure,SimpleDoc,simpleDocFrob,SimplicialComplexes,SimplicialDecomposability,SimplicialPosets,SimplifyFractions,sin,singularLocus,sinh,size,size2,SizeLimit,SkewCommutative,SlackIdeals,sleep,SLnEquivariantMatrices,SLPexpressions,SMALL,smithNormalForm,solve,someTerms,Sort,sort,sortColumns,SortStrategy,source,SourceCode,SourceRing,SPACE,SpaceCurves,SPAN,span,SparseMonomialVectorExpression,SparseResultants,SparseVectorExpression,Spec,SpechtModule,SpecialFanoFourfolds,specialFiber,specialFiberIdeal,SpectralSequences,splice,splitWWW,sqrt,SRdeformations,stack,stacksProject,Standard,standardForm,standardPairs,StartWithOneMinor,stashValue,StatePolytope,StatGraphs,status,stderr,stdio,step,StopBeforeComputation,stopIfError,StopWithMinimalGenerators,Strategy,String,STRONG,StronglyStableIdeals,STYLE,Style,style,SUB,sub,SubalgebraBases,sublists,submatrix,submatrixByDegrees,Subnodes,subquotient,SubringLimit,Subscript,subscript,SUBSECTION,subsets,substitute,substring,subtable,Sugarless,Sum,sum,SumOfTwists,SumsOfSquares,SUP,super,SuperLinearAlgebra,Superscript,superscript,support,SVD,SVDComplexes,switch,SwitchingFields,sylvesterMatrix,Symbol,symbol,SymbolBody,symbolBody,SymbolicPowers,symlinkDirectory,symlinkFile,symmetricAlgebra,symmetricAlgebraIdeal,symmetricKernel,SymmetricPolynomials,symmetricPower,synonym,SYNOPSIS,syz,Syzygies,SyzygyLimit,SyzygyMatrix,SyzygyRows,syzygyScheme,TABLE,Table,table,take,Tally,tally,tan,TangentCone,tangentCone,tangentSheaf,tanh,target,Task,taskResult,TateOnProducts,TD,temporaryFileName,tensor,tensorAssociativity,TensorComplexes,terminalParser,terms,TEST,Test,testExample,testHunekeQuestion,TestIdeals,TestInput,tests,TEX,tex,TeXmacs,texMath,Text,TH,then,Thing,ThinSincereQuivers,ThreadedGB,threadVariable,Threshold,throw,Time,time,times,timing,TITLE,TO,to,TO2,toAbsolutePath,toCC,toDividedPowers,toDual,toExternalString,toField,TOH,toList,toLower,top,top,topCoefficients,Topcom,topComponents,topLevelMode,Tor,TorAlgebra,Toric,ToricInvariants,ToricTopology,ToricVectorBundles,toRR,toRRi,toSequence,toString,TotalPairs,toUpper,TR,trace,transpose,TriangularSets,Tries,Trim,trim,Triplets,Tropical,true,Truncate,truncate,truncateOutput,Truncations,try,TSpreadIdeals,TT,tutorial,Type,TypicalValue,typicalValues,UL,ultimate,unbag,uncurry,Undo,undocumented,uniform,uninstallAllPackages,uninstallPackage,Unique,unique,Units,Unmixed,unsequence,unstack,Up,UpdateOnly,UpperTriangular,URL,urlEncode,Usage,use,UseCachedExampleOutput,UseHilbertFunction,UserMode,userSymbols,UseSyzygies,utf8,utf8check,validate,value,values,Variable,VariableBaseName,Variables,Variety,variety,vars,Vasconcelos,Vector,vector,VectorExpression,VectorFields,VectorGraphics,Verbose,Verbosity,Verify,VersalDeformations,versalEmbedding,Version,version,VerticalList,VerticalSpace,viewHelp,VirtualResolutions,VirtualTally,VisibleList,Visualize,wait,WebApp,wedgeProduct,weightRange,Weights,WeylAlgebra,WeylGroups,when,whichGm,while,width,wikipedia,Wrap,wrap,WrapperType,XML,xor,youngest,zero,ZeroExpression,zeta,ZZ,ZZParser,
makeGWClass,getDiagonalClass,makeDiagonalForm,getSignature,isAnisotropic,isIsotropic,getAnisotropicPart,getSumDecompositionString,getGlobalA1Degree,getLocalA1Degree,isIsomorphicForm,addGW,getSumDecomposition}
}
\lstalias{Macaulay2output}{Macaulay2}

\newcommand{\Bl}{Bl}

\newcommand{\Mn}{\overline{M}_{0,n}}
\newcommand{\Fplus}{{\;+_F\;}}
\newcommand{\Fminus}{{\; -_F\;}}
\newcommand{\SH}{\mathsf{SH}}
\newcommand{\Sm}{\mathsf{Sm}}
\newcommand{\Sp}{\mathsf{Sp}}
\newcommand{\MS}{\mathsf{MS}}
\newcommand{\Fsum}{%
  \mathop{\ensurestackMath{\stackinset{c}{3pt}{c}{0pt}{\scriptscriptstyle F}{\sum}}}
}
\newcommand{\ABl}{A^\bullet(Bl_ZX)}
\newcommand{\Kont}{\overline{{M}}_{0,0}(\PP^2, 2)}
\newcommand{\pointring}{A^\bullet(\operatorname{pt})}
\newcommand{\symtangent}{\operatorname{Sym}^2(\T_{\mathbb P^2})}

\DeclareMathOperator{\pt}{pt}

\title{Oriented Cohomology Rings of Some Moduli Spaces via Blowups}
\author{Arkamouli Debnath \and Michael Ruofan Zeng }
\date{\today}

\begin{document}

\maketitle

\begin{abstract}
    Oriented cohomology theories provide a general framework to perform intersection-theory-type calculus. The Chow ring, algebraic $K$-theory, and Levine--Morel's algebraic cobordism are all instances of such theories satisfying $\AA^1$-invariance. Topological Hochschild homology, topological cyclic homology, and Hodge cohomology are important examples of theories without $\AA^1$-invariance. In this paper, we prove an additive blowup formula for oriented cohomology theories in the non-$\AA^1$-invariant category of motivic spectra, developed by Annala, Hoyois, and Iwasa. Then, we specialize to $\AA^1$-invariant theories and give presentations of oriented cohomology rings of the blowup of a smooth scheme along a smooth center. We compute explicit examples of such presentations for the cases of del Pezzo surfaces, the blowup of $\PP^3$ along the twisted cubic, and the blowup of $\PP^5$ along the Veronese surface, which can be identified with the moduli space of complete conics. We demonstrate that one can recover solutions to classical enumerative geometry problems, such as Steiner's $3264$ conics, using arbitrary oriented cohomology theories. Finally, we give a presentation of oriented cohomology rings of $\Mn$, which generalizes Keel's presentation of the Chow ring. 
\end{abstract}

\tableofcontents

\section{Introduction}

A central theme in algebraic geometry is to use algebraic invariants to distinguish schemes. Some examples of such invariants include the Chow ring $CH^\bullet$ and the Grothendieck ring of vector bundles $K_0$. In \cite{LevineMorel}, Levine--Morel provided an axiomatic framework for \textit{oriented cohomology theory}, which unifies both the above examples of invariants. In particular, they constructed the geometric theory of algebraic cobordism $\Omega^\bullet$ and showed that it is the \textit{universal} oriented cohomology theory. 

The axioms of oriented cohomology theory provide the \textit{``intersection theory package''}, which includes functorial pullback and pushforward, projection formula, projective bundle formula, extended homotopy property, a theory of Chern classes governed by a formal group law, localization sequence, fundamental classes, and the excess intersection formula. Hence, we can extend many familiar intersection theory calculations (for example, the ones in Fulton \cite{Fulton} and Eisenbud--Harris \cite{EH3264}) to a general oriented cohomology theory. 

An important class of schemes to which we want to extend these calculations is moduli spaces. A \textit{moduli problem} asks one to parametrize a \textit{family} of objects with some defining properties. Under suitable conditions, solutions to many such moduli problems are parametrized by schemes called \textit{fine/coarse moduli spaces} (see \cite{Newstead2013IntroductionTM}, \cite{Alper2026StacksModuli}). Calculating intersection numbers on moduli spaces has been a central theme in enumerative geometry \cite{Kontsevich1995Enumeration,KontsevichManin1994GW}. Subsequent research by e.g. Lee and Coates--Givental extended these calculations to $K$-theory and cobordism-valued Gromov--Witten theory with rational coefficients \cite{Lee2004QuantumKTheory,CoatesGivental2006QuantumCobordisms}. Very recently, Ellis--Bloor extended $\psi$-class intersections on $\Mn$ to algebraic cobordism with integral coefficients \cite{EllisBloor2026CobordismIntersectionTheory}. Intersection numbers in general oriented cohomology theories often contain refined geometric information that is not witnessed by the Chow ring. 

On the other hand, many of these moduli spaces, such as the moduli space of stable maps $\overline{M}_{0,0}(\mathbb P^2,2)$ and the moduli space of stable rational $n$-pointed curves $\Mn$,  have descriptions as blowups of smooth schemes at smooth centers. In order to understand the oriented cohomology of such spaces with \textit{integer} coefficients, we develop a \textit{Blowup Presentation}, which refers to a presentation of the oriented cohomology ring of a blowup in terms of the cohomology rings of the base scheme and the blowup center. Such presentations of the integral Chow ring are well-known, and Fulton--Lang developed the blowup formalism for the $K_0$-ring \cite{FultonLang1985RiemannRochAlgebra}.

The goal of this paper is threefold.  First,  we provide a presentation in terms of generators and relations for the oriented cohomology ring of blowups of smooth schemes at smooth centers and apply this presentation to specific examples.
Secondly, we demonstrate that we can recover classical enumerative geometry results, using our Blowup Presentation for algebraic cobordism. Such examples include counting the number of conics tangent to $5$ general conics in $\mathbb P^2$ and counting the number of secant lines to the twisted cubic in $\mathbb P^3$ meeting two general lines. 
Thirdly,  we use the iterated blowup description of $\Mn$ in \cite{Keel}, to give a presentation of its oriented cohomology rings, which generalizes Keel's presentation of the Chow ring. We highlight the main results of our paper below.

\paragraph{Blowup Formulas.}
 Let $X$ be a smooth $k$-scheme and let $i\colon Z\hookrightarrow X$ be a regular closed immersion of codimension $r\ge 2$.
Let $\pi\colon \Bl_Z X\to X$ be the blowup of $X$ along $Z$. Let $j\colon E\hookrightarrow \Bl_Z X$ be the exceptional divisor with projection $p:E\to Z$. The \textit{blowup square} 
\[
\begin{tikzcd}
E \ar[r,"j"] \ar[d,"p"'] & \Bl_Z X \ar[d,"\pi"] \\
Z \ar[r,"i"'] & X.
\end{tikzcd}
\]
is a cartesian diagram in the category of smooth $k$-schemes $\Sm_k$.
 Our first contribution is an additive decomposition for the oriented cohomology of a blowup. While this decomposition is well-known for $\AA^1$-invariant cohomology theories, algebraic $K$-theory over non-regular schemes does not satisfy $\AA^1$-invariance. Topological Hochschild homology, topological cyclic homology, and Hodge cohomology are examples of oriented theories that fail to be $\AA^1$-invariant, and this failure already occurs over regular schemes in characteristic zero \cite{Elmanto2021THHandTC}. To prove the additive decomposition in its full generality, we work in the non-$\AA^1$-invariant category $\mathsf{MS}_S$ of motivic spectra over an arbitrary base scheme $S$, constructed by Annala--Hoyois--Iwasa \cite{AnnalaIwasa2022MotivicSpectra, AHI2025CobordismConnerFloyd, AHI2024AtiyahDuality}. In this context, we prove the following blowup formula, which recovers, for example, the blowup formula for Hodge cohomology \cite[Theorem 1.2]{RaoEtAl2023HodgeCohomology}.
\begin{theorem}[Blowup Formula for motivic spectra in $\mathsf{MS}_S$, \Cref{thm:blowup-formula-additive-bigraded}]
Let $\mathbb E$ be an oriented motivic ring spectrum in $\MS_S$.
For every pair of integers $p,q\in \ZZ$, there is a direct sum decomposition
\[
\mathbb E^{p,q}(\Bl_Z X)\cong \mathbb E^{p,q}(X)\oplus \bigoplus_{k=1}^{r-1}\mathbb E^{p-2k,q-k}(Z). \footnote{\text{Examples of these higher $\EE^{p,q}$ groups include higher Chow groups and higher $K$-groups. }}
\]
\end{theorem}

We turn to give a presentation of the $A^\bullet(\pt)$-algebra $A^\bullet(Bl_ZX)$ in terms of generators and relations, for theories represented in the $\AA^1$-invariant stable motivic category. The absence of certain formal properties for general non-$\AA^1$-invariant theories from the current literature prompted us to make such a switch. Throughout the rest of this introduction, let $k$ be an algebraically closed field of characteristic zero. Specializing the Blowup Formula to $\AA^1$-invariant oriented theories, we recover the well-known result that $A^\bullet(Bl_ZX)$ is generated by pullbacks of classes from $X$ and pushforwards of classes from $E \cong \PP(\N_{Z/X})$. We work out four kinds of relations among these generators via functorial properties of $A^\bullet$ and show that these relations are exhaustive. We refer to this particular presentation of $A^\bullet(Bl_ZX)$ as the \textit{Blowup Presentation}. 

\begin{theorem}[Blowup Presentation, \Cref{thm:blowup-presentation}]
Let $A^\bullet$ be an oriented cohomology theory on $\mathsf{Sm}_k$ satisfying the formal group law $F_A$. Let $\zeta=c_1^A(\mathcal{O}_E(1))\in A^1(E)$. The oriented cohomology ring $A^\bullet(\Bl_Z X)$ is generated over $A^\bullet(\pt)$ by pullbacks $\pi^*\alpha$ for $\alpha\in A^\bullet(X)$ and pushforwards $j_*\gamma$ for $\gamma\in A^\bullet(E)$, subject to the following families of relations.

\begin{enumerate}[(i)]
    \item (\textit{Intersection on $X$}) For $\alpha, \beta\in A^\bullet(X)$, $\pi^*\alpha \cdot \pi^*\beta = \pi^*(\alpha\cdot \beta) $.
    \item (\textit{Intersection between $X $ and $E$}) For $\alpha \in A^\bullet(X)$ and $\gamma\in A^\bullet(E)$, $\pi^*\alpha \cdot j_*\gamma = j_*(\gamma \cdot p^*i^*\alpha)$.
    \item (\textit{Intersection on $E$}) For $\gamma, \delta\in A^\bullet(E)$, $j_*\gamma \cdot j_*\delta = j_*(\gamma\cdot \delta\cdot \chi_A(\zeta))$, where $\chi_A(\zeta)$ is the formal inverse of $\zeta$ under the group law $F_A$. 
    \item (\textit{Excess intersection formula}) For every $\theta \in A^\bullet(Z)$, we have
$$
\pi^* i_*(\theta)=j_*\left(c_{r-1}^A(\mathcal{Q}) \cdot p^* \theta\right), 
$$ where $\mathcal{Q}:= p^* \mathcal{N}_{Z / X} / \mathcal{O}_E(-1).$
\end{enumerate}

\end{theorem}

Given the additional assumption that the pullback from $A^\bullet(X)$ to $A^\bullet(Z)$ is surjective, the Blowup Presentation collapses to a rather simple form that generalizes Keel's formula for the Chow ring \cite[Appendix Theorem A]{Keel}. We work out full applications of both presentations in \Cref{sec:first-examples} and \Cref{sec:oriented-cohomology-of-complete-conics}.

\begin{theorem}[Surjective Blowup Presentation, \Cref{thm:blowup-presentation-surjective}]
If the pullback $i^*\colon A^\bullet(X)\to A^\bullet(Z)$ is surjective, $A^\bullet(\Bl_Z X)$ is generated over $A^\bullet(X)$ by the exceptional class $\zeta=c_1^A(\mathcal{O}_{\Bl_Z X}(E))$. It admits the quotient presentation:
\[
A^\bullet(Bl_ZX)\;\cong\;
\frac{A^\bullet(X)[\zeta]}{\Bigl(\zeta^r-\widetilde c_1\zeta^{r-1}+\cdots+ (-1)^{r-1}\widetilde c_{r-1}\zeta + (-1)^ri_*1,\ \zeta\cdot\ker(i^*)\Bigr)}.
\]
where the elements $\widetilde c_k \in A^k(X)$ are chosen lifts of the Chern classes $ c_k^A(\mathcal{N}_{Z/X})$.
\end{theorem}

\paragraph{Enumerative Geometry via Cobordism.}
Equipped with the above formulae, we now demonstrate that we can recover well-known results in enumerative geometry via intersection theory in arbitrary oriented cohomology rings---in particular, algebraic cobordism. First, we consider the blowup $Bl_C\PP^3$ of $\mathbb{P}^3$ along the twisted cubic $C$, which is the universal secant line to the twisted cubic in $\PP^3$. 
\begin{theorem}[Blowup Presentation of $\Bl_C{\mathbb P^3}$, \cref{thm:oriented-coh-of-BLCP3}]
Let $A^\bullet$ be any oriented cohomology theory on $\Sm_k$ with formal group law $F_A(x,y)=x+y+a_{11}xy + \text{ higher order terms}$. Let $\alpha$ be the pullback of the hyperplane class $c_1^A(\O_{\PP^3}(1))$ from $\mathbb{P}^3$. Define the pushforwards $e:=j_*(1), x:=j_*(\eta), z:=j_*(\zeta)$, where $\eta = c_1^A(p^*\O_C(1))$ and $\zeta = c_1^A(\O_E(1))\in A^1(E)$. As an $A^\bullet(\pt)$-algebra, the oriented cohomology ring $A^\bullet(\Bl_C\mathbb{P}^3)$ has the presentation \[
A^\bullet(\Bl_{C}\mathbb P^3)\ \cong\
\frac{A^\bullet(\pt)[\alpha,e,x,z]}{
\begin{aligned}
&(\alpha^4,\ \alpha^2e,\ \alpha e-3x,\ \alpha x,\ \alpha z-3\alpha^3,\ z^2,\ x^2, \ xz,\\
&\ e^2+z+10a_{11}\alpha^3,\ ex+\alpha^3,\ ez-10\alpha^3,\ 3\alpha^2-z-10x-8a_{11}\alpha^3).
\end{aligned}}
\]
\end{theorem}
The product $([4]_F \alpha \Fminus [2]_F e)^2 \cdot \alpha = 4\alpha^3 = 4[\pt]_A$ \footnote{The notation $[\pt]_A$ means the fundamental class of a point in $A$-theory. For the notation $-_F$, see \Cref{subsec:formal-group-laws}.} in this ring recovers the classical enumerative result (\Cref{thm:4-secant-lines}), which states there are exactly $4$ secant lines to the twisted cubic meeting two general lines in $\mathbb{P}^3$.

Next, we apply our formalism to the moduli space of complete conics, $\Bl_V\mathbb P^5$, which is famously used in the resolution of Steiner's conic problem. Here $V$ is the Veronese surface inside $\mathbb P^5$, and it is known that the moduli space of complete conics is isomorphic to the moduli space of stable maps $\Kont$. To state the ring presentation, let $A^\bullet$ be any oriented cohomology theory over $\Sm_k$, and let $a_{ij}\in A^{1-i-j}(\pt)$ be the coefficient in front of the monomial $x^iy^j$ in $F_A$. Let $\alpha := \pi^*c_1^A(\O_{\PP^5}(1))$. Define 9 exceptional classes $e_{a,b}:=j_*(\eta^a\zeta^b)$ for $0 \leq a,b \leq 2$, where $\eta := p^*c_1^A(\O_V(1))$ and $\zeta:= c_1^A(\O_{E}(1))\in A^1(E)$.

\begin{theorem}[Blowup Presentation of $Bl_V\PP^5$, \Cref{thm:oriented-coh-of-BlVP5}]
The oriented cohomology ring $A^\bullet(Bl_V\PP^5)$ is generated over $A^\bullet(\pt)$ by the classes $\alpha$ and $e_{a,b}$, modulo the following families of relations.
\begin{enumerate}[(i)]
    \item $\alpha^6 = 0$.
    \item $\alpha \cdot e_{a,b} = 2e_{a+1,b} + a_{11}e_{a+2,b}$.
    \item $e_{a,b} \cdot e_{c,d} = -e_{a+c, b+d+1} + a_{11}e_{a+c, b+d+2} + 30a_{11}^2e_{a+c+2, b+d+1} + 9a_{11}^2e_{a+c+1, b+d+2} + s e_{a+c+2, b+d+2}$, where $s := 57a_{11}^3 + 51a_{11}a_{12} - 51a_{22} + 102a_{13}$.
    \item For $k = 0,1,2$, $\pi^*i_*(\eta^k) = 30e_{k+2,0} + 9e_{k+1,1} + e_{k,2} + 66a_{11}e_{k+2,1} + 9a_{11}e_{k+1,2} + 78a_{11}^2e_{k+2,2}$, where $\pi^*i_*1 = 4\alpha^3+3a_{11}\alpha^4+3a_{21}\alpha^5$, $\pi^*i_*(\eta) = 2\alpha^4+a_{11}\alpha^5$, and $\pi^*i_*(\eta^2) = \alpha^5$,
\end{enumerate}
These are subject to the reductions $e_{a,3} = -9e_{a+1,2}-6a_{11}e_{a+2,2}-30e_{a+2,1}$ and $e_{a,4}=51e_{a+2,2}$
\end{theorem}

This ring presentation enables us to lift the solution to Steiner's conic problem from Chow ring to algebraic cobordism. Computing the intersection of the pullbacks of five general sextic hypersurfaces in $\Omega^\bullet(Bl_V\PP^5)$ amounts to evaluating the class $([6]_\Omega \alpha \Fminus [2]_\Omega e_{0,0})^5$. Remarkably, we observe that all terms carrying formal group law coefficients $a_{ij}$ vanish modulo the above relations, yielding exactly $3264[\pt]_\Omega$. We demonstrate through this example that one can recover classical enumerative results from cobordism-valued intersection theory.

\paragraph{Oriented Cohomology of $\Mn$.}
Finally, we turn to the moduli space of stable rational $n$-pointed curves $\Mn$. Following Keel's iterated blowup description of $\Mn$, we give a presentation of its oriented cohomology ring. This presentation has a similar structure to Keel's presentation of the Chow ring $CH^\bullet(\Mn)$ \cite[Theorem 1]{Keel}.
In \cref{def:ring-R0n}, we define the ring $R_{0,n}^A$ as the quotient of the free $A^\bullet(\pt)$ algebra $S_{0,n}^A := A^\bullet(\pt)[x_S\mid 2\le |S| \le n-2]$, modulo the ideal generated by the following four kinds of relations.\footnote{These four relations do not correspond to the four families of relations in the Blowup Presentation (\Cref{thm:blowup-presentation})}
\begin{enumerate}[(1)]
    \item (\textit{Complements}) $x_S = x_{S^c}$.
    \item (\textit{Fundamental relations}) For all distinct $i,j,k,l \in [n]$, 
    \[\Fsum_{\substack{S\subset [n] \\ i, j \in S, k, \ell \notin S}} x_S=\Fsum_{\substack{S\subset [n] \\ i, k \in S, j, \ell \notin S}} x_S=\Fsum_{\substack{S\subset [n] \\ i, \ell \in S, j, k \notin S}} x_S. \footnote{\text{The notation $\Fsum$ means summation under the formal group law $F$. See \Cref{subsec:formal-group-laws} for the precise definition.}}\]
    \item (\textit{Incompatibility}) If $S$ and $T$ are geometrically incompatible, $x_Sx_T = 0$.
    \item (\textit{Nilpotency}) $x_S^{n-2}$, for all $S$ with $2\le |S|\le n-2$
\end{enumerate}

\begin{theorem}[Oriented Cohomology of $\Mn$, \Cref{thm:oriented-cohomology-of-M0nbar}]
Let $A^\bullet$ be an oriented cohomology theory satisfying the dimension axiom. Then, there is an isomorphism of graded rings $R_{0,n}^A \cong A^\bullet(\Mn)$
\end{theorem}

We conclude the paper by comparing these resulting ring structures. By specializing $A^\bullet$ to $K$-theory, our theorem explicitly recovers a presentation of $K_0(\overline{M}_{0,n})$ developed by Larson, Li, Payne, and Proudfoot \cite{LLPP2024WonderfulVarieties}. While they showed that $K_0(\Mn)$ and $CH^\bullet(\Mn)$ are abstractly isomorphic as rings, we prove that the naive bijection of boundary divisor generators fails to be a ring homomorphism for $n \ge 6$ (\Cref{lem:naive-map-fails}). However, for $n=5$, we note that there are generalized \textit{exceptional isomorphisms} in the form $A^\bullet(\overline{M}_{0,5}) \cong CH^\bullet(\overline{M}_{0,5}) \otimes_\ZZ A^\bullet(\pt)$, in the sense of \cite{LLPP2024WonderfulVarieties}. 

\paragraph{Outline of the paper.} 
In \Cref{subsec:blowup-formula-additive}, we review oriented cohomology theories, formal group laws, oriented motivic spectra, fundamental classes, and prove the additive blowup decomposition. \Cref{sec:multiplicative-blowup-formula} establishes the ring-level Blowup Presentation. In \Cref{sec:first-examples}, we apply this formalism to compute the oriented cohomology rings of del Pezzo surfaces and the blowup of $\mathbb{P}^3$ along the twisted cubic. In \Cref{sec:oriented-cohomology-of-complete-conics}, we give an explicit presentation of the oriented cohomology of the moduli space of complete conics and recover the count of $3264$ conics in algebraic cobordism. In \Cref{sec:oriented-cohomology-of-M0nbar}, we establish a Keel-style presentation for oriented cohomology rings of $\Mn$. We compare the resulting ring structures and note that there exist generalized exceptional isomorphisms for $\overline{M}_{0,5}$.

\paragraph{Acknowledgements.}

The authors would like to thank Toni Annala for invaluable comments and for suggesting various arguments used in the proofs of \Cref{subsec:blowup-square-and-additive}.  The authors would like to thank Shiyue Li for making them aware of a very recent work by Gandhi and Partida \cite{GandhiPartida2026Wonderful}, which is relevant to the current paper. The authors would like to thank Jarod Alper,  Sara Billey, Thomas Brazelton, Ting Gong, and Jackson Morris for many helpful comments.

\section{Background on Oriented Cohomology Theories} \label{subsec:blowup-formula-additive}

Cohomology theories with a \emph{complex orientation} have long played an important role in homotopy theory. Starting from the geometric constructions of Conner--Floyd, the subject was developed further by Quillen's identification of the formal group law of complex cobordism with the universal formal group law, by Landweber's exact functor theorem, and by the role of complex cobordism in the Adams--Novikov spectral sequence. Together, these results established complex cobordism as a fundamental object of study \cite{ConnerFloyd1966,Quillen1969ComplexCobordism,Quillen1971Elementary,Landweber1976,Adams1974}.

It is natural to ask for an algebro-geometric analogue of complex-oriented cohomology theories. Panin \cite{Panin2003Oriented} and then Levine--Morel \cite{LevineMorel} developed the foundations for cohomology theories with an \emph{algebraic orientation} on the category of smooth schemes. At the same time, the $\mathbb A^1$-homotopy theory of Morel--Voevodsky and the stable motivic homotopy category $\mathsf{SH}(k)$ provided a homotopical setting for studying oriented cohomology theories through their representing objects called oriented motivic ring spectra. Panin--Pimenov--R\"ondigs showed that Voevodsky's spectrum $MGL$ is again the universal object among oriented motivic spectra, and further extensions such as motivic Landweber exact theories are due to Naumann--Spitzweck--{\O}stv{\ae}r \cite{MorelVoevodsky1999, PaninPimenovRoendigs2008, NaumannSpitzweckOstvaer2009}. Levine--Morel \cite{LevineMorel} developed the corresponding geometric cohomology and Borel--Moore homology theory called algebraic cobordism. Déglise's work on Gysin morphisms and orientations, together with the six-functor formalism of Ayoub and Cisinski--Déglise, furnished $\mathsf{SH}(k)$ with functoriality and purity \cite{DegliseGysinI, DegliseGysinII, Ayoub2007I, Ayoub2007II, CisinskiDeglise2019}. Déglise--Jin--Khan  \cite{DegliseJinKhanFundamentalClasses} developed intersection-theoretic tools on $\mathsf{SH}(k)$, endowing oriented cohomology theories of schemes with fundamental classes and refined Gysin pullbacks. Thus, oriented cohomology theories provide a very general framework in which one may perform intersection theory-style calculus. In this section, we summarize axioms and facts for oriented cohomology theory of schemes constituting what we refer to as the ``intersection theory package,'' following \cite{LevineMorel, DegliseJinKhanFundamentalClasses}.

\subsection{Oriented Cohomology Theories}

Fix a field $k$ and let $\mathsf{Sm}_k$ denote the category of smooth quasi-projective $k$-schemes. Let $\mathsf{GrCommRing}$ denote the category of commutative graded rings with unit. An \emph{additive functor} from $\mathsf{Sm}_k$ to $\mathsf{GrCommRing}$ is a functor $A$ such that $A(X\amalg Y)\cong A(X)\oplus A(Y)$.

\begin{definition}[Oriented cohomology theory on $\mathsf{Sm}_k$]
 An \emph{oriented cohomology theory} on $\Sm_k$ is a pair $(A^\bullet, f\mapsto f_*)$ consisting of:
\begin{enumerate}
\item an additive contravariant functor
\[
A^\bullet:\Sm_k^{\mathrm{op}}\to \mathsf{GrCommRing},\qquad \bigl(f:X\to Y\bigr) \mapsto \bigl(f^*:A^\bullet(Y)\to A^\bullet(X)\bigr), \quad \text{ and }
\]

\item for each \emph{projective} morphism $f:Y\to X$ in $\Sm_k$ of relative dimension $d$, a
graded $A^\bullet(X)$-module homomorphism
\[
f_*:A^i(Y)\to A^{i-d}(X),
\]
where the $A^\bullet(X)$-module structure on $A^\bullet(Y)$ is induced by $f^*$.
\end{enumerate}

\end{definition}

The oriented cohomology theory ought to satisfy the standard axioms (A1), (A2), (PBF), (EH), which we adopt from \cite[Def.~1.1.2]{LevineMorel}. It follows from (PBF) that the theory admits Chern classes governed by a $1$-dimensional commutative formal group law. In addition, the theory ought to entertain properties (PF), (FC), (RG), which we detail below. 
\begin{enumerate}
\item[(A1)] \textbf{Functoriality of push-forward}. We have $(\mathrm{id}_X)_*=\mathrm{id}$. If $f$ and $g$ are composable projective morphisms, then $(f\circ g)_*=f_*\circ g_*$.
\item[(A2)] \textbf{Transverse base change.} For a Cartesian square
\[
\begin{tikzcd}
W \ar[r,"g'"] \ar[d,"f'"'] & Y \ar[d,"f"]\\
X \ar[r,"g"'] & Z
\end{tikzcd}
\]
with $f$ projective and $(f,g)$ transverse, \footnote{The morphisms $f$ and $g$ are \textit{transverse} if $\operatorname{Tor}_i^{\O_Z}(g_*\O_X,\; f_*\O_Y)$ vanishes for all $i>0$. }
one has $g^*\circ f_* = f'_*\circ g'^*$.
\item[(PBF)] \textbf{Projective Bundle Formula} -- additive version. For a rank $r$ vector bundle $\E\to X$ with associated projective bundle
$\pi:\mathbb P(E)\to X$, there exists an element $\zeta\in A^1(\PP(\E))$ called the \emph{relative hyperplane class}, such that $A^\bullet(\mathbb P(E))$ is the free $A^\bullet(X)$-module given by
\begin{equation}\label{eq:projective-bundle-formula-additive}
    A^\bullet(\PP(\E)) \cong \bigoplus_{i = 0}^{r-1} \zeta^i\cdot A^\bullet(X).
\end{equation}
\item[(EH)] \textbf{Extended homotopy.} for any $E$-torsor $p:V\to X$, the pullback
$p^*:A^\bullet(X)\to A^\bullet(V)$ is an isomorphism. In particular, $A^\bullet(X)\xrightarrow{\sim} A^\bullet(E)$ for a vector
bundle $E\to X$.
\end{enumerate}

Since the proper pushforward $f_*$ is a $A^\bullet(X)$-module homomorphism, the above axioms imply 

\begin{itemize}
    \item[(PF)] \textbf{Projection formula.} For a projective morphism $f:Y\to X$ and $\alpha\in A^\bullet(X)$, $\beta\in A^\bullet(Y)$, one has
\[
f_*(f^*\alpha\cdot \beta)=\alpha\cdot f_*(\beta).
\]
\end{itemize}

Next, we briefly recall how (PBF) endows any oriented cohomology theory with a theory of \emph{Chern classes}. The \emph{additive} projective bundle formula \eqref{eq:projective-bundle-formula-additive} states that $A^\bullet(\PP(\E))$ is a finite free $A^\bullet(X)$-module generated as an $A^\bullet(X)$-algebra by a single element $\zeta$, so there exists a unique monic polynomial $f(\zeta) = \zeta^r + a_1\zeta^{r-1} + a_2\zeta^{r-2} \cdots +a_r$ which is homogeneous of degree $r$, such that $A^\bullet(\PP(\E)) \cong A^\bullet(X)[\zeta] / (f)$. 

\begin{definition}[Chern classes in $A^\bullet$-theory]
 Let $\E$ be a vector bundle of rank $r$ over $X$. For $1\le i\le r$, the $A^\bullet$-theoretic \emph{$i^{\text{th}}$ Chern class} of $\E$ is defined to be the unique coefficient $c_i^A(\E):=a_i \in A^i(X)$. The $A^\bullet$-theoretic \emph{total Chern class} \footnote{The \textit{total Chern polynomial} is the formal sum $c_t^A(\E) := 1 + c_1^A(\E)t + c_2^A(\E)t^2 + \cdots + c_r^A(\E)t^r$ with the indexing variable $t$. Sometimes this is the preferred convention to avoid addition across different degrees.} of $\E$ is the formal sum of all Chern classes \[c^A(\E):= 1+ c_1^A(\E)+ c_2^A(\E) + \cdots + c_r^A(\E).\]
\end{definition}

\begin{theorem}[Projective Bundle Formula -- ring version]\label{thm:projective-bundle-formula}
Let $\E$ be a vector bundle of rank $r$ over $X$. The oriented cohomology ring of the projective bundle $\PP(\E)$ has the quotient presentation

\begin{equation}
    \label{eq:projective-bundle-formula-ring}
    A^\bullet(\PP(\E)) \cong \frac{A^\bullet(X)[\zeta]}{\zeta^r + \zeta^{r-1}c_1^A(\E) + \zeta^{r-2}c_2^A(\E) + \cdots + c_r^A(\E)},
\end{equation}
where $\zeta:=c_1^A(\cO_{\PP(\E)}(1))\in A^1(\PP(\E))$ is the $A^\bullet$-theoretic hyperplane class of $\PP(\E)$.\footnote{We are using the sign convention in the monic polynomial relation that is opposite to Levine--Morel.}
\end{theorem}

The $A^\bullet$-theoretic Chern classes satisfy the following axioms/properties, which we summarize from \cite[\S 1.1]{LevineMorel}.

\begin{itemize}
    \item \textbf{Naturality.} If $f\colon Y\to X$ is a morphism in $\Sm_k$ and $\E$ is a vector bundle on $X$, then
\[
c_i^A(f^\ast \E)=f^\ast c_i^A(\E)
\qquad\text{for all } i\geq 0.
\]
    \item \textbf{Normalization.} If $\E$ is a vector bundle of rank $r$ over $X$, then $c_i^A(\E) = 0$ for all $i > r$. Additionally, $c^A(\cO_X) = 1$. 
    \item \textbf{Whitney Sum Formula.} If $0\to \E \to \F \to \G\to 0$ is a short exact sequence of vector bundles, then \[c(\F) = c(\E)c(\G). \footnote{\text{Equivalently, $c_t^A(\F) = c_t^A(\E)c_t^A(\G)$, where $c_t^A$ is the total Chern polynomial.}}\]
    \item \textbf{Formal group law for tensor product of line bundles} \cite[Remark 4.1.7 and Corollary 4.1.8]{LevineMorel}. There exists a formal power series \[F_A(x,y) = x+ y + \sum_{i,j\ge 1} a_{ij}x^iy^j, \quad |a_{ij}| = 1 - i - j \footnote{\text{Note that the degree $1-i-j$ is to make the whole expression $F_A(x,y)$ remain in $A^1(X)$.}},\] called the \emph{formal group law} of $A^\bullet$, such that the tensor product of any two line bundles $\L$ and $ \M$ has first Chern class \[c_1^A(\L\otimes \M) = F_A(c_1^A(\L),c_1^A(\M)).\] We discuss these formal group laws in detail in the ensuing \Cref{subsec:formal-group-laws}. 
\end{itemize}

For the purposes of this paper, we require the oriented cohomology theories to satisfy two additional properties that are absent from Levine--Morel's definition, which we term the ``Gysin package.'' Fundamental classes of smooth l.c.i. subvarieties are useful for finding a convenient set of generators in the quotient ring presentation of the oriented cohomology ring of a blowup. Refined Gysin pullbacks are essential for studying the $A^\bullet$-product in the exceptional locus of a blowup. Note that Levine--Morel established this Gysin package for algebraic cobordism in \cite[\S6.5-6.6]{LevineMorel}. In \Cref{subsec:oriented-theories-and-oriented-spectra}, we discuss how the results of Déglise--Jin--Khan \cite{DegliseJinKhanFundamentalClasses} guarantee these additional properties for all oriented cohomology theories that we consider. 

\begin{itemize}
    \item[(FC)] \textbf{Fundamental classes and Gysin morphisms for smoothable lci maps.}
For every smoothable local complete intersection morphism $f\colon X\to Y$ between smooth $k$-schemes of virtual relative dimension $d$, there is a functorial Gysin morphism
\[
f_*\colon A^i(X)\to A^{i+d}(Y),\quad  \alpha \mapsto \eta_f\cdot \alpha
\]
where $\eta_f$ is fundamental class of $f$. In particular, if $i\colon Z\hookrightarrow X$ is a regular closed immersion of codimension $r$, then $[Z]_A:=i_*1\in A^r(X)$ is the fundamental class of $Z$ in $X$ \cite[Theorem 3.3.2 and Proposition 3.3.4]{DegliseJinKhanFundamentalClasses}.

\item[(RG)] \textbf{Refined Gysin pullbacks and excess intersection.}
For every cartesian square of smooth schemes
\[
\begin{tikzcd}
Z' \ar[r,"i'"] \ar[d,"g'"'] & X' \ar[d,"g"]\\
Z \ar[r,"i"'] & X
\end{tikzcd}
\]
with $i$ a regular closed immersion of codimension $r$, there is a functorial refined Gysin pullback
\[
g^*\colon A^m(Z)\to A^{m}(Z').
\]
One has the \emph{Excess Intersection Formula}
\begin{equation}\label{eq:excess-intersection-formula}
    g^*\circ i_*\;=\; i'_*\circ\bigl(c_\text{top}^A(\Q)\cdot g'^*\bigr),
\end{equation}
where $\Q$ is the excess normal bundle defined by \[0 \to \N_{Z^{\prime} / X^{\prime}} \to g^{\prime *} \N_{Z / X} \to \Q \to 0.\] In particular, taking $X'=Z$ and $g=i$ yields the \emph{Self Intersection Formula}  \cite[Proposition 4.2.2 and Corollary 4.2.3]{DegliseJinKhanFundamentalClasses}
\begin{equation}\label{eq:self-intersection-formula}
    i^*\circ i_*\;=\; c^A_{\text{top}}(\N_{Z/X})\cdot (-).
\end{equation}

\end{itemize}

\subsection{Formal Group Laws, Algebraic Cobordism, and Other Examples of Oriented Cohomology Theories}
\label{subsec:formal-group-laws}
\begin{definition}[Formal group law]
Let $R$ be a commutative ring. A \emph{one-dimensional commutative formal group law} over $R$ is a formal power series
\[
F(x,y)=x+y+\sum_{i,j\ge 1} a_{ij}x^iy^j\in R\llbracket x,y\rrbracket
\]
satisfying

\begin{itemize}
    \item \textbf{Identity.} $F(x,0)=x$ and $F(0,y)=y$.
    \item \textbf{Commutativity.} $F(x,y)=F(y,x)$.
    \item \textbf{Associativity.} In $R\llbracket x,y,z\rrbracket$, one has $F\bigl(x,F(y,z)\bigr)=F\bigl(F(x,y),z\bigr)$.
    \item \textbf{Formal inverse.} There exists a unique power series $\chi_F(x)\in R\llbracket x\rrbracket$
with $F\bigl(x,\chi_F(x)\bigr)=0$ and $F\bigl(\chi_F(x),x\bigr)=0$.
\end{itemize}
\end{definition}
\paragraph{Notation.} We define shorthands for addition under such a group law $F$ with \[x \Fplus y := F(x,y),\quad x \Fminus y := F(x, \chi_F(y)),  \quad \text{ and } \Fsum_{i = 1}^r x_i = x_1 \Fplus x_2 \Fplus \cdots \Fplus x_r.\]There are many scenarios where one needs to add the same element multiple times under $\Fsum$. We introduce another shorthand notation for this as follows. 

\begin{definition}[$n$-series of $F$]
    Let $R$ and $F$ be as above. Let $n$ be any positive integer. The \textit{$n$-series} $[n]_F\;x$ is the formal power series $\Fsum_{i = 1}^n x$ in $R\llbracket x\rrbracket$.
\end{definition}

\begin{theorem}[Lazard \cite{Lazard1955FormalGroupLaw}]
There exists a graded ring $\mathbb L$ and a formal group law $F_{\text{univ}}$ over $\mathbb L$ with the following property.
For every commutative ring $R$ and every formal group law $F$ over $R$, there exists a unique ring homomorphism
$\varphi_F\colon \mathbb L\to R$ such that $F$ is obtained from $F_{\text{univ}}$ by base change along $\varphi_F$.
\end{theorem}

The ring $\mathbb{L}$, called the \emph{Lazard ring}, is the graded ring $\mathbb{L}=\mathbb{Z}[a_{i j} \mid i, j \geq 1]/ I$, where $a_{i j}$ has degree $1-i-j$,\footnote{Note that $a_{ij}$ would be in even degrees for complex cobordism. This degree-doubling phenomenon is also observed in the comparison between gradings in the Chow ring versus in singular cohomology for a complex variety.} and $I$ is the ideal generated by the associativity and commutativity relations. The \emph{universal formal group law} is

$$
F_{\text {univ }}(x, y)=x+y+\sum_{i, j \geq 1} a_{i j} x^i y^j \in \mathbb{L} \llbracket x, y \rrbracket.
$$
The assignment $A^\bullet\mapsto F_A$ is functorial in the sense that a morphism of oriented cohomology theories
induces a homomorphism $A^\bullet(\pt)\to B^\bullet(\pt)$ that carries $F_A$ to $F_B$. Levine--Morel showed that for a field of characteristic zero, there exists an oriented cohomology theory $\Omega^\bullet$, called \emph{algebraic cobordism}, such that $\Omega^\bullet(\pt) \cong \mathbb{L}$, and the formal group law of $\Omega^\bullet$ is $F_{\text{univ}}$. 
For an arbitrary oriented cohomology theory $A^\bullet$, the formal group law $F_A$ therefore determines a canonical
graded ring homomorphism
\[
\varphi_A\colon \mathbb L\to A^\bullet(\pt)
\]
classifying $F_A$, making $A^\bullet(X)$ into a $\Omega^\bullet(X)$-module. In this sense, algebraic cobordism is the \emph{universal oriented cohomology theory} on $\mathsf{Sm}_k$. 

We now give a few examples of oriented cohomology theories and their formal group laws.

\begin{example}[Chow theory]
The Chow ring is a prototypical example of an oriented cohomology theory. The Chow-theoretic first Chern class of a tensor product of line bundles is governed by the \emph{additive formal group law},
\[
c_1^{CH}(\L\otimes \M)=c_1^{CH}(\L)+c_1^{CH}(\M).
\]
Thus, $F_{CH}(x,y)=x+y$ and $\chi_{CH}(x)=-x$.
\end{example}

\begin{example}[Algebraic $K$ theory]
Let $K_0(X) := K_0(\mathsf{VB}_X)$ be the Grothendieck group of vector bundles over $X\in \mathsf{Sm}_k$, where addition is given by direct sum and multiplication is given by tensor product. The \emph{(periodic) algebraic $K$-theory} of $X$ is $K_0(X)\llbracket \beta,\beta^{-1}\rrbracket$, where the \emph{Bott element} $\beta$ is a formal variable of degree $-1$. The $K$-theoretic first Chern class of a line bundle $\L$ is
\[
c_1^K(\L):=\beta^{-1}\bigl(1-[\L^\vee]\bigr).
\]
Then, $$\begin{gathered}
    c_1^K(\L \otimes \M) = \beta^{-1}(1-[ (\L \otimes \M)^{\vee} ]) = \beta^{-1}(1-[\L^{\vee}][\M^{\vee}]) \\ = \beta^{-1}\bigl((1-[\L^{\vee}]) + (1-[\M^{\vee}]) - (1-[\L^{\vee}])(1-[\M^{\vee}])\bigr),
\end{gathered}$$ which implies that
\[
F_K(x,y)=x+y-\beta xy,
\qquad \text{ and }\quad 
\chi_K(x)=\frac{-x}{1-\beta x} = -x \sum_{i \geq 0}(\beta x)^i .
\]
This is the \emph{multiplicative formal group law.} Note that the $-1$-degree of $\beta$ keeps all terms of $F_K$ in degree $1$. Specializing to $\beta \mapsto 1$, we recover the theory of Chern classes in the $K_0$-ring described by, for example, Karoubi \cite{Karoubi1978KTheory}. 
\end{example}

\begin{example}[Algebraic cobordism]
For a quasi-projective $k$-scheme $X$, Levine--Morel defines the algebraic cobordism group $\Omega_i(X)$ to be the quotient of the free abelian group generated by cobordism cycles
\(
[f\colon Y\to X],
\)
where $Y$ is a smooth quasi-projective $k$-scheme of pure dimension $i$ and $f$ is a projective morphism, modulo the following relations.
\begin{itemize}
\item \textbf{Dimension.} If $Y$ is smooth and $\L_1,\dots,\L_r$ are line bundles on $Y$ with $r>\dim(Y)$, then
\[
c_1^\Omega(\L_1)\cdots c_1^\Omega(\L_r)\cdot [\,\mathrm{id}_Y:Y\to Y\,]=0.
\]
\item \textbf{Section.} If $Y$ is smooth and $\L$ is a line bundle on $Y$ admitting a section $s$ transverse to the zero section, with smooth zero locus $i\colon Z(s)\hookrightarrow Y$, then
\[
c_1^\Omega(L)\cdot[\,\mathrm{id}_Y:Y\to Y\,]=i_\ast[\,\mathrm{id}_{Z(s)}:Z(s)\to Z(s)\,].
\]
\item \textbf{Formal group law.} Algebraic cobordism satisfies $F_{\text{univ}}$.
\end{itemize}
The resulting functor $\Omega_\bullet$ on quasi-projective $k$-schemes is an \textit{oriented Borel--Moore homology theory.} If $X$ is smooth of pure dimension $d$, one sets
\[
\Omega^n(X)=\Omega_{d-n}(X),
\]
and this makes $\Omega^\bullet$ into an oriented cohomology theory on $\Sm_k$.
The universality of algebraic cobordism is the algebraic analogue of Quillen's theorem for the universality of complex cobordism among complex-oriented cohomology theories on topological spaces \cite[Theorem 2]{Quillen1969ComplexCobordism}.
\end{example}

\begin{theorem}[Universality of algebraic cobordism, \cite{LevineMorel} Theorem 7.1.3]\label{thm:universality-of-cobordism}
    Let $A^\bullet$ be any oriented cohomology theory on $\Sm_k$, where $\operatorname{char} k = 0$. For any scheme $X\in \Sm_k$, there exists a ring homomorphism $\Omega^\bullet(X) \to A^\bullet(X)$, which induces a natural graded homomorphism of $\mathbb{L}$-algebras
$$
\Omega^{\bullet}(X) \otimes_{\mathbb{L}} A^{\bullet}(\pt) \rightarrow A^{\bullet}(X) .
$$
\end{theorem}
\begin{example}[Ordinary cohomology theories]
There are several classical examples of \emph{ordinary} oriented cohomology theories, that is, oriented cohomology theories whose formal group law is additive. Let $k$ be a field. If $\ell\neq \operatorname{char}(k)$ is a prime, one may consider the even part of  $\ell$-adic \'{e}tale cohomology theory \cite{Milne1980Etale}
\[
A^\bullet_\text{et}(X):=\bigoplus_{n\geq 0} H^{2n}_\text{et}(X,\mathbb{Q}_\ell(n)),
\qquad X\in \Sm_k.
\] If $\operatorname{char}(k)=0$, one may also consider the even part of algebraic de Rham cohomology theory \cite{Hartshorne1975deRham}
\[
A^\bullet_\text{dR}(X):=\bigoplus_{n\geq 0} H^{2n}_\text{dR}(X/k),
\qquad
H^i_\text{dR}(X/k):=\mathbb{H}^i(X,\varOmega^\bullet_{X/k}),
\]
where $\varOmega^\bullet_{X/k}$ is the algebraic de Rham complex. 
In both cases, the first Chern classes satisfy the additive formal group law \(
c_1(\mathcal{L}\otimes \mathcal{M})=c_1(\mathcal{L})+c_1(\mathcal{M}).
\)

Over a field of characteristic zero, Chow theory is universal among ordinary oriented cohomology theories. That is, if $A^\bullet_\text{ord}$ is any oriented cohomology theory on $\Sm_k$ with the additive formal group law, then for any $X\in \Sm_k$, there is a natural homomorphism of graded rings $CH^\bullet(X) \to A^\bullet_\text{ord}(X)$.
The usual \textit{cycle class maps} from Chow groups to even \'{e}tale cohomology, even algebraic de Rham cohomology, and, when $k\hookrightarrow \mathbb{C}$ is a subfield, even Betti cohomology, are all instances of this universal homomorphism.
\end{example}

\begin{example}[Landweber exact theories] \label{eg:landweber-exact-theories}
\textit{Landweber exactness} is a flatness condition on a formal group law that guarantees that the naive change of coefficients construction from $\Omega^\bullet$ produces a genuine oriented cohomology theory.
Let $R$ be a graded ring and let $\varphi\colon \mathbb L\to R$ be a graded ring homomorphism classifying a formal group law over $R$.
One may then form the extension of scalars
\[
A_R^\bullet(X)=\Omega^\bullet(X)\otimes_{\mathbb L} R.
\]
When $\varphi$ is Landweber exact, the resulting functor $A_R^\bullet$ satisfies the axioms of an oriented cohomology theory and its formal group law is the base change of the universal formal group law $F_{\text{univ}}$ along $\varphi$. For suitable choices of $(R, \varphi)$, this recovers Chow groups, algebraic $K$-theory, and other Landweber exact theories such as motivic Brown--Peterson $BP$-theories, motivic Johnson--Wilson $E(n)$ theories, and \textit{elliptic cohomology} determined by the formal group of an elliptic curve \cite{NaumannSpitzweck2009Landweber, NaumannEtAl2009MotivicLandweber}. 
\end{example}

Certain oriented cohomology theories satisfy one additional axiom on first Chern classes called the \textit{dimension axiom}. One class of examples are those corresponding to Levine--Morel's oriented Borel--Moore homology theories of \textit{geometric type}. With the dimension axiom, the formal group law always truncates to a finite sum on every smooth scheme. We will assume this additional axiom in \Cref{sec:oriented-cohomology-of-M0nbar}.
\begin{itemize}
\item[(Dim)] \textbf{Dimension axiom (optional).} For every $X\in \Sm_k$ of pure dimension $d$ and every collection of line bundles $\mathcal L_1,\dots,\mathcal L_n$ on $X$ with $n>d$, one has
\[
c_1^A(\mathcal L_1)\cdots c_1^A(\mathcal L_n)=0
\qquad\text{in}\qquad
A^\bullet(X).
\]
Equivalently, for every line bundle $\mathcal L$ on $X$, the class $c_1^A(\mathcal L)\in A^1(X)$ is nilpotent, with \(c_1^A(\mathcal L)^{d+1}=0\).
\end{itemize}
Chow theory and algebraic cobordism are basic examples of theories of geometric type. More generally, any theory obtained from algebraic cobordism by extension of scalars is geometric in this sense, so in particular periodic algebraic $K$-theory is geometric. By contrast, the classical ordinary theories given by even \'{e}tale cohomology, even algebraic de Rham cohomology, even Betti cohomology, and the complex realization $X\mapsto MU^{2,\ast}(X(\mathbb C))$ are oriented cohomology theories on $\Sm_k$, but are not usually presented as restrictions of Levine--Morel geometric-type Borel--Moore theories on all finite type schemes. In particular, \emph{``not geometric''} here does not mean that (Dim) fails on smooth schemes. In these ordinary theories, (Dim) holds on a smooth $d$-fold simply for degree reasons.

\subsection{Oriented Motivic Spectra}
\label{subsec:oriented-theories-and-oriented-spectra}

To establish the ``Gysin package'' of fundamental classes, refined Gysin morphisms, and the excess intersection formula, it is convenient to work in the \emph{stable motivic homotopy category} $\mathsf{SH}(k)$. This is the algebro-geometric analogue of the classical stable homotopy category, obtained by stabilizing the homotopy theory of smooth $k$ schemes with respect to $\PP^1$ suspension and imposing Nisnevich descent and $\AA^1$ homotopy invariance. Objects of $\mathsf{SH}(k)$ are \emph{motivic spectra} that represent cohomology theories on $\mathsf{Sm}_k$. Déglise--Jin--Khan \cite{DegliseJinKhanFundamentalClasses} showed that the six-functors formalism on $\mathsf{SH}(k)$ implies the above ``Gysin package.'' We now summarize some useful facts about $\mathsf{SH}(k)$ following \cite{DegliseJinKhanFundamentalClasses}.

We begin by summarizing some basic notions. A \textit{motivic space} is a presheaf of topological spaces on $\mathsf{Sm}_k$ satisfying Nisnevich descent and $\mathbb{A}^1$-homotopy invariance. A \textit{pointed motivic space} $(X, x)$ is a motivic space $X$ with a chosen basepoint $x \cong \Spec k\to X$. In particular, $X_+$ denotes the motivic space $X$ with a disjoint basepoint. The category of pointed motivic spaces $\mathsf{H}(k)$ admits all limits and colimits. It is symmetric monoidal under the smash product $X\wedge Y := X\times Y / X\vee Y$, where $X\vee Y$ is the wedge sum identifying the basepoints. The inclusion $\Sm_k \to \mathsf{H}(k)$ is fully faithful, given by the Yoneda embedding.

Let $S^1$ denote the pointed simplicial circle, and let $\GG_m = \AA^1-0$ be pointed at $1$. For non-negative integers $p>q$, define the $(p,q)$-\textit{suspension} of $X$ to be $\Sigma^{p,q}X := (S^1)^{\wedge (p-q)}\wedge (\GG_m)^{\wedge q} \wedge X$. In particular, because there is an equivalence $(\PP^1, \infty) \simeq S^1\wedge \GG_m$ in $\mathsf{H}(k)$, the \textit{$\PP^1$-suspension} of $X$ is defined as $\Sigma^{2,1}X$. 

A \textit{$\PP^1$-suspension spectrum}\footnote{Simply, a $\PP^1$-spectrum.} is a sequence of pointed motivic spaces $\mathbb  E = \left(\dots, \mathbb{E}_{-1},\mathbb E_0, \mathbb E_1, \mathbb E_2, \ldots\right)$ equipped with structure maps $\Sigma^{2,1} \mathbb E_n \rightarrow \mathbb E_{n+1}$. The Morel--Voevodsky \textit{stable motivic category} $\mathsf{SH}(k)$ consists of such $\PP^1$-spectra. It again has all limits and colimits, and finite limits canonically coincide with colimits. The category $\mathsf{H}(k)$ embeds into $\mathsf{SH}(k)$ by sending each motivic space $X$ to the spectrum $\Sigma^\infty X_+ := (\dots, \pt,\pt,X_+,\pt, \pt, \dots)$, where $X_+$ is in degree $0$. Each $\PP^1$-spectrum represents a bigraded cohomology theory on $\mathsf{Sm}_k$ via $\mathbb E^{p, q}(X)= \left[\Sigma^{p, q}\Sigma_{\PP^1}^{\infty} X_{+},  \mathbb E\right]_{\mathsf{SH}(k)}$. A spectrum $\mathbb{E}$ is a \textit{motivic ring spectrum} if the associated cohomology theory $\mathbb{E}^{\bullet, \bullet}(X)$ is a bigraded ring. The direction of $\PP^1$-suspension in the $(p,q)$-bigrading where $p = 2q$ is often referred to as the \textit{Chow line}. Restricting to this line, one obtains the associated singly-graded ring $A_{\mathbb{E}}^\bullet(X) := \mathbb{E}^{2\bullet, \bullet}(X)$. 

Let $p\colon X\to \Spec(k)$ be a separated morphism of finite type. Let $\V$ be a vector bundle, and denote its total space by $V$. Then, the \emph{Thom space} of $\V$ is the motivic space
\[
Th_X(V)=V/(V- X),
\]
is the cofiber of the inclusion $V-X \to V$. This definition can be extended to any class $v \in K_0(X)$. The \textit{relative purity theorem} of Morel--Voevodsky states that for a closed immersion of smooth schemes $i: Z\to X$, there is an equivalence $X /(X-Z) \simeq \operatorname{Th}(\N_{Z / X})$ in $\mathsf{SH}(k)$ \cite[Theorem 3.2.23]{MorelVoevodsky1999}. As a consequence, one obtains the localization sequence, which is standard.

\begin{theorem}[Localization sequence for cohomology theories in $\mathsf{SH}(k)$] \label{thm:localization-sequence}
Let $X\in \Sm_k$. Let $i: Z\to X$ be the closed immersion of a smooth subvariety of codimension $r$, and let $j: \, U = X-Z \to X$ be the open complement. Then, for any cohomology theory $A^\bullet$ represented in $\mathsf{SH}(k)$, there is an exact sequence \[A^{\bullet-r}(Z) \to A^\bullet(X) \to A^\bullet(U)\to 0.\]
\end{theorem}

Let $Gr(n,\infty)=\operatorname*{colim}_{m\ge n} Gr(n,m)$ be the infinite Grassmannian viewed as a pointed motivic space. Let $MGL$ be the $\PP^1$-spectrum whose $n^\text{th}$ space is the Thom space of the universal rank $n$ bundle $\gamma_n$ on $Gr(n,\infty)$. Equivalently, one may write
\[
MGL\simeq \operatorname*{colim}_{n\ge 0}\,\Sigma^{2n,n}\,\Sigma^\infty_{\PP^1}Th(\gamma_n).
\]
In this context, one can define the notion of orientation on motivic spectra using the universality of cobordism (\Cref{thm:universality-of-cobordism}).

\begin{definition}[Oriented motivic ring spectra]
    A $\PP^1$-ring spectrum $\mathbb{E}\in \mathsf{SH}(k)$ is \textit{oriented} if it receives a morphism of motivic ring spectra $\vartheta: MGL \to \mathbb{E}$. 
\end{definition}

We note that both Panin \cite{Panin2003Oriented} and Déglise--Jin--Khan \cite[Definition 4.4.1]{DegliseJinKhanFundamentalClasses} gave axiomatic definitions of orientation in terms of Thom spaces, which turn out to be equivalent to the above universality statement. The various results of \cite{DegliseJinKhanFundamentalClasses} can be summarized as the following theorem.

\begin{theorem}[Representability of Oriented Cohomology Theories]
\label{thm:OCT-representability}
Let $\mathbb E$ be a motivic ring spectrum in $\SH(k)$. For $X\in \Sm_k$ and $n\in \mathbb{Z}$, set
\[
A_\mathbb E^n(X)=\Hom_{\SH(k)}\bigl(\Sigma^\infty_{\PP^1}X_+,\Sigma^{2n,n}\mathbb E\bigr).
\]
The following statements are equivalent.

\begin{enumerate}
    \item The motivic ring spectrum $\mathbb E$ admits an orientation.
    \item The functor $A_\mathbb E^\bullet$ is an oriented cohomology theory admitting the ``Gysin package.''
\end{enumerate}

\end{theorem}

\begin{example}[A few oriented motivic ring spectra]
    All oriented cohomology theories on $\Sm_k$ considered in this paper are represented by oriented motivic ring spectra in $\SH(k)$. We list some recurring examples below. 

    \begin{itemize}
        \item Chow theory is represented by the motivic Eilenberg--MacLane spectrum $H\mathbb{Z}$.
One convenient model is the $\PP^1$-spectrum formed by the motivic Eilenberg--MacLane spaces $K(\mathbb{Z}(n),2n)$, and one may write
\[
H\mathbb{Z}\simeq \operatorname*{colim}_{n\ge 0}\,\Sigma^{-2n,-n}\,\Sigma^\infty_{\PP^1}K(\mathbb{Z}(n),2n).
\]

\item Algebraic $K$-theory \footnote{Note that the presheaf of Quillen $K$ theory spectra on $\Sm_k$ is famously non-$\AA^1$-invariant in general.
Its $\AA^1$-localization is Weibel's \textit{homotopy invariant $K$-theory} $KH$.
The motivic spectrum $KGL$ represents $KH$, and for $X\in \Sm_k$ one has a canonical identification $KH_\bullet(X)\cong K_\bullet(X)$ because $X$ is regular.} is represented by the Bott periodic spectrum $KGL$.
It can be constructed from the motivic classifying space $\mathbb{Z}\times BGL$ by inverting the Bott element $\beta$, so that
\[
KGL\simeq \operatorname*{colim}_{n\ge 0}\,\Sigma^{-2n,-n}\,\Sigma^\infty_{\PP^1}(\mathbb{Z}\times BGL),
\]
where the transition maps are induced by multiplication by $\beta$. The $KGL^{2p,p}$-cohomology is sometimes referred to as periodic $K$-theory. 
\item Algebraic cobordism $\Omega^\bullet$ is represented by the universal oriented motivic ring spectrum $MGL$. 
    \end{itemize}
\end{example}

\subsection{Fundamental Classes}
In this section, we recall the definition of fundamental classes and record some of their useful properties. 
\begin{definition}[Fundamental class for a regular immersion]
Let $X\in \Sm_k$, and let $i\colon Z\hookrightarrow X$ be a regular closed immersion of codimension $d$ with $Z$ smooth.
The \textit{fundamental class} of $Z$ in $X$ is the $A^\bullet$-cohomology class
\[
[Z]_A \in A^d(X),
\qquad
[Z]_A := i_*1_Z,\footnote{\text{Since $A^\bullet(Z)$ is a unital ring, it has a unit which we denote by $1_Z$.}}
\]
where $i_*\colon A^\bullet(Z)\to A^{\bullet+d}(X)$ is the Gysin morphism attached to $i$.
\end{definition}

\begin{proposition}[The intersection product in $A^\bullet$-theory]\label{prop:intersection-product}
    Let $X\in \Sm_k$. Suppose $i: Z\to X$ and $j: W\to X$ are closed immersions of smooth subschemes of codimensions $c$ and $d$. Further assume $i, j$ are transverse and $Y := Z\times_X W$ is smooth of codimension $c+d$ in $X$. Then, the identity \[[Y]_A = [Z]_A\cdot [W]_A\]
    holds in $A^{c+d}(X)$.
\end{proposition}

\begin{proof}
    Consider the Cartesian square \[ \begin{tikzcd} Y\ar[r,"q"] \ar[d,"p"'] & W \ar[d,"j"]\\ Z \ar[r,"i"'] & X . \end{tikzcd}\]
    By transverse base change (A2), we have $i^*\circ j_* = p_* \circ q^*$. Applying to $1_W \in A^0(W)$ yields \[
       i^*[W]_A  =  i^*(j_*1_W) = p_*(q^*1_W) = p_*1_Y \in A^d(Z).
    \]
    By the projection formula (PF), \[[Z]_A\cdot [W]_A = i_*1_Z \, \cdot [W]_A = i_*(1_Z \cdot i^*[W]_A) = i_*(p_*1_Y) = [Y]_A\in A^{c+d}(X).\]
\end{proof}

\begin{remark}
   The above proposition requires the maps $i$ and $j$ to be transverse. Non-transverse intersections prompt the use of the Excess Intersection Formula (RG). 
\end{remark}

\begin{proposition}[Fundamental class of a divisor]\label{funclass-divisor}
    Let $X\in \Sm_k$ and let $i\colon D\hookrightarrow X$ be a smooth divisor. Then,
\[
[D]_A = c_1^A(\O_X(D))\in A^1(X).
\]
\end{proposition}

\begin{proof}
    Let $\mathcal{L} := \O_X(D)$, and let $p: L \to X$ be the total space of the line bundle $\mathcal{L}$.  Let $z: X\to L$ be the zero section, and let $s: X \to L$ be the section induced by $D$. Then, 
\[
\begin{tikzcd}
D \ar[r, "i"] & X \ar[r, shift left, "s"] \ar[r, shift right, "z"'] & L
\end{tikzcd}
\]
is a coequalizer diagram, which is in particular cartesian. Furthermore, $s$ and $z$ are transverse because $D$ is smooth. By transverse base change (A2), we have $z^*s_* = i_*i^*$. However, both $z$ and $s$ are sections to $p$, so we have $p\circ z = \operatorname{id}_X = p\circ s$. Furthermore, $p^*$ is an isomorphism by the extended homotopy property (EH), so we conclude that $z^* = s^*$. The transverse base change formula then simplifies to $s^*s_* = i_*i^*$. Applying both sides on the element $1_D$ and using the Self Intersection Formula, we obtain the equality \[[D]_A = i_*1_D = i_*i^*1_X = s^*s_*1_X = c_1^A(\mathcal{L}) \in A^1(X).\]
\end{proof}

\begin{remark}
    We see from the above proposition that divisor classes for any $A^\bullet$-theory are invariant under linear equivalence. 
\end{remark}

\begin{proposition}[Fundamental class of a sum of divisors]
\label{prop:divisor-sum-fgl}
Let $X\in \Sm_k$ and let $D_1$ and $D_2$ be effective Cartier divisors on $X$.
Then,
\[
[D_1+D_2]_A=[D_1]_A\Fplus[D_2]_A\in A^1(X),
\]
where $F$ is the formal group law of $A^\bullet$.
\end{proposition}

\begin{proof}
One has $\mathcal{O}_X(D_1+D_2)\simeq \mathcal{O}_X(D_1)\otimes \mathcal{O}_X(D_2)$.
The defining property of the formal group law gives
\[
c_1^A(\mathcal{O}_X(D_1+D_2))
=
c_1^A(\mathcal{O}_X(D_1))\;\Fplus  \;c_1^A(\mathcal{O}_X(D_2)).
\]
By \Cref{funclass-divisor}, we have $[D]_A=c_1^A(\mathcal{O}_X(D))$, which completes the proof.
\end{proof}
The above proposition for algebraic cobordism is recorded in \cite[Example 3.1.1]{LevineMorel}.

\begin{proposition}[Product of disjoint divisors]
\label{prop:disjoint-divisors-zero}
Let $X\in \Sm_k$ and let $D_1$ and $D_2$ be smooth divisors with $D_1\cap D_2=\varnothing$.
Then, $[D_1]_A\cdot [D_2]_A = 0$ in $A^2(X)$. 
\end{proposition}

\begin{proof}
Let $i_1\colon D_1\to X$ and $i_2\colon D_2\to X$ be the inclusions, and let $j\colon \, U=X- D_1\to X$ be the open complement. Since $D_2\cap D_1=\varnothing$, the map $i_2$ factors as $i_2=j\circ \widehat{i}_2$ for a unique $\widehat{i}_2\colon D_2\to U$.

By the projection formula for $i_{2,*}$,
\[
[D_1]_A\cdot [D_2]_A=[D_1]_A\cdot i_{2,*}(1)=i_{2,*}\bigl(i_2^\ast[D_1]_A\bigr),
\]
so it suffices to show that $i_2^\ast([D_1]_A)=0$ in $A^1(D_2)$.

The localization sequence \Cref{thm:localization-sequence} yields \[A^0(D_1) \xrightarrow{i_{1,*}} A^1(X) \xrightarrow{j^*} A^1(U) \to 0.\]

It follows that $j^\ast([D_1]_A) = j^*(i_{1,*}1_{D_1})=0$ by exactness at $A^1(X)$. Then
\[
i_2^\ast([D_1]_A)=\widehat{i}_2^\ast\bigl(j^\ast([D_1]_A)\bigr)=0,
\]
and hence $[D_1]_A\cdot [D_2]_A=i_{2,*}(0)=0$ in $A^2(X)$.
\end{proof}

\paragraph{Residue formula.} We now review the \textit{invariant differential} and Quillen's \textit{residue formula} for computing the proper pushforward along a projective bundle. This formula is of crucial importance for computing fundamental classes of varieties embedded in projective spaces (see for example \Cref{prop:fundamental-class-of-veronese}). Quillen's original proof was for complex cobordism in the topological category, but much of the proof translates immediately to the algebraic setting. The appendix in Vishik \cite{Vishik2007OperationsCobordism} contains a nice exposition of these concepts, culminating in Theorem 5.35 which has Quillen's residue formula for algebraic cobordism. Again, let $R$ be a commutative ring and let $F(x,y)\in R\llbracket x,y\rrbracket$ be a one-dimensional commutative formal group law. Write $\partial_2F(x,y)$ for the formal partial derivative of $F$ with respect to the second variable. Let $\operatorname{Res}$ be the usual formal residue operator on Laurent series. For example, if $\alpha(t)\,dt=\sum_{n\geq -N} a_n t^n\,dt$ is a formal Laurent series, then
\[
\operatorname{Res}_{t=0}\bigl(\alpha(t)\,dt\bigr)=a_{-1}.
\]

\begin{definition}[Invariant differential]
\label{def:invariant-differential}
The invariant differential of $F$ is the formal power series
\[
\omega_F(t)= \sum_{r = 0}^\infty \omega_rt^r\; dt :=\frac{dt}{\partial_2F|_{(t,0)}} \in R\llbracket t\rrbracket\,dt.
\]
\end{definition}

If we think of the formal group law as the addition rule in a formal Lie algebra, then its invariant differential can be thought of as the \textit{Haar measure} of the corresponding ``Lie group.''

\begin{example}
The Chow ring satisfies the additive group law
\[
F_{CH}(u,v)=u+v,
\]
so $\partial_2F_{CH}(t,0)=1$. Hence, the invariant differential is simply
\(
\omega_F(t)\,dt=dt.
\)
\end{example}

\begin{example}
Algebraic $K$-theory satisfies the multiplicative group law 
\[
F_K(u,v)=u+v-\beta uv.
\]
We have
\[
\partial_2F_K|_{(t,0)}=\frac{\partial}{\partial v}\bigl(t+v-\beta tv\bigr)\Big|_{v=0}=1-\beta t,
\]
so the invariant differential is
\[
\omega_{F_K}(t)\,dt=\frac{dt}{1-\beta t},
\qquad\text{i.e.}\qquad
\omega_{F_K}(t)=\frac{1}{1-\beta t}=\sum_{r\ge 0}\beta^r t^r.
\]
Thus the coefficients are
\[
\omega_r=\beta^r \in \ZZ[\beta, \beta^{-1}].
\]
\end{example}

\begin{example}
Algebraic cobordism satisfies the universal group law
\[
F(u,v)=u+v+\sum_{i,j\ge 1} a_{ij}u^iv^j,
\qquad a_{ij}\in \mathbb{L}^{\,1-i-j}.
\]
Differentiating in the second variable and setting $v=0$ kills all terms with $j\neq 1$, so
\[
\partial_2F|_{(t,0)}=1+\sum_{i\ge 1} a_{i1}t^i
=1+a_{11}t+a_{21}t^2+a_{31}t^3+\cdots.
\]
Therefore,
\[
\omega_F(t)\,dt=\sum_{r\ge 0}\omega_r t^r dt=\bigl(\partial_2F|_{(t,0)}\bigr)^{-1}dt.
\]
Equating coefficients in
$
\Bigl(1+\sum_{i\ge 1}a_{i1}t^i\Bigr)\Bigl(\sum_{r\ge 0}\omega_r t^r\Bigr)=1
$
gives the recursion
\[
\omega_0=1,\qquad
\omega_r=-\sum_{i=1}^r a_{i1}\,\omega_{r-i}\quad(r\ge 1),
\]
and in particular the first few terms are
\[
\omega_1=-a_{11},\qquad
\omega_2=a_{11}^2-a_{21},\qquad
\omega_3=-a_{31}+2a_{11}a_{21}-a_{11}^3.
\]
Any oriented theory $A^\bullet$ with formal group law $F_A$ is obtained from $\Omega^\bullet$ by a universal morphism, and the same formulas for $\omega_r$ hold after applying the induced ring map $\mathbb{L}\to A^\bullet(\Spec k)$.
\end{example}

\begin{theorem}[Quillen's Residue Formula, {\cite[Theorem 1]{Quillen1969ComplexCobordism}}]
\label{thm:quillen-residue-formula}
Let $A^\bullet$ be an oriented cohomology theory characterized by the formal group law $F_A$.
Let $\E$ be a rank $n$ vector bundle on a smooth scheme $X$. Let $\pi\colon \mathbb{P}(\E)\to X$ be the projective bundle, and let $\xi=c_1^A(\mathcal{O}_{\mathbb{P}(\E)}(1))$ be the hyperplane class.
Let $\alpha_1,\dots,\alpha_n$ denote the Chern roots of $E$, so that $c(\E)=\prod_{i=1}^n (1+ \alpha_i)$.
Then for every power series $\varphi(t)\in A^\bullet(X)\llbracket t\rrbracket$, one has
\[
\pi_\ast\bigl(\varphi(\xi)\bigr)
=
\operatorname{Res}_{t=0}
\frac{\varphi(t)\,\omega_{F_A}(t)}
{\prod_{i=1}^n \bigl(t -_F \alpha_i\bigr)}.
\]
\end{theorem}

\begin{corollary}[Pushforward from projective space to a point]
\label{cor:projective-space-residue}
Let $A^\bullet$ be an oriented cohomology theory on $\Sm_k$ with formal group law $F_A$. Let $\pi\colon \PP^{n}\to \Spec k$ be the structure morphism.
Let $\xi=c_1^A(\cO_{\PP^{n}}(1))$ be the hyperplane class and let
\(
\omega_{F_A}(t)=\sum_{r\ge 0}\omega_r t^r.
\)
Then,
\[
\pi_\ast(\xi^m)=\omega_{n-m}\quad (0\le m\le n),
\qquad\text{and}\qquad
\pi_\ast(\xi^m)=0\quad (m > n).
\]
\end{corollary}

\begin{example}
    
Continuing with previous examples, let $h\in CH^1(\PP^n)$ be the Chow-theoretic hyperplane class. The residue formula recovers the standard facts $\pi_\ast(h^{n-1})=1 \in CH^\bullet(\pt) \cong \ZZ$ and $\pi_\ast(h^m)=0$ for $m\neq n-1$. 

Let $\xi = \beta^{-1}(1 - [\O(-1)])$ be the $K$-theoretic hyperplane class. The residue formula recovers $\pi_\ast(\xi^m)=\beta^{\,n-m}$ for all $0\le m\le n$.  This computation is a manifestation of the Koszul resolution of the structure sheaf $[\O_{\PP^{n-m}}]$ inside of $\PP^n$. There is a locally free resolution \[
0 \to \cO_{\PP^n}(-m)
\to \cO_{\PP^n}(-(m-1))^{\binom{m}{m-1}}
\to \cdots
\to \cO_{\PP^n}(-2)^{\binom{m}{2}}
\to \cO_{\PP^n}(-1)^{\binom{m}{1}}
\to \cO_{\PP^n}
\to \cO_{\PP^{n-m}}
\to 0,
\]
which produces the relation $[\O_{\PP^{n-m}}] = (1-[\O(-1)])^m \in K_0(\PP^n)$. The pushforward $\pi_*: K_0(\PP^n) \to K_0(\pt) \cong \ZZ$ amounts to computing the Euler characteristics. In particular, we have $\chi(\PP^n, \O_{\PP^{n-m}}) = 1$. Since $\pi_*$ preserves the \textit{dimension} grading, the class $\pi_*(\xi^{m})$ lands in degree $m-n$, which is generated by the element $\beta^{n-m}$. Therefore, $\pi_*(\xi^m) = 1 \cdot \beta^{n-m}$ in $K$-theory. 

\end{example}

\section{The Blowup Presentation for Oriented Cohomology Theories}\label{sec:multiplicative-blowup-formula}

In this section, we first prove a blowup formula for arbitrary (not necessarily $\AA^1$-invariant) oriented cohomology theories represented in the category $\mathsf{MS}_S$, which states that the $\mathbb{E}^{p,q}$-cohomology of the blowup of a smooth scheme along a smooth center decomposes into a direct sum of the cohomology of the base and copies of the cohomology of the blowup center. Here, $S$ can be taken to be an arbitrary scheme. This enables such a decomposition to be applied to non-$\AA^1$-invariant theories like topological Hochschild homology, topological cyclic homology, and Hodge cohomology. 

Secondly, we move back to the $\AA^1$-invariant category $\mathsf{SH}(k)$ and give a general presentation of the oriented cohomology ring (along the degree $(2p,p)$-line) of such a blowup, in terms of geometric generators and four families of relations. Examples of theories here include the Chow ring, algebraic $K$-theory, and algebraic cobordism. The switch back to $\mathsf{SH}(k)$ is necessitated by the absence of established ``intersection theory package'' and ``Gysin package'' in $\mathsf{MS}_S$, to the best of our knowledge. Nevertheless, we remark that intersection-theoretic calculations can still be largely performed. For example, we sketch an argument for the self-intersection formula in $\mathsf{MS}_S$, which has been generously pointed out to us by Toni Annala.

\subsection{A Blowup Formula in $\mathsf{MS}_S$}
\label{subsec:blowup-square-and-additive}

Let $X\in \Sm_k$ and let $i\colon Z\hookrightarrow X$ be a regular closed immersion of codimension $r\ge 2$.
Let $\pi\colon \Bl_Z X\to X$ be the blowup of $X$ along $Z$, and let $j\colon E\hookrightarrow \Bl_Z X$ be the exceptional divisor.
The projection $p\colon E\to Z$ identifies $E$ with $\mathbb{P}(\mathcal{N}_{Z/X})$, where $\mathcal{N}_{Z/X}$ is the normal bundle of $Z$ in $X$.
The \textit{blowup square} is the cartesian square
\[
\begin{tikzcd}
E \ar[r,"j"] \ar[d,"p"'] & \Bl_Z X \ar[d,"\pi"] \\
Z \ar[r,"i"'] & X.
\end{tikzcd}
\]

The additive blowup decomposition is obtained from a Mayer--Vietoris long exact sequence for the blowup square together with the projective bundle formula for $p\colon E\to Z$. While the corresponding result is well-known for $\mathbb A^1$-invariant oriented theories, we prove it in the non-$\mathbb A^1$-invariant setting by working in the category $\mathsf{MS}_S$ of motivic spectra developed by Annala, Hoyois, and Iwasa \cite{AnnalaIwasa2022MotivicSpectra, AHI2024AtiyahDuality, AHI2025CobordismConnerFloyd}. This allows one to treat many non-$\AA^1$-invariant cohomology theories that occur in practice, including topological Hochschild homology, topological cyclic homology, and Hodge cohomology. \textit{For the blowup formula in this subsection, we work over an arbitrary base scheme $S$ and do not assume $\mathbb A^1$-invariance.}

Let $S$ be any scheme and let $\Sm_S$ denote the category of smooth schemes over $S$.
A presheaf $\mathcal{F}\colon \Sm_S^{\op}\to \Sp$ satisfies \textit{smooth blowup excision} (sbu) if it sends the empty scheme to a terminal object and carries every blowup square in $\Sm_S$ to a cartesian square \cite[Definition 2.1]{AHI2025CobordismConnerFloyd}. Annala--Iwasa defined a similar condition called \textit{elementary blowup excision} (ebu). The category of \textit{motivic spectra} is the $\PP^1$-stabilization of Zariski presheaves of spectra on $\Sm_S$ satisfying ebu. In other words,
\[
\MS_S:=\mathsf{PSh}_{\operatorname{Zar},\operatorname{ebu}}(\Sm_S,\Sp)[(\mathbb{P}^1)^{-1}].
\]
For oriented motivic spectra satisfying Nisnevich descent, Annala--Iwasa showed that sbu is equivalent to ebu, which is furthermore equivalent to the projective bundle formula. Moreover, this non-$\AA^1$-invariant category is an enlargement of the usual Morel--Voevodsky motivic category. Its $\AA^1$-invariant full subcategory recovers the classical stable motivic homotopy category $\mathsf{SH}(S)$, so all $\AA^1$-invariant theories live naturally inside this broader framework.

 The category $\mathsf{MS}_S$ also admits Gysin maps, despite the fact that the full ``intersection theory package'' and ``Gysin package'' has not yet been established. As a consequence, one obtains the self-intersection formula, which we discuss below. For a closed immersion \(i\colon Z\hookrightarrow X\) of codimension $r$ in \(\Sm_S\), Annala--Iwasa--Hoyois \cite{AHI2024AtiyahDuality} proved the various functoriality properties of the Gysin map \(\operatorname{gys}_i\colon X_+\to \operatorname{Th}_Z(\N_{Z/X})\) which is based on Tang's construction of the Thom collapse map \cite{Tang2024PrismaticDuality}. On the other hand, for any oriented motivic ring spectrum $\mathbb{E}\in \mathsf{MS}_S$, there is a \textit{Thom isomorphism} $t h_{\N_{Z/X}}: \mathbb{E}^{a, b}(Z) \xrightarrow{\cong} \mathbb{E}^{a+2 r, b+r}( \operatorname{Th}_Z(\N_{Z/X}))$ for the rank-$r$ bundle $\N_{Z/X}$. Thus, one obtains the cohomological pushforward
\[
i_\ast
=
\operatorname{gys}_i^\ast\circ \mathrm{th}_{\N_{Z/X}}
\colon
\mathbb E^{a,b}(Z)\to \mathbb E^{a+2r,b+r}(X).
\]
Thus the composite $i^*i_*$ is obtained by pulling back the Thom class along the composite
\[
Z \longrightarrow X \xrightarrow{\mathrm{gys}_i} \operatorname{Th}_Z(N_{Z/X}),
\]
which is the zero section of the Thom space. It is a standard fact that this pullback is precisely the top Chern class $c_r(\N_{Z/X})$. As a consequence, the self-intersection formula
\begin{equation}\label{eq:MS-self-intersection-formula}
    i^\ast i_\ast y=c_r(\N_{Z/X})\cdot y
\end{equation}
still holds for an oriented cohomology theory in $\mathsf{MS}_S$. See also \cite[Theorem 2.14]{AnnalaShin2025Steenrod}.

We turn to the proof of the blowup formula in the general context of $\mathsf{MS}_S$. We remark that the following arguments transfer verbatim to theories in the $\AA^1$-invariant category $\mathsf{SH}(k)$. Fix a motivic ring spectrum $\mathbb E$ in $\MS_S$ and as before, write
\(
\mathbb E^{p,q}(X)=\left[\Sigma^{p, q}\Sigma_{\PP^1}^{\infty} X_{+},  \mathbb E\right]_{\mathsf{MS}_S}
\).
\begin{proposition}
\label{prop:LES-of-blowup-bigraded}
Let $\mathbb E$ be a motivic ring spectrum in $\MS_S$ satisfying smooth blowup excision.
Then, for any $p,q\in \ZZ$, the blowup square induces a long exact sequence 
\[
\cdots \to \mathbb E^{p,q}(X)\to \mathbb E^{p,q}(\Bl_Z X)\oplus \mathbb E^{p,q}(Z)\to \mathbb E^{p,q}(E)\to \mathbb E^{p+1,q}(X)\to \cdots .
\]
\end{proposition}

\begin{proof}
Smooth blowup excision implies that the blowup square is sent to a homotopy cartesian square of spectra.
Applying $\mathsf{Map}_{\MS_S}(\Sigma^\infty_{\PP^1}(-)_+,\Sigma_{\PP^1}^q \mathbb E[p])$ preserves homotopy limits and therefore sends the blowup square to a homotopy cartesian square.
Taking $\pi_0$ and the associated connecting morphisms gives the desired long exact sequence.
\end{proof}
\begin{lemma}
\label{lem:surjectivity-on-exceptional-divisor}
Let $\mathbb{E}$, $Z$, and $X$ be as above. For every $p,q\in \ZZ$, the natural map
\[
\mathbb E^{p,q}(\Bl_ZX)\oplus \mathbb E^{p,q}(Z)\longrightarrow \mathbb E^{p,q}(E)
\]
appearing in \Cref{prop:LES-of-blowup-bigraded} is surjective.
\end{lemma}
\begin{proof}
Set $\xi=c_1(\cO_E(-1))\in \mathbb E^{2,1}(E)$. Since $\mathbb E$ satisfies the projective bundle formula and $E\simeq \PP(\N_{Z/X})$ with $\operatorname{rk}(\N_{Z/X})=r$, multiplication by powers of $\xi$ gives an isomorphism
\[
\bigoplus_{k=0}^{r-1}\mathbb E^{p-2k,q-k}(Z)\xrightarrow{\sim}\mathbb E^{p,q}(E),
\qquad
(x_0,\dots,x_{r-1})\mapsto \sum_{k=0}^{r-1}\xi^k p^\ast(x_k).
\]
Thus every class $u\in \mathbb E^{p,q}(E)$ may be written uniquely as $u=p^\ast(x_0)+\sum_{k=1}^{r-1}\xi^k p^\ast(x_k)$ with $x_k\in \mathbb E^{p-2k,q-k}(Z)$.

The first summand $p^\ast(x_0)$ lies in the image of the $\mathbb E^{p,q}(Z)$-summand. Hence, it remains to show that each class $\xi^k p^\ast(x_k)$ for $1\le k\le r-1$ lies in the image of $\mathbb E^{p,q}(\Bl_ZX)$. Since $j\colon E\hookrightarrow \Bl_ZX$ is a codimension-one regular immersion, the orientation on $\mathbb E$ provides a Gysin pushforward $j_\ast\colon \mathbb E^{a,b}(E)\to \mathbb E^{a+2,b+1}(\Bl_ZX)$. Moreover, the normal bundle of the exceptional divisor is $\N_{E/\Bl_ZX}\cong \cO_E(-1)$, so by the self-intersection formula \eqref{eq:MS-self-intersection-formula}, we have $j^\ast j_\ast(y)=c_1(\N_{E/\Bl_ZX})\cdot y=\xi\cdot y$ for all $y\in \mathbb E^{\bullet,\bullet}(E)$. Applying this to $y=\xi^{k-1}p^\ast(x_k)\in \mathbb E^{p-2,q-1}(E)$ gives $j^\ast j_\ast\!\bigl(\xi^{k-1}p^\ast(x_k)\bigr) = \xi^k p^\ast(x_k)$. Hence, each summand $\xi^k p^\ast(x_k)$ for $1\le k\le r-1$ lies in the image of the restriction map $\mathbb E^{p,q}(\Bl_ZX)\to \mathbb E^{p,q}(E)$. Therefore, every class in $\mathbb E^{p,q}(E)$ lies in the image of $\mathbb E^{p,q}(\Bl_ZX)\oplus \mathbb E^{p,q}(Z)\to \mathbb E^{p,q}(E)$, which proves the surjectivity.
\end{proof}
\begin{theorem}[Blowup Formula for motivic spectra in $\mathsf{MS}_S$]
\label{thm:blowup-formula-additive-bigraded}
Let $\mathbb E$ be an oriented motivic ring spectrum in $\MS_S$ satisfying smooth blowup excision and the projective bundle formula.
Let $i\colon Z\hookrightarrow X$ be a regular closed immersion of codimension $r\ge 2$ and let $\pi\colon \Bl_Z X\to X$ be the blowup with exceptional divisor $p\colon E\simeq \PP(\mathcal{N}_{Z/X})\to Z$.
Then, for every $p,q\in \ZZ$, there is a direct sum decomposition
\[
\mathbb E^{p,q}(\Bl_Z X)\cong \mathbb E^{p,q}(X)\oplus \bigoplus_{k=1}^{r-1}\mathbb E^{p-2k,q-k}(Z). \footnote{\text{Examples of these higher $\EE^{p,q}$ groups include higher Chow groups and higher $K$-groups. }}
\]
\end{theorem}

\begin{proof}
\Cref{prop:LES-of-blowup-bigraded} gives a long exact sequence in bidegree $(p,q)$.
Since $E\cong \PP(\mathcal{N}_{Z/X})$ and $\mathcal{N}_{Z/X}$ has rank $r$, the projective bundle formula gives an isomorphism
\[
\Phi\colon \bigoplus_{k=0}^{r-1}\mathbb E^{p-2k,q-k}(Z)\xrightarrow{\sim}\mathbb E^{p,q}(E),
\qquad
(x_0,\dots,x_{r-1})\mapsto \sum_{k=0}^{r-1} \zeta^k\cdot p^\ast(x_k),
\]
where $\zeta=c_1(\mathcal{O}_E(1))\in \mathbb E^{2,1}(E)$.
The summand $k=0$ is the map $p^\ast\colon \mathbb E^{p,q}(Z)\to \mathbb E^{p,q}(E)$, which admits a splitting by the projective bundle formula. 
The connecting morphism $\mathbb E^{p,q}(E)\to \mathbb E^{p+1,q}(X)$ vanishes by \Cref{lem:surjectivity-on-exceptional-divisor}, and the long exact sequence breaks into short exact sequences
\[
0\to \mathbb E^{p,q}(X)\to \mathbb E^{p,q}(\Bl_Z X)\oplus \mathbb E^{p,q}(Z)\to \mathbb E^{p,q}(E)\to 0.
\]
Cancelling the direct summand $\mathbb E^{p,q}(Z)$ identified by $p^\ast$ yields
\[
\mathbb E^{p,q}(\Bl_Z X)\cong \mathbb E^{p,q}(X)\oplus \bigoplus_{k=1}^{r-1}\mathbb E^{p-2k,q-k}(Z).
\]
\end{proof}

\subsection{Multiplicative Blowup Presentation in $\mathsf{SH}(k)$}

Since the ``intersection theory package'' and the ``Gysin package'' have not been formally established for $\mathsf{MS}_S$ to the best of our knowledge, we return to the $\AA^1$-invariant subcategory $\mathsf{SH}(k)$ over an algebraically closed field $k$ of characteristic zero. Taking $p = 2q$ for an oriented motivic ring spectrum $\mathbb E$ that represents the oriented cohomology theory $A^\bullet$, we recover the additive decomposition for $A^\bullet(Bl_ZX)$ as a $A^\bullet(\pt)$-module.

\begin{theorem}[Additive Blowup Decomposition for Oriented Cohomology Theories]
\label{thm:blowup-formula-additive}
Let $A^\bullet$ be an oriented cohomology theory on $\Sm_k$ represented in $\mathsf{SH}(k)$. 
Let $\pi: Bl_ZX \to X$ be the blowup of a smooth scheme $X$ along a smooth center $Z$.
Then for every $m$, there is a direct sum decomposition
\[
A^m(\Bl_Z X)\cong A^m(X)\oplus \bigoplus_{k=1}^{r-1}A^{m-k}(Z)
\]
 as $A^\bullet(\pt)$-modules. \hfill \qed
\end{theorem}

We now give a presentation of the oriented cohomology ring $A^\bullet(Bl_ZX)$ in terms of generators and relations. Let $i:Z\hookrightarrow X$ be a smooth closed embedding of codimension $r\ge2$, $\pi:Y:=\Bl_Z X\to X$ the blowup, $j:E\hookrightarrow Y$ the exceptional divisor, and $p:E\to Z$ the projection. By \Cref{thm:blowup-formula-additive}, all elements of $A^\bullet(Bl_ZX)$ come from pullbacks $\pi^*\alpha$ for $\alpha\in A^\bullet(X)$ and pushforwards $j_*\gamma$ for $\gamma\in A^\bullet(E)$. The following lemmas contain the $4$ kinds of relations we need for the presentation, characterized by the origins of the elements in a product.

\begin{lemma}
    \label{lem:relation-i}
    For $\alpha, \beta\in A^\bullet(X)$, we have $\pi^*\alpha \cdot \pi^*\beta = \pi^*(\alpha\cdot \beta) $.
\end{lemma}
\begin{proof}
    The lemma follows from the fact that $\pi^*$ is a ring homomorphism.
\end{proof}

\begin{lemma}
    \label{lem:relation-ii}
    For $\alpha \in A^\bullet(X)$ and $\gamma\in A^\bullet(E)$, we have $\pi^*\alpha \cdot j_*\gamma = j_*(\gamma \cdot p^*i^*\alpha)$.
\end{lemma}

\begin{proof}
    By the projection formula, we have
\[\pi^*\alpha \cdot j_*\gamma = j_*(\gamma\cdot j^*\pi^*\alpha) = j_*(\gamma\cdot p^*i^*\alpha).\]
\end{proof}

\begin{lemma}
    \label{lem:relation-iii}
   For $\gamma, \delta\in A^\bullet(E)$, we have $j_*\gamma \cdot j_*\delta = j_*(\gamma\cdot \delta\cdot \chi_A(\zeta))$, where $\chi_A(\zeta)$ is the formal inverse of $\zeta$ under the group law $F_A$.
\end{lemma}
\begin{proof}
     By the projection formula, we have \[j_*\gamma \cdot j_*\delta = j_*(\gamma \cdot j^*j_*\delta) = j_*(\gamma\cdot \delta\cdot c_1^A(\mathcal{N}_{E/Bl_ZX})),\]
where $\mathcal{N}_{E/Bl_ZX} \cong \O_E(-1)$ is the normal bundle of $E$ inside $Bl_ZX$. By the self-intersection formula, we have\[c_1^A(\O_E(-1)) = \chi_A(c_1^A(\O_E(1))) = \chi_A(\zeta),\]
which concludes the proof.

\end{proof}

We observe that \Cref{lem:relation-ii} endows $A^\bullet(E)$ with the structure of an $A^\bullet(X)$-algebra by the action $\alpha .\gamma := \gamma \cdot p^*i^*\alpha $. Therefore, we can put a $A^\bullet(X)$-algebra structure on the direct sum $A^\bullet(X)\oplus A^\bullet(E)$ by \[(\alpha, \gamma) \cdot(\beta, \delta):=\left(\alpha \beta, \alpha .\delta+\beta.\gamma+\gamma \delta \cdot \chi_A\left(\zeta\right)\right).\]
The above three lemmas and the additive blowup decomposition (\Cref{thm:blowup-formula-additive}) imply that the map \[\Phi: A^\bullet(X)\oplus A^\bullet(E) \to A^\bullet(Bl_ZX), \quad (\alpha, \gamma) \mapsto \pi^*\alpha + j_*\gamma \]
is a surjective ring homomorphism. For the case of $K_0$, Fulton--Lang referred to this construction as the ``star product $\star$'' \cite{FultonLang1985RiemannRochAlgebra}. In the next theorem, we prove that the kernel of $\Phi$ is generated by precisely the relation coming from the excess intersection formula. 

\begin{lemma}[Excess Intersection Relation]
    \label{lem:blowup-presentation-implicit} Let $\mathcal{Q}:= p^* \mathcal{N}_{Z / X} / \mathcal{O}_E(-1)$ be the quotient bundle of rank $r-1$. The kernel of $\Phi$ is generated by the elements \[\rho(\theta):=\left( i_*(\theta),\; c_{r-1}(\mathcal{Q}) \cdot p^* \theta\right),\quad \text{ for all }\theta \in A^\bullet(Z).\]
\end{lemma}

\begin{proof}
   Let $R= A^\bullet(X)\oplus A^\bullet(E)$ with the above multiplication, and let $I= (\rho(\theta) \mid \theta \in A^\bullet(Z))\subseteq R$.  A direct application of the excess intersection formula on the blowup square implies that for every $\theta \in A^\bullet(Z)$, we have
$$
\pi^* i_*(\theta)=j_*\left(c_{r-1}^A(\mathcal{Q}) \cdot p^* \theta\right),
$$
so $\rho(\theta) \in \ker \Phi$ for all $\theta \in A^\bullet(Z)$. Therefore, we have $I \subseteq \ker \Phi$.

To prove that $\ker\Phi\subseteq I$, we take $(\alpha, \gamma) \in \ker\Phi$ and show that it is zero after reducing modulo the ideal $I$. To do this, we compute the Chern class $c_{r-1}^A(\Q)$. 
Applying the Whitney sum formula to the short exact sequence \[0\to \O_{E}(-1)\to p^*\N_{Z/X}\to \Q\to 0\]
yields an equality of total Chern polynomials \begin{equation}
        c_t^A(\Q) = \frac{c_t^A(p^*\N_{Z/X})}{c_t^A(\O_E(-1)) } =\frac{1+c_1^A(p^*\N_{Z/X})t + \cdots + c_r^A(p^*\N_{Z/X})t^r}{1 \;+\; \chi_A(\zeta)t}.
\end{equation}

\begin{quote}
    \textit{Claim.} $c_{r-1}^A(\Q)$ is a monic polynomial in $\zeta$ with top $\zeta$-degree term $\zeta^{r-1}$.
\end{quote}
\begin{quote}
    \textit{Proof of the claim.} Let $u := -\chi_A(\zeta)$. By expanding the geometric series $(1+\chi_A(\zeta)t)^{-1} = (1-u)^{-1}$, we have $c_{r-1}^A(\Q) = u^{r-1} + p^*c_1^A(\N_{Z/X}) u^{r-2} + \cdots + p^*c_{r-1}^A(\N_{Z/X})$. Hence, the coefficient in front of $u^{r-1}$ is $1$. Since the change-of-basis matrix of $A^\bullet(E)$ from $\{1,u, \dots, u^{r-1} \}$ to $\{1, \zeta, \dots, \zeta^{r-1}\}$ is upper triangular with diagonal entries all $1$, we conclude that $c_{r-1}^A(\Q)$ is also monic in $\zeta$. \hfill \qed
\end{quote}

By the projective bundle formula, we can expand any $\gamma\in A^\bullet(E)$ uniquely as  \[\gamma=\sum_{s=0}^{r-1} \zeta^s \cdot p^*\left(\gamma_s\right), \quad \text{ for }\gamma_s \in A^{\bullet-1-s}(Z).\]
Let $\operatorname{top}(\gamma) := \gamma_{r-1} \in A^{\bullet}(Z)$. Define the reduction operator \[\operatorname{red}: A^\bullet(E) \to A^\bullet(E), \quad \operatorname{red}(\gamma) := \gamma - p^*(\operatorname{top}(\gamma))\cdot c_{r-1}^A(\Q).\]
Equivalently, writing $\gamma=\sum_{s=0}^{r-1}\zeta^s p^*(\gamma_s)$ and $\theta:=\operatorname{top}(\gamma)=\gamma_{r-1}$, we have
\begin{equation}\label{eq:reduction-congruence}
(\alpha,\gamma)\ =\ \bigl(\alpha-i_*\theta,\ \operatorname{red}(\gamma)\bigr)\qquad \text{in }R/I,
\end{equation}
since $(i_*\theta,\ c_{r-1}^A(\Q)p^*\theta)=\rho(\theta)\in I$. Since $c_{r-1}^A(\Q)$ is monic in $\zeta$, the top degree $\zeta^{d-1}$-terms cancel, so $\operatorname{red}(\gamma)$ has $\zeta$-degree less than or equal to $d-2$.  Now for $(\alpha,\gamma)\in \ker\Phi$, we have
\[
0=\Phi(\alpha,\gamma)=\pi^*\alpha + j_*\gamma \in A^\bullet(\Bl_ZX).
\]
Let $\theta:=\operatorname{top}(\gamma)$ and replace $(\alpha,\gamma)$ by the congruent element
$(\alpha-i_*\theta,\ \operatorname{red}(\gamma))$ mod $I$ as in \eqref{eq:reduction-congruence}. Since $I\subseteq \ker\Phi$, the map $\Phi$ factors through $R/I$, so we still have
\[
0=\Phi(\alpha-i_*\theta,\ \operatorname{red}(\gamma))=\pi^*(\alpha-i_*\theta)+j_*\operatorname{red}(\gamma).
\]
By construction, $\operatorname{red}(\gamma)$ has $\zeta$-degree $\le r-2$, so by the projective bundle formula, we may write uniquely
\[
\operatorname{red}(\gamma)=\sum_{s=0}^{r-2}\zeta^s\cdot p^*(\delta_s),
\qquad \delta_s\in A^{\bullet-1-s}(Z).
\]
Therefore, we have 
\begin{equation}\label{eq:decomposition-of-kernel-element}
    0=\pi^*(\alpha-i_*\theta)+\sum_{s=0}^{r-2} j_*\bigl(\zeta^s p^*\delta_s\bigr).
\end{equation}
Since the additive blowup decomposition
\[
A^\bullet(\Bl_ZX)\ \cong\ \pi^*A^\bullet(X)\ \oplus\ \bigoplus_{s=0}^{r-2} j_*\bigl(\zeta^s p^*A^{\bullet-1-s}(Z)\bigr)
\]
is a direct sum decomposition, every element of $\ABl$ has a \textit{unique} expression in the form of \eqref{eq:decomposition-of-kernel-element}. Hence, $\alpha-i_*\theta=0$ and $ \delta_s=0$ for all $0\le s\le r-2$. In particular, $\operatorname{red}(\gamma)=0$ and $\alpha=i_*\theta$. Returning to \eqref{eq:reduction-congruence}, we conclude that
\[
(\alpha,\gamma)= (0,0)\ \text{ in }R/I,
\]
so $(\alpha,\gamma)\in I$. This proves $\ker\Phi\subseteq I$, and therefore $\ker\Phi=I$.

\end{proof}

The above \Cref{lem:relation-i}, \Cref{lem:relation-ii}, \Cref{lem:relation-iii}, and \Cref{lem:blowup-presentation-implicit} altogether imply the main structural theorem of this paper.

\begin{theorem}[Blowup Presentation for Oriented Cohomology Rings]\label{thm:blowup-presentation}
Let $i:Z\hookrightarrow X$ be a smooth closed embedding of codimension $r\ge2$, $\pi:Y:=\Bl_Z X\to X$ the blowup, $j:E\hookrightarrow Y$ the exceptional divisor, and $p:E\to Z$ the projection. Let $\zeta:=c_1^A(\cO_E(1))\in A^1(E)$ be the relative hyperplane class of the projective bundle $E$ in $A^\bullet$-theory. The ring $A^\bullet(\Bl_ZX)$ is generated over $A^\bullet(\pt)$ by pullbacks $\pi^*\alpha$ for $\alpha\in A^\bullet(X)$ and pushforwards $j_*\gamma$ for $\gamma\in A^\bullet(E)$ subject to the following families of relations.

\begin{enumerate}[(i)]
    \item (\textit{Intersection on $X$}) For $\alpha, \beta\in A^\bullet(X)$, $\pi^*\alpha \cdot \pi^*\beta = \pi^*(\alpha\cdot \beta) $.
    \item (\textit{Intersection between $X $ and $E$}) For $\alpha \in A^\bullet(X)$ and $\gamma\in A^\bullet(E)$, $\pi^*\alpha \cdot j_*\gamma = j_*(\gamma \cdot p^*i^*\alpha)$.
    \item (\textit{Intersection on $E$}) For $\gamma, \delta\in A^\bullet(E)$, $j_*\gamma \cdot j_*\delta = j_*(\gamma\cdot \delta\cdot \chi_A(\zeta))$, where $\chi_A(\zeta)$ is the formal inverse of $\zeta$ under the group law $F_A$. 
    \item (\textit{Excess intersection formula}) For every $\theta \in A^\bullet(Z)$, we have
$$
\pi^* i_*(\theta)=j_*\left(c_{r-1}^A(\mathcal{Q}) \cdot p^* \theta\right), 
$$ where $\mathcal{Q}:= p^* \mathcal{N}_{Z / X} / \mathcal{O}_E(-1).$ \hfill \qed
\end{enumerate}
\end{theorem}

In addition, if we assume that the pullback $i^*\colon A^\bullet(X)\to A^\bullet(Z)$ is surjective, then the same argument of Keel \cite[Appendix Theorem A]{Keel} applies verbatim and yields a more compact ring presentation for $A^\bullet(Bl_ZX)$. The first polynomial relation comes from a combination of the projective bundle formula and the excess intersection formula. 

\begin{theorem}[Blowup presentation - surjective version]\label{thm:blowup-presentation-surjective}
Assume in addition that the pullback $i^*\colon A^\bullet(X)\to A^\bullet(Z)$ is surjective. Let $\zeta\in A^1(Bl_ZX)$ denote the first Chern class $c_1^A(\mathcal O_{\widetilde X}(E))$.
Choose lifts $\widetilde c_k\in A^k(X)$ for $1\le k\le r-1$ such that $i^*(\widetilde c_k)=c_k^A(\mathcal N_{Z/X})$ in $A^\bullet(Z)$.
Then the canonical map $A^\bullet(X)[\zeta]\to A^\bullet(Bl_Z X)$ given by $\alpha\mapsto \pi^*\alpha$ and $\zeta\mapsto c_1^A(\mathcal O_{Bl_Z X}(E))$
induces an isomorphism
\[
A^\bullet(Bl_ZX)\;\cong\;
\frac{A^\bullet(X)[\zeta]}{\Bigl(\zeta^r-\widetilde c_1\zeta^{r-1}+\cdots+ (-1)^{r-1}\widetilde c_{r-1}\zeta + (-1)^ri_*1,\ \zeta\cdot\ker(i^*)\Bigr)}.
\]
The resulting quotient is independent of the choices of the lifts $\widetilde c_k$. \hfill \qed
\end{theorem}

Finally, the following lemma is a useful sufficient condition that implies the surjectivity of the pullback homomorphism.

\begin{lemma}\label{lem:section-implies-surjective}
    If the inclusion $i: Z \hookrightarrow X$ admits a section, in the sense that there exists a morphism $r: X \to Z$ with $r \circ i=\operatorname{id}_Z$,
then the pullback map
$
i^*: A^{\bullet}(X) \to A^{\bullet}(Z)
$
is surjective.
\end{lemma}
\begin{proof}
  The identity $r \circ i=\operatorname{id}_Z$ implies
$
i^* \circ r^*=(r \circ i)^*=\operatorname{id}_{A^{\bullet}(Z)} .
$ Hence, $i^*$ is surjective since it has a right inverse
\end{proof}

\section{First Examples}\label{sec:first-examples}

In this section, we provide some applications of the above Blowup Presentation formulae. We start by considering del Pezzo surfaces, where the surjective Blowup Presentation is applied. Then, we move on to the blowup of $\PP^3$ along the twisted cubic $C$, which prompts the use of the full Blowup Presentation. Furthermore, we demonstrate how to recover the count of secant lines to $C$ simultaneously meeting two general lines in $\PP^3$, using an arbitrary oriented cohomology theory. Throughout the rest of the paper, let $k $ be an algebraically closed field of characteristic zero. 

\subsection{Del Pezzo Surfaces}

 Del Pezzo surfaces are blowups of the projective plane $\PP^2$ or $\PP^1\times \PP^1$ at $n$ general points for $0\leq n\leq 8$. Since the pullback $A^\bullet(\PP^2) \cong A^\bullet(\pt)[\zeta] / (\zeta^3) \to A^\bullet(\pt)$ is always surjective, given by $\zeta\mapsto 0$, we are in a situation to apply the surjective Blowup Presentation (\Cref{thm:blowup-presentation-surjective}). 

\begin{proposition}[Oriented cohomology rings of del Pezzo surfaces]\label{prop:del-pezzo-oriented-rings}
Let $A^\bullet$ be an oriented cohomology theory.
Let $h$ denote $c_1^A(\mathcal O_{\mathbb P^2}(1))\in A^1(\mathbb P^2)$.

\begin{enumerate}
\item The degree $9$ del Pezzo surface $dP_9=\mathbb P^2$ has
\[
A^\bullet(dP_9)\cong \frac{A^\bullet(\pt)[h]}{(h^3)}.
\]

\item The degree $8$ del Pezzo surfaces $Bl_p\PP^2$ and $\PP^1\times \PP^1$ both have
\[
A^\bullet(\PP^1\times \PP^1)\cong A^\bullet(Bl_p\PP^2)\cong \frac{A^\bullet(\pt)[u,v]}{(u^2,\ v^2)}.
\]

\item Let $2\le k\le 8$ and let $dP_{9-k}$ be the blowup of $\mathbb P^2$ at $k$ general closed points.
Let $h\in A^1(dP_{9-k})$ denote the pullback of $c_1^A(\mathcal O_{\mathbb P^2}(1))$, and let $e_1,\dots,e_k\in A^1(dP_{9-k})$ denote the classes $c_1^A(\mathcal O(E_i))$ of the exceptional divisors.
Then
\[
A^\bullet(dP_{9-k})
\cong
\frac{A^\bullet(\pt)\bigl[h,e_1,\dots,e_k\bigr]}
{\Bigl(h^3,\ e_i h,\ e_i e_j\ (i\ne j),\ e_i^2+h^2\ \text{for }1\le i,j\le k\Bigr)}.
\]
\end{enumerate}
\end{proposition}

\begin{proof}
The presentations for $\mathbb{P}^2$ and $\PP^1\times \PP^1$ follow from the projective bundle formula \cref{thm:projective-bundle-formula}. For the blowups, we argue by induction on $k$. The case $k=0$ is $\mathbb P^2$, already done. Assume $k\ge 1$ and write $dP_{9-k}=\Bl_{p_k}(dP_{10-k})$ for the blowup of $dP_{10-k}$ at a closed point $p_k$, with exceptional divisor $E_k$.
Since the inclusion $p_k\to dP_{10-k}$ has a section given by the structure morphism $dP_{10-k}\to p_k$, the pullback
\(
i^*\colon A^\bullet(dP_{10-k})\to A^\bullet(p_k)
\)
is surjective by \Cref{lem:section-implies-surjective}. The surjective Blowup Presentation therefore applies with $r=2$ and $i:p_k \to dP_{10-k}$. Hence, $A^\bullet(dP_{9-k})$ is obtained from $A^\bullet(dP_{10-k})$ by adjoining the extra generator $e_k=c_1^A(\mathcal O(E_k))$ and relations $e_k^2+ (-1)^2i_*1_{p_k}=0$ and $e_k\cdot \ker(i^*)=0$. Since $i_*1_{p_k}$ is the fundamental class of a point in $dP_{10-k}$, it is equal to $h^2$, yielding $e_k^2+h^2=0$. The second relation implies $e_k\cdot \alpha=0$ for every $\alpha\in A^{>0}(dP_{10-k})$. In particular, $e_kh = 0$, and $e_je_k = 0$ for $j<k$. The result then follows from induction.
\end{proof}

\begin{remark}\label{rmk: delpezzo exceptional}
    The above proposition shows that any oriented cohomology ring $A^\bullet(dP_k)$ has an \textit{exceptional isomorphism}\footnote{The term \textit{exceptional isomorphism} is used in the sense of Larson--Li--Payne--Proudfoot \cite{LLPP2024WonderfulVarieties}, where they showed such isomorphisms from the $K_0$-ring to the Chow ring over the integers.} to the Chow ring, in the sense that \[A^\bullet(dP_k) \cong CH^\bullet(dP_k)\otimes_{\ZZ}A^\bullet(\pt).\]
\end{remark}

\subsection{$\PP^3$ Blown up at the Twisted Cubic}

Next, we consider the blowup of $\PP^3$ along the twisted cubic, which prompts the application of the full Blowup Presentation (\Cref{thm:blowup-presentation}). We also illustrate through this example that one can obtain familiar intersection numbers in enumerative geometry from a general oriented cohomology theory. See \Cref{apx:4-secant-lines} for an exposition to a classical solution to the enumerative problem in question. Much of the geometric setup that follows can also be found in \cite{EisenbudHulekPopescu2003IntersectionVeronese, SidmanVermeire2011SecantVarieties, Peskine2015LineCongruences}.

We demonstrate how to determine the oriented cohomology rings of this blowup. Throughout this section, let $A^\bullet$ be an oriented cohomology theory characterized by the formal group law $F(x,y) = x+y + a_{11}xy + \text{ higher order terms}$. Let $\nu_3:\PP^1 \to \PP^3$ be the cubic Veronese map, which embeds $\PP^1$ as the twisted cubic $C\subseteq\PP^3$.  Let 
\[
\begin{tikzcd}
E \ar[r,"j"] \ar[d,"p"'] & \Bl_C \PP^3 \ar[d,"\pi"] \\
C \ar[r,"{\nu_3}"'] & \PP^3.
\end{tikzcd}
\]
be the relevant blowup square. We begin by studying surjectivity of the pullback $A^\bullet(\PP^3) \to A^\bullet(C)$. By the projective bundle formula, we have \[A^\bullet(\PP^3) \cong \frac{A^\bullet(\pt)[\alpha]}{\alpha^4}, \quad \text{ and } \quad A^\bullet(C) \cong \frac{A^\bullet(\pt)[\eta]}{\eta^2},\]
where $\alpha := c_1^A(\O_{\PP^3}(1))$ and $\eta := c_1^A(\O_{C}(1))$ are the $A^\bullet$-theoretic hyperplane classes. By definition of $\nu_3$, the induced ring map $\nu_3^*$ satisfies 
\[\nu_3^*(\alpha) = c_1^A(\O_C(3)) = [3]_F\eta = 3\eta,\]
where the vanishing of all higher order terms is due to the relation $\eta^2 = 0$. We see that $\nu_3^*$ is \textit{not} surjective if $A^\bullet$ has integer coefficients, as $\eta, 2\eta$ are not in its image. Therefore, we are in a situation to apply the full version of the Blowup Formula (\Cref{thm:blowup-presentation}). We break down the application of \Cref{thm:blowup-presentation} on the above square into the following three general steps. First, we compute the normal bundle $\N_{C/\PP^3}$. Secondly, we use the Projective Bundle Formula to give a presentation to $A^\bullet(E) = A^\bullet(\PP(\N_{C/\PP^3}))$. We work out powers and the formal inverse of $\zeta = c_1^A(\O_E(1))$. Finally, we present the oriented cohomology ring $A^\bullet(Bl_C\PP^3)$ with generators consisting of pullbacks from $\PP^3$ and pushforwards from $E$, subject to relations (i) - (iv).

The following proposition is well-known.

\begin{proposition}[Normal bundle of the twisted cubic]
Let
\(
\nu_3 \colon \mathbb{P}^1 \hookrightarrow \mathbb{P}^3
\)
be the cubic Veronese embedding, equivalently the twisted cubic $C = \nu_3(\PP^1)$. Then,
\(
\N_{C/\mathbb{P}^3} \cong \mathcal{O}_{\mathbb{P}^1}(5) \oplus \mathcal{O}_{\mathbb{P}^1}(5).
\)
\end{proposition}

\begin{proof}
We work inside the Chow ring. Let \(h=c_1(\mathcal{O}_{\mathbb{P}^1}(1))\). Since \(\nu_3^*\mathcal{O}_{\mathbb{P}^3}(1)\cong \mathcal{O}_{\mathbb{P}^1}(3)\), pulling back the Euler sequence on \(\mathbb{P}^3\) yields
\(
0 \to \mathcal{O}_{\mathbb{P}^1}
\to \mathcal{O}_{\mathbb{P}^1}(3)^{\oplus 4}
\to \nu_3^*\mathcal T_{\mathbb P^3}
\to 0.
\)
Therefore,
\(
c\bigl(\nu_3^*\mathcal T_{\mathbb P^3}\bigr)
=
(1+3h)^4.
\)

On the other hand, the normal bundle sequence
\(
0 \to T_{\mathbb{P}^1}
\to \nu_3^*\mathcal T_{\mathbb P^3}
\to \N_{C/\PP^3}
\to 0
\)
yields
\(
c_1(\N_{C/\PP^3})
=
10h
\) and $c_2(\N_{C/\PP^3}) = 0$. By Grothendieck splitting, there is an integer \(a\) such that
\(
\mathcal{N}_{C/\mathbb{P}^3}\cong
\mathcal{O}_{\mathbb{P}^1}(a)\oplus \mathcal{O}_{\mathbb{P}^1}(10-a).
\)
We claim that \(a=5\). Indeed, twisting the pullback of the Euler sequence from \(\mathbb{P}^3\) to $\PP^1$ by \(\mathcal{O}_{\mathbb{P}^1}(-6)\) yields $$0 \rightarrow \mathcal{O}_{\mathbb{P}^1}(-6) \rightarrow \mathcal{O}_{\mathbb{P}^1}(-3)^{\oplus 4} \rightarrow \nu_3^* T_{\mathbb{P}^3}(-6) \rightarrow 0.$$ Using the surjectivity of the multiplication map
\(
H^0\bigl(\mathcal{O}_{\mathbb{P}^1}(1)\bigr)^{\oplus 4}\to H^0\bigl(\mathcal{O}_{\mathbb{P}^1}(4)\bigr),
\)
one obtains \(H^0(C, \nu_3^*\mathcal T_{\mathbb{P}^3}(-6))=0\). Twisting the normal bundle sequence by \(\mathcal{O}_{\mathbb{P}^1}(-6)\) then gives
\(
H^0\bigl(C, \mathcal{N}_{C/\mathbb{P}^3}(-6)\bigr)=0.
\)
If \(a\leq 4\), then \(10-a\geq 6\), so
\(
H^0\bigl(C, \mathcal{N}_{C/\mathbb{P}^3}(-6)\bigr)
\neq 0,
\)
which is a contradiction, so \(a\geq 5\). Therefore, \(a=5\) by the symmetry of this argument. Hence,
\(
\mathcal{N}_{C/\mathbb{P}^3}\cong
\mathcal{O}_{\mathbb{P}^1}(5)\oplus \mathcal{O}_{\mathbb{P}^1}(5).
\)
\end{proof}
We use the description of the normal bundle of $C$ in $\mathbb P^3$ above, to give the oriented cohomology ring presentation of the exceptional divisor in $\Bl_C\mathbb P^3$.
\begin{proposition}[Oriented cohomology ring of the exceptional divisor]\label{prop:twisted-cubic-PBF}
   Let $E\xrightarrow{j} Bl_C\PP^3$ be the exceptional divisor. Let $\zeta := c_1^A(\O_E(1))$ be the relative hyperplane class, and with slight abuse of notation, let $\eta := p^*c_1^A(\O_{C}(1))$ be the pullback of the hyperplane class of $C$. The oriented cohomology ring of $E$ has the quotient presentation 
    \[A^\bullet(E) \cong \frac{A^\bullet(\pt)[\eta, \zeta]}{\zeta^2+ 10\eta\zeta}.\]
\end{proposition}

\begin{proof}
    By the Whitney sum formula, \[c^A(\mathcal{N}_{C/\mathbb{P}^3}) = c^A(\mathcal{O}_{\mathbb{P}^1}(5))\cdot c^A(\mathcal{O}_{\mathbb{P}^1}(5)) = (1 + [5]_F\eta)^2.\]
Since $\eta^2 = 0$, all higher order terms in the formal group law vanish, and the above expression simplifies as $(1+5\eta)^2 = 1 + 10\eta$. In particular, $c_1^A(\mathcal{N}_{C/\mathbb{P}^3}) = 10\eta$ and $c_2^A(\mathcal{N}_{C/\mathbb{P}^3}) = 0$. The projective bundle relation then gives the desired presentation.
\end{proof}

Notice that $\eta^2 = 0$ makes this entire calculation collapse to being the same as the Chow ring. In particular, we obtain the additive blowup decomposition as $A^\bullet(\pt)$-modules, \[A^\bullet(E) \cong A^\bullet(\pt)\cdot \{1, \eta, \zeta, \eta\zeta\}.\]
\begin{proposition} \label{prop:twisted-cubic-formal-inverse}
   Let $\zeta$ be as above. The relation $\zeta^3 = 0 $ holds in $A^\bullet(E)$. The formal inverse of $\zeta$ is \[\chi_A(\zeta)=-\zeta-10 a_{11} \eta \zeta \quad  \in A^1(E).\]
\end{proposition}
\begin{proof}
    The projective bundle relation implies $\zeta^3 = \zeta \cdot \zeta^2 = \zeta (-10\eta\zeta) = -10\eta(-10\eta\zeta) = -100\zeta\cdot \eta^2 = 0$. Therefore, the formal inverse simplifies to $\chi_A(\zeta) = -\zeta + a_{11}\zeta^2 + 0 = -\zeta - 10a_{11}\eta\zeta$. 
\end{proof}
We now compute the fundamental class of the twisted cubic which will be needed for the excess intersection formula.
\begin{proposition}[Fundamental class of the twisted cubic]\label{prop:fundamental-class-of-twisted-cubic}
  Let $\alpha = c_1^A(\O_{\PP^3}(1)) \in A^1(\PP^3)$ be the fundamental class of a hyperplane.  The $A$-theoretic fundamental class of the twisted cubic in $\PP^3$ is $[C]_A = 3\alpha^2 + 2a_{11}\alpha^3$.
\end{proposition}

\begin{proof}
    This is an application of Quillen's residue formula (see \Cref{cor:projective-space-residue}). Let $[C]_A = i_*1_C$ be the fundamental class of $C$. Since $[C]_A \in A^2(\PP^3)$ and $\alpha^4 = 0$, it expands uniquely as $[C]_A=b_2\alpha^2+b_3\alpha^3$. Let $\pi_3\colon\PP^3\to\Spec k$ and $\pi_C\colon C\to\Spec k$ be the structural morphisms. By the projection formula and the functoriality $\pi_{C,*} = \pi_{3,*}\circ i_*$, we have
\(
\pi_{3,\ast}([C]_A\alpha)=\pi_{3,\ast}(i_*1_C\cdot \alpha) = \pi_{C,\ast}(i^\ast\alpha).
\)
Because $\alpha^4=0$, we have $[C]_A \cdot \alpha=b_2\alpha^3$, so
\(
b_2=\pi_{3\ast}(b_2\alpha^3)=\pi_{3,\ast}([C]_A\alpha)=\pi_{C,\ast}(i^\ast\alpha)=\pi_{C,*}([3]_{F_A}\eta).
\)
Since $\eta^2 = 0$, we see that $[3]_{F_A}$ only has the linear part, so $[3]_{F_A}\eta = 3\eta$. Using Corollary~\ref{cor:projective-space-residue} for $\PP^1$ gives $\pi_{C\ast}(\eta)=1$. Hence, $b_2 = 3$. Again by the residue formula, we have \(
\pi_{3\ast}([C]_A)=b_2\,\pi_{3\ast}(\alpha^2)+b_3\,\pi_{3\ast}(\alpha^3)=b_2\omega_1+b_3,
\) while $\pi_{3\ast}([C]_A)=\pi_{C\ast}(1)=\omega_1$. Hence, $b_3=(1-b_2)\omega_1=-2\omega_1$.
\end{proof}

\begin{remark}
    Specializing to $K_0$ by setting $a_{11}\mapsto -1$, we recover the fundamental class $[\O_C] = 3\alpha^2 - 2\alpha^3$ for $\alpha:= 1-[\O_{\PP^3}(-1)]$, manifested by the resolution of the structure sheaf \[0 \to \mathcal{O}_{\PP^3}(-3)^{\oplus 2} \to \mathcal{O}_{\PP^3}(-2)^{\oplus 3} \to \O_{\PP^3} \to \O_C\to 0 .\]
\end{remark}

\begin{theorem}[Oriented cohomology ring of $Bl_C\PP^3$]\label{thm:oriented-coh-of-BLCP3}
Let $A^\bullet$ be an oriented cohomology theory. Let $C\subset \mathbb P^3$ be the twisted cubic and $i=v_3$ be the inclusion map, let
\(
\pi\colon \Bl_{C}\mathbb P^3 \to \mathbb P^3
\)
be the blowup, let $j\colon E\hookrightarrow \Bl_{C}\mathbb P^3$ be the exceptional divisor, and let $p\colon E\to C$ be the projection. With slight abuse of notation, let
\(
\alpha:=\pi^*c_1^A\!\big(\mathcal O_{\mathbb P^3}(1)\big)\in A^1(\Bl_{C}\mathbb P^3),\;
e:=j_*(1)\in A^1(\Bl_{C}\mathbb P^3),
\)
\(
x:=j_*\!\big(p^*c_1^A(\mathcal O_{C}(1))\big)\in A^2(\Bl_{C}\mathbb P^3),\) and \(
z:=j_*\!\big(\zeta\big)\in A^2(\Bl_{C}\mathbb P^3),
\)
where $\zeta:=c_1^A(\mathcal O_E(1))\in A^1(E)$. 
Then, there is an isomorphism of graded $A^\bullet(\pt)$-algebras
\[
A^\bullet(\Bl_{C}\mathbb P^3)\ \cong\
\frac{A^\bullet(\pt)[\alpha,e,x,z]}{
\begin{aligned}
&(\alpha^4,\ \alpha^2e,\ \alpha e-3x,\ \alpha x,\ \alpha z-3\alpha^3,\ z^2,\ x^2, \ xz,\\
&\ e^2+z+10a_{11}\alpha^3,\ ex+\alpha^3,\ ez-10\alpha^3,\ 3\alpha^2-z-10x-8a_{11}\alpha^3).
\end{aligned}}
\]
\end{theorem}

\begin{proof}
The ring $A^\bullet(\Bl_{C}\mathbb P^3)$ is generated over
$A^\bullet(\pt)$ by the pullback $\alpha$ and pushforwards
$j_*\gamma$ for all $\gamma\in \{1, \eta, \zeta, \eta\zeta\} \subseteq A^\bullet(E)$, subject to the four families of relations (i)--(iv) in \cref{thm:blowup-presentation}.

\begin{enumerate}[(1)]
\item (\emph{Intersection on $\mathbb P^3$}). Relation \emph{(i)} straightforwardly yields $\alpha^4=0$.

\item (\emph{Intersection between $\mathbb P^3$ and $E$}).
Relation \emph{(ii)} gives $\alpha^m\cdot j_*\gamma=j_*\big(\gamma\cdot p^*i^*(\alpha^m)\big)$ for all $m\ge 0$ and $\gamma\in A^\bullet(E)$.
Since $C$ is embedded via $\mathcal O_{\mathbb P^1}(3)$, we have $i^*\alpha=c_1^A\!\big(\mathcal O_{\mathbb P^1}(3)\big)=[3]_F(\eta)=3\eta$, because $\eta^2=0$ on $C$.
Writing $e:=j_*(1)$, $x:=j_*(\eta)$, and $z:=j_*(\zeta)$, we obtain $\alpha e=j_*\big(p^*i^*\alpha\big)=j_*(3\eta)=3x$ and $\alpha^2e=j_*\big(p^*i^*(\alpha^2)\big)=j_*\big((3\eta)^2\big)=0$; similarly $\alpha x=j_*(\eta\cdot 3\eta)=0$.
Moreover, $\alpha z=j_*\big(\zeta\cdot p^*i^*\alpha\big)=j_*(3\eta\zeta)$, and relation \emph{4.} below identifies $j_*(\eta\zeta)=\alpha^3$, hence $\alpha z=3\alpha^3$.

\item (\emph{Intersection on $E$}).
Relation \emph{(iii)} gives $j_*\gamma\cdot j_*\delta=j_*\big(\gamma\delta\chi_A(\zeta)\big)$ for $\gamma,\delta\in A^\bullet(E).$
Using \Cref{prop:twisted-cubic-PBF} and \Cref{prop:twisted-cubic-formal-inverse}, and exhausting all choices of $\gamma,\delta\in \{1,\eta,\zeta,\eta\zeta\}$, we obtain $e^2=-z-10a_{11}\alpha^3$, $ex=-\alpha^3$, $ez=10\alpha^3$, and $x^2=z^2=xz=0$.

\item (\emph{Excess intersection}).
Since $\codim(C,\mathbb P^3)=2$, the bundle $\mathcal Q:=p^*\mathcal N_{C/\mathbb P^3}/\mathcal O_E(-1)$ is a line bundle, and relation \emph{(iv)} reads $\pi^*i_*(\theta)=j_*\!\big(c_1^A(\mathcal Q)\cdot p^*\theta\big)$ for all $\theta\in A^\bullet(C)$.
Using $\mathcal N_{C/\mathbb P^3}\cong \mathcal O_{\mathbb P^1}(5)\oplus \mathcal O_{\mathbb P^1}(5)$ and $\eta^2=0$, we have $c_1^A(\mathcal N_{C/\mathbb P^3})=10\eta$ and $c_2^A(\mathcal N_{C/\mathbb P^3})=0$. Hence, $c_1^A(\mathcal Q)=\zeta+10\eta+10a_{11}\eta\zeta$ in $A^\bullet(E)$.
Since $A^\bullet(C)\cong A^\bullet(\pt)\cdot\{1_C, \eta\}$ additively, it remains to determine $i_*(1_C)$ and $i_*(\eta)$. By \Cref{prop:fundamental-class-of-twisted-cubic}, $i_*1 = 3\alpha^2 + 2a_{11}\alpha^3$. Since $\eta$ pushes forward to the fundamental class of a point, we have $i_*(\eta) = \alpha^3$. Substituting this and $\pi^*i_*(1_C) = 3\alpha^2 + 2a_{11}\alpha^3$ into the $\theta=1$ equation above yields
\(
3\alpha^2 + 2a_{11}\alpha^3 = z+10x+10a_{11}\alpha^3,
\)
which simplifies to $3\alpha^2 = z+10x+8a_{11}\alpha^3$.
For $\theta=\eta$, the excess intersection formula gives $\pi^*i_*(\eta)=j_*\bigl((\zeta+10\eta+10a_{11}\eta\zeta)\eta\bigr)$. Since $\eta^2=0$, the right-hand side simplifies to $j_*(\eta\zeta)$. Substituting $\pi^*i_*(\eta) = \alpha^3$, we obtain the identity $j_*(\eta\zeta) = \alpha^3$.

\end{enumerate}

Collecting the generators $\alpha,e,x,z$ and the relations obtained from \emph{1-4}, produces the stated quotient presentation of $A^\bullet(\Bl_{C}\mathbb P^3)$ over $A^\bullet(\pt)$.
\end{proof}

\begin{remark}
Although \Cref{thm:oriented-coh-of-BLCP3} utilizes codimension-$2$ generators $x$ and $z$, the oriented cohomology ring $A^\bullet(\Bl_C\PP^3)$ is actually generated entirely by the divisor classes $\alpha$ and $e$ over $A^\bullet(\pt)$. The relations allow us to explicitly write $z$ and $x$ as follows:
\begin{align*}
    z &= -e^2 - 10a_{11}\alpha^3, \\
    x &= 10x - 3(3x) = 3\alpha^2 + e^2 - 3\alpha e + 2a_{11}\alpha^3.
\end{align*}
Thus, all generators are polynomials in the divisors $\alpha$ and $e$.
\end{remark}

We note that since Chow theory and $K$-theory are Landweber exact (\Cref{eg:landweber-exact-theories}), tensoring with $\ZZ$ by sending $a_{11}$ to $0$ in the above Blowup Presentation recovers the Chow ring $CH^\bullet(Bl_C\PP^3)$. Tensoring with $\ZZ[\beta, \beta^{-1}]$ by setting $a_{11}\mapsto -\beta$ recovers $K_0(Bl_C\PP^3)[\beta, \beta^{-1}]$. In particular, specializing to $\beta \mapsto 1$ yields a Blowup Presentation for the $K_0$-ring. 

\paragraph{Counting secant lines.} Let $\operatorname{Sect}(C) \cong \PP^2$ be the secant variety of $C\subseteq \PP^3$. It is well-known that $Bl_C\PP^3 \to \operatorname{Sect}(C)$ is the universal secant line to $C$. We can use this moduli interpretation of $Bl_C\PP^3$ to solve the following enumerative problem. 

\begin{quote}
   \begin{itemize}
       \item[$\dagger$]  How many secant lines to $C$ meet two general lines in $\PP^3$?
   \end{itemize}
\end{quote}
A proof using classical Schubert calculus on the Grassmannian $Gr(2,4)$ can be found in \Cref{apx:4-secant-lines}. Here, we detail the proof using an arbitrary oriented cohomology ring.

\begin{theorem}\label{thm:4-secant-lines}
Let $C\subset \PP^3$ be a twisted cubic and let $L_1,L_2\subset \PP^3$ be two general lines. Then there are exactly $4$ secant lines to $C$ meeting both $L_1$ and $L_2$. 
\end{theorem}

\begin{proof}
    The morphism $f:Bl_C\PP^3 \to \operatorname{Sect}(C) \cong \PP^2$ is the resolution of the base locus of the linear system $|\I_C(2)|$. If $H$ is the pullback of a hyperplane from $\PP^3$ \footnote{By this, we mean the effective divisor $1H$ on $Bl_C\PP^3$, not to be confused with the \textit{hyperplane class} $\alpha = [H]_A$ in the oriented cohomology ring.}, then the pullback of $\O_{\PP^2}(1)$ satisfies \begin{equation}
    f^*\O_{\PP^2}(1) \cong \O_{Bl_C\PP^3}(2H-E).
\end{equation}
We need to find the locus $D_{L_1}$ in $Bl_C\PP^3$ corresponding to secant lines meeting a given line $L_1\subseteq \PP^3$. Because $f$ is the universal secant curve, $S_{L_1} = f(L_1)\subseteq \operatorname{Sect}(C)\cong \PP^2$ is the locus of secant lines to $C$ meeting $L_1$. The restriction of the linear system $|\I_C(2)|$ to $L_1$ is the linear system $|\O_{L_1}(2)|$, so $S_{L_1} $ is a plane conic. The locus $D_{L_1}$ then is the pullback $f^*S_{L_1}$, which is a divisor on $Bl_C\PP^3$. It corresponds to the line bundle \begin{equation}
    f^*\O_{\PP^2}(2) \cong \O_{Bl_C\PP^3}(4H-2E).
\end{equation} 
Take two general lines $L_1, L_2\subseteq \PP^3$. Then, the intersection $D_{L_1} \cap D_{L_2}$ corresponds to the universal family over the locus $S_{L_1}\cap S_{L_2}\subseteq \operatorname{Sect}(C)$ consisting of the $4$ secant lines meeting $L_1$ and $L_2$. To recover the secant lines themselves, we need to further intersect with the general pulled-back hyperplane $H$. Therefore, the number of secant lines to $C$ meeting two general lines in $\PP^3$ can be recovered by the fundamental class \[[D_{L_1} \cap D_{L_2} \cap H]_A = [D_{L_1}]_A \cdot [D_{L_2}]_A \cdot [H]_A= ([4]_F \alpha \; -_F\; [2]_F e)^2 \cdot \alpha \quad \in A^3(Bl_C\PP^3).\]
A \texttt{Macaulay2} calculation (\Cref{apx:4-secant-lines-m2}) shows that this class is indeed $4\alpha^3 = 4[\pt]_A$. Therefore, we recover the classical count $4$ by doing intersection theory in an arbitrary oriented cohomology ring of $Bl_C\PP^3$. 
\end{proof}

\section{Oriented Cohomology Ring of the Moduli Space of Complete Conics}\label{sec:oriented-cohomology-of-complete-conics}

In this section, we compute the oriented cohomology ring of the moduli space of stable maps $\Kont$, which is known to be isomorphic to the moduli space of complete conics $\Bl_V\mathbb P^5$, where $V \subset \mathbb P^5$ is the Veronese surface. The moduli space of stable maps is a generalization of the moduli space of stable curves. We refer the reader to \Cref{apx:moduli-of-complete-conics} for a brief exposition to the moduli stack of stable maps $\overline{\M}_{g,n}(X,\beta)$, its coarse moduli space $\overline{M}_{g,n}(X,\beta)$, and in particular $\Kont$. Using the Blowup Presentation of this oriented cohomology ring, we demonstrate how to recover the number $3264$ of conics simultaneously tangent to five general conics in $\PP^2$, using algebraic cobordism. Throughout this section, let $A^\bullet$ be an oriented cohomology theory with formal group law $x+_Fy=x+y+\sum_{i, j \geq 1} a_{i j} x^i y^j$.

\subsection{The Normal Bundle of the Veronese Surface} 

We start by recalling the isomorphism between the moduli space of complete conics and the moduli space of stable maps $\Kont$. 
\begin{theorem}[ \cite{Kock2006}]
    There exists an isomorphism $\Kont \cong Bl_{V}\mathbb P^5$, where $\mathbb P^2 \cong V\hookrightarrow \mathbb P^5$ via the degree $2$ Veronese embedding.
\end{theorem}

Accordingly, we consider the following blowup square.

\[
\begin{tikzcd}
E \ar[r,"j"] \ar[d,"p"'] & \Bl_V\PP^5 \ar[d,"\pi"] \\
V\ar[r,"{\nu_2}"'] & \PP^5
\end{tikzcd}
\]

\begin{proposition}[Additive blowup decomposition of $\Kont$] For a given oriented cohomology theory $A^\bullet$, we have that $A^m(\Kont)\cong\;A^m(\mathbb P^5)\;\oplus\;\bigoplus_{k=1}^{2}A^{m-k}(\mathbb P^2)$. In particular, since all the spaces involved are projective spaces, by the projective bundle formula, $A^\bullet(\Kont)$ is a finitely generated free $A^\bullet(pt)$-module of rank $12$.
    
\end{proposition}
\begin{proof}
    This follows directly from \Cref{thm:blowup-formula-additive} and the identification $\Kont \cong \Bl_V\mathbb P^5$.
\end{proof}
We now identify the normal bundle of the Veronese surface $V$ in $\mathbb P^5$, with the symmetric square of the tangent bundle of $\mathbb P^2$.
\begin{proposition}
Let
\(
\nu_2 \colon \mathbb{P}^2 \hookrightarrow \mathbb{P}^5
\)
be the quadratic Veronese embedding, and let \(V=\nu_2(\mathbb{P}^2)\subset \mathbb{P}^5\) be the Veronese surface. The normal bundle of this embedding is
\[
\N_{V/\PP^5} \cong \operatorname{Sym}^2\!\bigl(\mathcal{T}_{\PP^2}\bigr).
\]

\end{proposition}

\begin{proof}
Let \(W\) be a three-dimensional vector space, such that \(\mathbb{P}^2=\mathbb{P}(W)\) and \(\mathbb{P}^5=\mathbb{P}(\operatorname{Sym}^2 W)\). Let
\[
0 \to \mathcal{L} \to W \otimes \mathcal{O}_{\mathbb{P}^2} \to \mathcal{Q} \to 0
\]
be the universal exact sequence on \(\mathbb{P}^2\), with \(\mathcal{L}:= \mathcal{O}_{\mathbb{P}^2}(-1)\). Then the tangent bundle is given by
\[
\mathcal{T}_{\PP^2}\cong \operatorname{Hom}(\L, \Q) \cong \mathcal{L}^{\vee}\otimes \mathcal{Q}.
\]

The quadratic Veronese map sends a line \([\ell]\subset W\) to the line \([\ell^2]\subset \operatorname{Sym}^2 W\). Hence the pullback of the tautological line bundle on \(\mathbb{P}(\operatorname{Sym}^2 W)\) is \(\mathcal{L}^2\), and therefore
\[
\nu_2^*\T_{\mathbb{P}^5}
\cong
(\mathcal{L}^2)^{\vee}\otimes
\frac{\operatorname{Sym}^2 W \otimes \mathcal{O}_{\mathbb{P}^2}}{\mathcal{L}^2}.
\]

Taking the symmetric square of the universal sequence yields the standard filtration on \(\operatorname{Sym}^2 W \otimes \mathcal{O}_{\mathbb{P}^2}\) whose successive quotients are
\(
\mathcal{L}^2,
\mathcal{L}\otimes \mathcal{Q},\) and \(
\operatorname{Sym}^2 \mathcal{Q}.
\)
Quotienting by \(\mathcal{L}^2\) and tensoring with \((\mathcal{L}^2)^{\vee}\) gives
\[
0 \to \mathcal{L}^{\vee}\otimes \mathcal{Q}
\to
\nu_2^{*}\mathcal{T}_{\mathbb{P}^5}
\to
(\mathcal{L}^2)^{\vee}\otimes \operatorname{Sym}^2 \mathcal{Q}
\to 0.
\]
Since we have the identifications \(\mathcal{L}^{\vee}\otimes \mathcal{Q}\cong \mathcal{T}_{\PP^2}\) and $(\mathcal{L}^2)^{\vee}\otimes \operatorname{Sym}^2 \mathcal{Q} \cong \operatorname{Sym}^2(\L^{\vee} \otimes \Q)$, comparison with the normal bundle sequence
\[
0 \to \mathcal{T}_{\PP^2}
\to
\nu_2^*\T_{\mathbb{P}^5}
\to
\N_{V/\PP^5}
\to 0
\]
proves the asserted description of the normal bundle.
\end{proof}

\subsection{The Exceptional Divisor and its Oriented Cohomology}
 We use the description of the normal bundle to compute the oriented cohomology ring of the exceptional divisor in $\Bl_V\mathbb P^5$. We interchangeably use $i=v_2$ for the Veronese embedding of $\mathbb P^2$ inside $\mathbb P^5$.
\begin{lemma}
 The pullback map $i^* : A^\bullet(\mathbb P^5) \to A^\bullet(\mathbb P^2)$ is not surjective for a general theory $A^\bullet$.
\end{lemma}
\begin{proof}
 Let $A^\bullet(pt)$ be denoted by $R$. The oriented cohomology rings of $\mathbb P^5$ and $\mathbb P^2$ are given by $R[\alpha]/\alpha^6$ and $R[\eta]/\eta^3$, where $\alpha$ and $\eta$ denote the corresponding fundamental classes of the hyperplanes in the projective spaces. Then we have $i^*(\alpha) = [2]_F \eta = 2\eta + a_{11}\eta^2$. Therefore, for a general theory, the class $\eta$ is not in the image of $i^*$.
\end{proof}

The above lemma justifies why we need the full Blowup Presentation (\cref{thm:blowup-presentation}) for the current example. We now compute the oriented cohomology ring of the exceptional divisor which is the projectivization of the normal bundle, which we computed in the previous section.
\begin{lemma} The oriented cohomology ring of the exceptional divisor, which is $\mathbb P(\N_{\mathbb P^2/\mathbb P^5}) = \mathbb P(\operatorname{Sym}^2 (\T_{\mathbb P^2}))$, is given by \[\frac{A^\bullet(\mathbb P^2)[\zeta]}{\zeta^3+(9\eta+6a_{11}\eta^2)\cdot\zeta^2+30\eta^2\cdot\zeta.}\] Here $\zeta$ is the fundamental class of the hyperplane corresponding to $\mathcal{O}(1)$ on $\mathbb P(\symtangent)$. 

\end{lemma}
\begin{proof}
 To prove this lemma we first find the Chern classes for our given oriented cohomology theory for $\operatorname{Sym}^2(\T_{\mathbb P^2})$ and then use \cref{thm:projective-bundle-formula}.

From the Euler exact sequence on $\mathbb P^2$ given by \[0\to \mathcal{O}\to \mathcal{O}(1)^{\oplus{3}}\to \T_{\mathbb P^2}\to 0, \]
we deduce that the Chern classes of the tangent bundle $\T_{\mathbb P^2}$ are given by \[c_0=1, c_1=3\eta, c_2= 3\eta^2, c_3=0,\] by using the Whitney sum formula. In order to now get the Chern classes for $\symtangent$, we will use the splitting principle. Assume formally $\T_{\mathbb P^2} = \L_1\oplus \L_2$, such that the $c_1(\L_1)=x$ and $c_1(\L_2)=y$, where $\L_i$ are line bundles. We conclude that \[\begin{gathered}x+y= 3\eta \\
xy=3\eta^2.\end{gathered}\] Then we have \[\symtangent = \L_1^{\otimes2} \oplus \L_1\otimes \L_2 \oplus \L_2^{\otimes2}.\]
The total Chern class of $\symtangent$ is thus given by $(1+x+_Fx)(1+x+_Fy)(1+y+_Fy)$, where $+_F$ denotes the formal sum. Note that we have $z+_Fw=z+w+a_{11}zw+\textit{higher order terms}$, however since the higher order terms in this case are all of the form $\eta^i$ such that $i \geq 3$, we can conclude that they are all zero. \\
Expanding the above expression and using the relations between $x$ and $y$, we conclude that the total Chern class is given by 
\[1+(9\eta+6a_{11}\eta^2) +30\eta^2.\]
The result now follows from the projective bundle formula \cref{thm:projective-bundle-formula}.
\end{proof}
\begin{corollary}
The oriented cohomology ring of the exceptional divisor as a free $\pointring$-module \\ has basis $\{1, \eta, \zeta, \eta^2, \eta\zeta, \zeta^2, \eta^2\zeta, \eta\zeta^2,\eta^2\zeta^2 \}$. \hfill \qed
\end{corollary}

\begin{lemma}\label{zeta-relations} The relations $\zeta^4 = 51\eta^2\zeta^2$ and $\zeta^5=0$ hold in $A^\bullet(E)$.
\end{lemma}
\begin{proof}
 We iteratively use our previous relations $\zeta^3=-9(\eta+6a_{11}\eta^2)\zeta^2-30\eta^2\zeta$ and $\eta^3=0$. Multiplying by $\zeta$, we get the desired expressions.
\end{proof}

We now compute the formal inverse $\chi_A(\zeta)$ of $\zeta$, which will be needed for writing down the relations for the Blowup Presentation.

\begin{proposition}\label{prop:formal-inverse-zeta}
Assume the formal inverse of $\zeta$ has the form $\chi_A(\zeta) = d_1\zeta+d_2\zeta^2+d_3\zeta^3+d_4\zeta^4$. Then the coefficients are given by $d_1 = -1$, $d_2 = a_{11}$, $d_3 = -a_{11}^2$, and $d_4 = a_{11}^3 + a_{11}a_{12} - a_{22} + 2a_{13}$.
\end{proposition}
\begin{proof}
Expanding the formal group law for $\zeta+_F\chi_A(\zeta)=0$, we have $0=\zeta+ (d_1\zeta+d_2\zeta^2+d_3\zeta^3+d_4\zeta^4)+a_{11}\zeta(d_1\zeta+d_2\zeta^2+d_3\zeta^3+d_4\zeta^4) +a_{12}\zeta(d_1\zeta+d_2\zeta^2+d_3\zeta^3+d_4\zeta^4)^2+\text{higher order terms}.$
Comparing the coefficients of the powers of $\zeta$, and noting that $\zeta^5=0$, we sequentially deduce the values of $d_i$. In degree $1$, we get $d_1=-1$. In degree $2$, the relation $d_2+a_{11}d_1=0$ yields $d_2=a_{11}$. In degree $3$, we have $d_3+a_{11}d_2+a_{12}d_1^2+a_{21}d_1=0$, which gives $d_3= -a_{11}^2$ (note that $a_{12} = a_{21}$ forces a cancellation). In degree $4$, the relation $d_4+a_{11}d_3+a_{12}(2d_1d_2)+a_{21}d_2+a_{31}d_1+a_{22}d_1^2+a_{13}d_1^3=0$ simplifies to $d_4= a_{11}^3 + a_{11}a_{12} - a_{22} + 2a_{13}$.
\end{proof}

\subsection{Blowup Presentation of {$Bl_V\PP^5$}}

In order to write down the Blowup Presentation for the oriented cohomology ring of the moduli space of complete conics, we begin by computing the fundamental classes for the Veronese embedding. These will be required for writing down the excess intersection formula.
\begin{proposition}[Fundamental classes of the Veronese embedding]\label{prop:fundamental-class-of-veronese}
Let
$\alpha:=c_1^A(\cO_{\PP^5}(1))\in A^1(\PP^5)$ and $\eta:=c_1^A(\cO_{\PP^2}(1))\in A^1(\PP^2)$. Then in
$A^\bullet(\PP^5)\cong A^\bullet(\pt)[\alpha]/(\alpha^6)$, one has $i_\ast 1_V=4\alpha^3+3a_{11}\alpha^4+3a_{21}\alpha^5$, $i_*\eta=2\alpha^4+a_{11}\alpha^5$, and $i_*(\eta^2)= \alpha^5$.
\end{proposition}

\begin{proof}
We work out each pushforward separately. For $i_*1_V$, since $A^\bullet(\PP^5)\cong A^\bullet(\pt)[\alpha]/(\alpha^6)$ with $\deg(\alpha)=1$, the codimension-3 group decomposes as $A^3(\PP^5) = A^0(\pt)\alpha^3 \oplus A^{-1}(\pt)\alpha^4 \oplus A^{-2}(\pt)\alpha^5$. We may therefore write $i_\ast1_V=b_3\alpha^3+b_4\alpha^4+b_5\alpha^5$. Let $\pi_5\colon \PP^5\to \Spec k$ and $\pi_2\colon \PP^2\to \Spec k$ be the structure morphisms. For $k=0,1,2$, the projection formula gives$$\pi_{5,\ast}\bigl(i_\ast1_V\,\cdot \alpha^k\bigr)=\pi_{2,\ast}\bigl(i^\ast(\alpha^k)\bigr).$$ Since $i^\ast\cO_{\PP^5}(1)\cong \cO_{\PP^2}(2)$ and $\eta^3=0$, we have $i^\ast(\alpha)=c_1^A(\cO_{\PP^2}(2))=[2]_F(\eta)= 2\eta+a_{11}\eta^2$, which immediately yields $i^\ast(\alpha^2)= 4\eta^2$. Next, write the invariant differential as $\omega_F(t)\,dt=\bigl(\partial_2F|_{(t,0)}\bigr)^{-1}dt$ with $\omega_F(t)=\sum_{r\ge 0}\omega_r t^r$. Since $\partial_2F|_{(t,0)}=1+a_{11}t+a_{21}t^2+\cdots$, we find $\omega_1=-a_{11}$ and $\omega_2=a_{11}^2-a_{21}$.By \Cref{cor:projective-space-residue} applied to $\PP^5$ and $\PP^2$, we have $\pi_{5,\ast}(\alpha^5)=1$, $\pi_{5,\ast}(\alpha^4)=\omega_1$, $\pi_{5,\ast}(\alpha^3)=\omega_2$, along with $\pi_{2,\ast}(\eta^2)=1$, $\pi_{2,\ast}(\eta)=\omega_1$, and $\pi_{2,\ast}(1)=\omega_2$.

We now evaluate the projection formula for $k=2,1,0$ to determine the coefficients. For $k=2$, we obtain $b_3=\pi_{2,\ast}(4\eta^2)=4$. For $k=1$, the left side evaluates to $b_3\omega_1+b_4$, while the right side is $\pi_{2,\ast}(2\eta+a_{11}\eta^2)=2\omega_1+a_{11}$; substituting $\omega_1$ gives $b_4=3a_{11}$. Finally, for $k=0$, the relation $b_3\omega_2+b_4\omega_1+b_5=\pi_{2,\ast}(1)=\omega_2$ forces $b_5=3a_{21}$.

For $i_*(\eta)$, we use the observation that the image of a line in $\mathbb P^2$ inside $\mathbb P^5$ under the veronese embedding is the complete intersection of three hyperplanes and a quadric. Let us call the image of the line under the veronese embedding to be $C$. Then we have $C=H_1\cap H_2 \cap H_3 \cap Q$, where $H_i$ are hyperplanes and $Q$ is a quadric hypersurface in $\mathbb P^5$, such that the hypersurfaces are all smooth and intersect transversally. Hence we can use \cref{prop:intersection-product} to get $[C]=[H_1\cap H_2 \cap H_3 \cap Q]=[H_1].[H_2].[H_3].[Q]$. Here $Q$ corresponds to a section of $\mathcal{O}_{\mathbb P^5}(2)$ and $H_i$ corresponds to a section of $\mathcal{O}_{\mathbb P^5}(1)$. Using \cref{funclass-divisor} we conclude that $[Q]=c_1(\mathcal{O}_{\mathbb P^5}(2))=\alpha+_F\alpha$ and $[H_i]=\alpha$. Hence we have $[C]=\alpha^3.(2\alpha+a_{11}\alpha^2)$. Thus we get $i_*(\eta)=2\alpha^4+a_{11}\alpha^5$.

Finally, $\eta^2$ is the fundamental class of a point in $\mathbb P^2$, whose image under the Veronese embedding is also a point, which in turn is the complete intersection of $5$ hyperplanes in $\mathbb P^5$. Therefore, $i_*(\eta^2)=\alpha^5$.

\end{proof}

\begin{theorem}[Blowup Presentation of $Bl_V\PP^5$] \label{thm:oriented-coh-of-BlVP5}
Let $\alpha := \pi^*c_1^A(\mathcal{O}_{\mathbb{P}^5}(1))$ be the pullback of the hyperplane class from $\mathbb{P}^5$. Let $\eta=p^*c_1^A(\mathcal{O}_{\mathbb{P}^2}(1))$ and $\zeta=c_1^A(\mathcal{O}_E(1))$ be the hyperplane classes on the exceptional divisor $E$. Define the $9$ exceptional pushforward classes 
\[ e_{a,b} := j_*\big(\eta^a \zeta^b\big) \quad \text{for } 0 \le a, b \le 2. \]
The oriented cohomology ring $A^\bullet(\Kont)$ is generated as an $A^\bullet(\pt)$-algebra by $\alpha$ and the $9$ classes $e_{a,b}$, subject to the following four families of relations.
For notational convenience in the relations below, we define $e_{a,b} = 0$ for $a \ge 3$ or $b \ge 5$. The symbols $e_{a,3}$ and $e_{a,4}$ are not new generators, but a recursive shorthand for the following linear combinations of the base generators
\begin{align*}
e_{a, 3} &= -9e_{a+1, 2} - 6a_{11}e_{a+2, 2} - 30e_{a+2, 1}, \\
 e_{a, 4} &= 51e_{a+2, 2}.
 \end{align*}

 \begin{enumerate}[(i)]
 \item \emph{(Intersection on $\mathbb P^5$)} $\alpha^6=0$.
 \item \emph{(Intersection between $E$ and $\mathbb P^5$)} For all $0 \le a,b \le 2$,
 \[ \alpha \cdot e_{a,b} = 2e_{a+1, b} + a_{11}e_{a+2, b}. \]
 \item \emph{(Intersection on $E$)} For all $0 \le a,b,c,d \le 2$,
 \[ e_{a,b} \cdot e_{c,d} = -e_{a+c, b+d+1} + a_{11}e_{a+c, b+d+2} + 30a_{11}^2e_{a+c+2, b+d+1} + 9a_{11}^2e_{a+c+1, b+d+2} + s e_{a+c+2, b+d+2}, \]
 where $s := 57a_{11}^3 + 51a_{11}a_{12} - 51a_{22} + 102a_{13}$.
 \item \emph{(Excess intersection)} For $k \in \{0, 1, 2\}$,
 \[ \pi^*i_*(\eta^k) = 30e_{k+2,0} + 9e_{k+1,1} + e_{k,2} + 66a_{11}e_{k+2,1} + 9a_{11}e_{k+1,2} + 78a_{11}^2e_{k+2,2}, \]
where the left-hand side evaluations are given by
 \begin{align*}
\pi^*i_*1 &= 4\alpha^3+3a_{11}\alpha^4+3a_{21}\alpha^5, \\
 \pi^*i_*(\eta) &= 2\alpha^4+a_{11}\alpha^5, \\
 \pi^*i_*(\eta^2) &= \alpha^5.
 \end{align*}
\end{enumerate}
\end{theorem}
\begin{proof}

We will establish the oriented cohomology ring of $\Kont$ in the following steps. We first consider the blowup square 
\[
\begin{tikzcd}
E \ar[r,"j"] \ar[d,"p"'] & \Bl_V\PP^5 \ar[d,"\pi"] \\
V\ar[r,"{\nu_2}"'] & \PP^5
\end{tikzcd}
\]
We will workout each family of relations from \cref{thm:blowup-presentation}.
 \begin{enumerate}[(i)]
\item (Intersection on $\mathbb P^5$) We have $\pi^*(\alpha^s.\alpha^t)=\pi^*(\alpha^s).\pi^*(\alpha^t)$, so $\alpha^6=0$.
 \item (Intersection between $E$ and $\mathbb P^5$) For any basis element $e_{a,b} = j_*(\eta^a \zeta^b)$ of $A^\bullet(E)$, we have $e_{a,b} \cdot \pi^*(\alpha) = j_*(\eta^a \zeta^b) \cdot \pi^*(\alpha) = j_*\big(\eta^a \zeta^b \cdot p^*i^*(\alpha)\big)$. Since $i^*(\alpha) = c_1^A(\mathcal{O}_{\mathbb{P}^2}(2)) = 2\eta + a_{11}\eta^2$, this becomes:
 \[ j_*\big(\eta^a \zeta^b (2\eta + a_{11}\eta^2)\big) = j_*(2\eta^{a+1}\zeta^b + a_{11}\eta^{a+2}\zeta^b). \]
 Translating back into our basis notation yields $2e_{a+1, b} + a_{11}e_{a+2, b}$. Note that we are interchangeably using $p^*\eta$ and $\eta$, since $p^*$ is a ring homomorphism.

 \item (Intersection on $E$) From \cref{prop:formal-inverse-zeta}, we have that the formal inverse of $\zeta$ expands to $\chi_A(\zeta) = -\zeta + a_{11}\zeta^2 - a_{11}^2\zeta^3 + d_4\zeta^4$, where $d_4 = a_{11}^3 + a_{11}a_{12} - a_{22} + 2a_{13}$. Simplifying the formal inverse using the relations $\zeta^3 = -(9\eta+6a_{11}\eta^2)\zeta^2 - 30\eta^2\zeta$ and $\zeta^4 = 51\eta^2\zeta^2$, gives us the relation: 
 \[\chi_A(\zeta) = -\zeta + a_{11}\zeta^2 + 30a_{11}^2\eta^2\zeta + 9a_{11}^2\eta\zeta^2 + s\eta^2\zeta^2,\]
 where we define the coefficient $s := 57a_{11}^3 + 51a_{11}a_{12} - 51a_{22} + 102a_{13}$.

 Applying $j_*(\gamma) \cdot j_*(\delta) = j_*\big(\gamma \cdot \delta \cdot \chi_A(\zeta)\big)$ to our basis elements gives
 \[ e_{a,b} \cdot e_{c,d} = j_*(\eta^a \zeta^b) \cdot j_*(\eta^c \zeta^d) = j_*\big(\eta^{a+c}\zeta^{b+d} \chi_A(\zeta)\big). \]
 Substituting the above expression for $\chi_A(\zeta)$ into this pushforward yields the relation
\[ e_{a,b} \cdot e_{c,d} = -e_{a+c, b+d+1} + a_{11}e_{a+c, b+d+2} + 30a_{11}^2e_{a+c+2, b+d+1} + 9a_{11}^2e_{a+c+1, b+d+2} + s e_{a+c+2, b+d+2}. \]

 \item (Excess intersection) We will work out the LHS and RHS of the excess intersection formula separately.
 
 For the right hand side, let $\mathcal{Q}$ be the universal quotient bundle of rank $2$ on $E$ defined by the tautological exact sequence \[0 \to \mathcal{O}_E(-1) \to p^*\mathcal{N}_{\mathbb{P}^2/\mathbb{P}^5} \to \mathcal{Q} \to 0.\]
 By the Whitney sum formula, the total Chern class satisfies $c^A(p^*\mathcal{N}_{\mathbb{P}^2/\mathbb{P}^5}) = c^A(\mathcal{O}_E(-1)) \cdot c^A(\mathcal{Q})$. We substitute the Chern classes to get
 \[1 + (9\eta + 6a_{11}\eta^2) + 30\eta^2 = \big(1 + \chi_A(\zeta)\big)\big(1 + c_1^A(\mathcal{Q}) + c_2^A(\mathcal{Q})\big).\]
 Comparing the graded components of this equation yields:
 \begin{align*}
 \text{Degree 1: } & c_1^A(\mathcal{Q}) + \chi_A(\zeta) = 9\eta + 6a_{11}\eta^2 \implies c_1^A(\mathcal{Q}) = 9\eta + 6a_{11}\eta^2 - \chi_A(\zeta)\\
 \text{Degree 2: } & c_2^A(\mathcal{Q}) + c_1^A(\mathcal{Q})\chi_A(\zeta) = 30\eta^2.
 \end{align*}
 Substituting the degree 1 identity into the degree 2 equation yields the top Chern class
 \[c_2^A(\mathcal{Q}) = 30\eta^2 - (9\eta + 6a_{11}\eta^2)\chi_A(\zeta) + \chi_A(\zeta)^2.\]
The middle term can be simplified using the expansion of $\chi_A(\zeta)$ into
 \[-(9\eta + 6a_{11}\eta^2)\chi_A(\zeta) = 9\eta\zeta - 9a_{11}\eta\zeta^2 + 6a_{11}\eta^2\zeta - 87a_{11}^2\eta^2\zeta^2.\]
Furthermore, we have 
 \[\chi_A(\zeta)^2 = \zeta^2 + 60a_{11}\eta^2\zeta + 18a_{11}\eta\zeta^2 + 165a_{11}^2\eta^2\zeta^2.\]
 Summing the terms together, we get
 \[c_2^A(\mathcal{Q}) = 30\eta^2 + 9\eta\zeta + \zeta^2 + 66a_{11}\eta^2\zeta + 9a_{11}\eta\zeta^2 + 78a_{11}^2\eta^2\zeta^2.\]

 The right-hand side of the excess intersection formula requires evaluating $j_*(\eta^k \cdot c_2^A(\mathcal{Q}))$ for $k \in \{0, 1, 2\}$. The above computations yields
 \[j_*(\eta^k \cdot c_2^A(\mathcal{Q})) = j_*\big(30\eta^{k+2} + 9\eta^{k+1}\zeta + \eta^k\zeta^2 + 66a_{11}\eta^{k+2}\zeta + 9a_{11}\eta^{k+1}\zeta^2 + 78a_{11}^2\eta^{k+2}\zeta^2\big). \]
 Converting this back into our standard basis notation yields the right-hand side of relation (iv):
 \[ 30e_{k+2,0} + 9e_{k+1,1} + e_{k,2} + 66a_{11}e_{k+2,1} + 9a_{11}e_{k+1,2} + 78a_{11}^2e_{k+2,2}. \]

 Finally, for the left hand side, we compute $\pi^*i_*(\theta),\ \text{ for all }\theta \in \{1,\eta,\eta^2\}$. We calculate for each basis element using \cref{prop:fundamental-class-of-veronese}.
 For $\theta=1$, we have $\pi^*i_*1=\pi^*(4\alpha^3+3a_{11}\alpha^4+3a_{21}\alpha^5)$. For $\theta=\eta$, we have $\pi^*i_*(\eta)=\pi^*(2\alpha^4+a_{11}\alpha^5)$. For $\theta=\eta^2$, we have $\pi^*i_*(\eta^2)=\pi^*(\alpha^5)$.
 \end{enumerate}
\end{proof}

For explicit calculations, it is convenient to list the multiplication table (\Cref{tab:mul-table}) of the nine pushforward classes $e_{a,b}$ for $0\le a, b\le 2$. For $0 \leq a,b,c,d \leq 2$, the product $e_{a,b} \cdot e_{c,d}$ depends entirely on the integers $a+c$ and $b+d$. 
\begin{table}[ht]
\centering
\resizebox{\textwidth}{!}{%
\renewcommand{\arraystretch}{2}
\begin{tabular}{|c@{\hspace{\tabcolsep}\vrule width 1.5pt\hspace{\tabcolsep}}c|c|c|c|c|c|c|c|c|}
    \hline
    $\cdot$ & $\mathbf{e_{0,0}}$ & $\mathbf{e_{0,1}}$ & $\mathbf{e_{0,2}}$ & $\mathbf{e_{1,0}}$ & $\mathbf{e_{1,1}}$ & $\mathbf{e_{1,2}}$ & $\mathbf{e_{2,0}}$ & $\mathbf{e_{2,1}}$ & $\mathbf{e_{2,2}}$ \\
    \noalign{\hrule height 1.5pt}
    
    $\mathbf{e_{0,0}}$ & 
    $\begin{array}{@{}c@{}} -e_{0,1} + a_{11}e_{0,2} \\ + 30a_{11}^2e_{2,1} \\ + 9a_{11}^2e_{1,2} + s e_{2,2} \end{array}$ & 
    $\begin{array}{@{}c@{}} -30a_{11}e_{2,1} - e_{0,2} \\ - 9a_{11}e_{1,2} - 57a_{11}^2e_{2,2} \end{array}$ & 
    $\begin{array}{@{}c@{}} 30e_{2,1} + 9e_{1,2} \\ + 57a_{11}e_{2,2} \end{array}$ & 
    $\begin{array}{@{}c@{}} -e_{1,1} + a_{11}e_{1,2} \\ + 9a_{11}^2e_{2,2} \end{array}$ & 
    $-e_{1,2} - 9a_{11}e_{2,2}$ & 
    $9e_{2,2}$ & 
    $-e_{2,1} + a_{11}e_{2,2}$ & 
    $-e_{2,2}$ & 
    $0$ \\
    \hline
    
    $\mathbf{e_{0,1}}$ & 
    $\begin{array}{@{}c@{}} -30a_{11}e_{2,1} - e_{0,2} \\ - 9a_{11}e_{1,2} - 57a_{11}^2e_{2,2} \end{array}$ & 
    $\begin{array}{@{}c@{}} 30e_{2,1} + 9e_{1,2} \\ + 57a_{11}e_{2,2} \end{array}$ & 
    $-51e_{2,2}$ & 
    $-e_{1,2} - 9a_{11}e_{2,2}$ & 
    $9e_{2,2}$ & 
    $0$ & 
    $-e_{2,2}$ & 
    $0$ & 
    $0$ \\
    \hline
    
    $\mathbf{e_{0,2}}$ & 
    $\begin{array}{@{}c@{}} 30e_{2,1} + 9e_{1,2} \\ + 57a_{11}e_{2,2} \end{array}$ & 
    $-51e_{2,2}$ & 
    $0$ & 
    $9e_{2,2}$ & 
    $0$ & 
    $0$ & 
    $0$ & 
    $0$ & 
    $0$ \\
    \hline
    
    $\mathbf{e_{1,0}}$ & 
    $\begin{array}{@{}c@{}} -e_{1,1} + a_{11}e_{1,2} \\ + 9a_{11}^2e_{2,2} \end{array}$ & 
    $-e_{1,2} - 9a_{11}e_{2,2}$ & 
    $9e_{2,2}$ & 
    $-e_{2,1} + a_{11}e_{2,2}$ & 
    $-e_{2,2}$ & 
    $0$ & 
    $0$ & 
    $0$ & 
    $0$ \\
    \hline
    
    $\mathbf{e_{1,1}}$ & 
    $-e_{1,2} - 9a_{11}e_{2,2}$ & 
    $9e_{2,2}$ & 
    $0$ & 
    $-e_{2,2}$ & 
    $0$ & 
    $0$ & 
    $0$ & 
    $0$ & 
    $0$ \\
    \hline
    
    $\mathbf{e_{1,2}}$ & 
    $9e_{2,2}$ & 
    $0$ & 
    $0$ & 
    $0$ & 
    $0$ & 
    $0$ & 
    $0$ & 
    $0$ & 
    $0$ \\
    \hline
    
    $\mathbf{e_{2,0}}$ & 
    $-e_{2,1} + a_{11}e_{2,2}$ & 
    $-e_{2,2}$ & 
    $0$ & 
    $0$ & 
    $0$ & 
    $0$ & 
    $0$ & 
    $0$ & 
    $0$ \\
    \hline
    
    $\mathbf{e_{2,1}}$ & 
    $-e_{2,2}$ & 
    $0$ & 
    $0$ & 
    $0$ & 
    $0$ & 
    $0$ & 
    $0$ & 
    $0$ & 
    $0$ \\
    \hline
    
    $\mathbf{e_{2,2}}$ & 
    $0$ & 
    $0$ & 
    $0$ & 
    $0$ & 
    $0$ & 
    $0$ & 
    $0$ & 
    $0$ & 
    $0$ \\
    \hline
\end{tabular}%
}
\caption{Multiplication table of the pushforward classes of the exceptional divisor in $\Bl_V\mathbb P^5$.}
\label{tab:mul-table}
\end{table}

\begin{remark}
In general, the oriented cohomology ring $A^\bullet(\Bl_V\PP^5)$ is not generated by divisor classes over $A^\bullet(\pt)$. The only degree $1$ generators are the hyperplane pullback $\alpha$ and the exceptional divisor $e_{0,0}$. We have $\alpha \cdot e_{0,0} = 2e_{1,0} + a_{11}e_{2,0}$. Because $2$ is not generally invertible in $A^\bullet(\pt)$, the codimension-$2$ class $e_{1,0}$ cannot be expressed as a polynomial in $\alpha$ and $e_{0,0}$, making it a higher-degree generator.
\end{remark}
\subsection{$3264$ Conics via Cobordism}

We demonstrate how to recover the solution to Steiner's enumerative problem.

\begin{quote}
\begin{itemize}
    \item[$\dagger$] \textit{``Einen Kegelschnitt $K$ zu finden, welche irgend fünf gegebenen Kegelschnitte berührt.'' }

\hfill ----Jakob Steiner, 1848 \cite{Steiner1848Kegelschnitt}.
    \item[$\dagger$] (Glossary) How many conics are simultaneously tangent to $5$ general conics in $\PP^2$?
\end{itemize}
\end{quote}
Steiner then wrote in his article,

\begin{quote}
    \begin{itemize}[>]
        \item \textit{``Daß fünf beliebige gegebene Kegelschnitte im Allgemeinen (und höchstens) von $7776$ andern Kegelschnitten K berührt werden.''}
        \item That five arbitrary given conic sections are in general (and at most) touched by [tangent to] $7776$ other conic sections $K$.
    \end{itemize}
\end{quote}
Steiner's count turned out to be incorrect. Chasles in 1864 revisited this problem and proposed the count $3264$, which is now known to be correct \cite{Chasles1864Construction, Chasles1864Determination}. While a well-known modern solution is given by Fulton's Excess Intersection Formula (\cite{FultonMacPherson1978DefiningAlgebraicIntersections}, see also \cite[Example 9.1.9]{Fulton} and \cite[Chapter 8]{EH3264}), we outline a solution via cobordism-valued intersection theory. Consider the parameter space of plane conics $\PP(H^0(\O_{\PP^2}(2))) \cong \PP^5$. Fix a general conic $C_1\subseteq \PP^2$. A conic $C$ is \textit{tangent} to $C_1$ if each point of its intersection with $C_1$ has multiplicity greater than or equal to $2$. The locus of conics tangent to $C_1$ is a degree $6$ hypersurface in $\PP^5$. Then, given $5$ general conics, the locus $Z$ of conics simultaneously tangent to all five is the intersection of $5$ such sextic hypersurfaces $S_1, \dots, S_5$ in $\PP^5$. If these hypersurfaces were general, then $Z$ is expected to be a zero-dimensional scheme of length $6^5 = 7776$. 

This number can be recovered from cobordism of $\PP^5$ in a straightforward way. The cobordism ring of $\PP^5$ has presentation $\Omega^\bullet(\PP^5) \cong \mathbb{L}[\alpha] / \alpha^6$ given by the Projective Bundle Formula. The fundamental class of a general sextic hypersurface is given by the first Chern class $c_1^\Omega(\O_{\PP^5}(6))$, which equals $[6]_\Omega \alpha$. The fundamental class of the complete intersection $Z$ is then given by $([6]_\Omega\alpha)^5$, which one can easily check to be $7776 \alpha^5 = 7776[\pt]_\Omega$ due to the vanishing of $\alpha^6$. This number would indeed be the naïve solution, recovered from cobordism.

However, the assumption that $Z$ is zero-dimensional is false. Indeed, the parameter space $\PP^5$ contains all degenerate conics, and any conic in the form of a \textit{double line} intersects every given conic $C_i$ at points with multiplicity at least $2$. Therefore, double lines are ``excess'' solutions to Steiner's problem. Since any double line is defined by the square of a linear form on $\PP^2$, the locus of double lines in $\PP^5$ is exactly the Veronese surface $V$. For an exposition to the construction of the moduli space of complete conics, see \Cref{apx:moduli-of-complete-conics}. Blowing up $\PP^5$ at $V$ therefore removes the double line solutions, and the desired number is given by the complete intersection of pullbacks $\widetilde S_i$ of the sextic hypersurfaces $S_i$. 

Let $H$ be the divisor on $Bl_V\PP^5$ corresponding to the pullback of the hyperplane, and let $E $ be the exceptional divisor. Then, the divisor class of each $\widetilde S_i$ is $6H - 2E$.  One checks that $\widetilde Z := \widetilde S_1 \cap \cdots \cap \widetilde S_5$ is indeed zero-dimensional, so its fundamental class is given by 
\begin{equation}
  \begin{split}
        [\widetilde Z]_\Omega & = [\widetilde S_1]_\Omega  \cdots [\widetilde S_5]_\Omega \\
    & = \bigl(c_1^\Omega(\O_{Bl_V\PP^5}(6H - 2E))\bigr)^5 \\
    & = ([6]_\Omega \alpha \; -_\Omega \; [2]_\Omega e_{0,0})^5 \quad \in \Omega^5(Bl_V\PP^5).
  \end{split}
\end{equation}

A \texttt{Macaulay2} calculation (\Cref{apx:3264-conics}) shows that this class evaluates precisely to $3264 \alpha^5 = 3264 [\pt]_\Omega$. Remarkably, because all formal group law coefficients $a_{ij}$ cancel out in the top degree of this intersection, the cobordism-theoretic count yields a pure integer with no correction terms. This demonstrates that the enumerative geometry of this specific problem is rigid across all oriented cohomology theories. Therefore, we recover the well-celebrated result of ``$3264$ conics'' using algebraic cobordism. 

\begin{theorem}[3264 Conics]
    There are $3264$ conics simultaneously tangent to $5$ general conics in $\PP^2$, which are not double lines. \hfill \qed
\end{theorem}

\section{Oriented Cohomology Ring of \texorpdfstring{$\overline{M}_{0,n}$}{M0n-bar}} \label{sec:oriented-cohomology-of-M0nbar}

In this section we give a presentation of the oriented cohomology ring of $\overline{M}_{0,n}$ for any oriented cohomology theory $A^\bullet$ satisfying the dimension axiom $(\mathrm{Dim})$. By specializing $A^\bullet$ to algebraic $K$-theory, the same presentation also recovers one of the presentations of $K_0(\overline{M}_{0,n})$ in \cite{LLPP2024WonderfulVarieties}. Our argument is based on Keel's iterated blowup description of $\overline{M}_{0,n+1}$ via the morphism
\(
\overline{M}_{0,n+1}\longrightarrow \overline{M}_{0,n}\times \PP^1.
\)
We begin by recalling the geometric facts from \cite{Keel} that will be used throughout the section.

\subsection{Keel's iterated blowup description and boundary divisors}

Let $\pi \colon \overline{M}_{0,n+1}\to \overline{M}_{0,n}$ be the universal curve, and let $\sigma_1,\dots,\sigma_n$ be its tautological sections. Forgetting all marked points except $1,2,3,$ and $n+1$ gives a morphism $\overline{M}_{0,n+1}\to \overline{M}_{0,4}$. Together with $\pi$, this defines a morphism
\[
\pi_1\colon \overline{M}_{0,n+1}\longrightarrow \overline{M}_{0,n}\times \overline{M}_{0,4}\cong \overline{M}_{0,n}\times \PP^1.
\]
Since $\overline{M}_{0,4}\cong \PP^1$, we may regard $\pi_1$ as a morphism to $\overline{M}_{0,n}\times \PP^1$.

For a subset $T\subset [n]=\{1,\dots,n\}$ with $2\le |T|\le n-2$, let $D_n^T\subset \overline{M}_{0,n}$ be the boundary divisor whose general point parametrizes a two-component stable curve for which the markings indexed by $T$ lie on one component and those indexed by $T^c$ lie on the other. By Knudsen's description of the boundary, there is an isomorphism
\[
D_n^T \cong \overline{M}_{0,|T|+1}\times \overline{M}_{0,|T^c|+1}.
\]
For any $i\in T$, the composite $\pi_1\circ \sigma_i\colon \overline{M}_{0,n}\to \overline{M}_{0,n}\times \overline{M}_{0,4}$ restricts to an embedding of $D_n^T$, and Keel shows that this restriction is independent of the choice of $i\in T$.

Set $B_1\coloneqq \overline{M}_{0,n}\times \overline{M}_{0,4}$. Define $B_2$ to be the blowup of $B_1$ along the disjoint union of the embedded divisors $D_n^T$ with $|T^c|=2$. Inductively, having defined $B_k$, let $B_{k+1}$ be the blowup of $B_k$ along the disjoint union of the strict transforms of the embedded divisors $D_n^T$ with $|T^c|=k+1$. The centers that appear in this construction are smooth and are isomorphic to products $\overline{M}_{0,i}\times \overline{M}_{0,j}$ with $i,j<n$.

\begin{theorem}
With the above notation, the morphism $\pi_1\colon \overline{M}_{0,n+1}\to \overline{M}_{0,n}\times \overline{M}_{0,4}$ is isomorphic to the composition of the blowups $B_{n-2}\to B_1$.
\end{theorem}

On the intermediate spaces $B_k$, the strict transforms of pullbacks of boundary divisors from $B_1$, together with the exceptional divisors arising in Keel's construction, are commonly referred to as \emph{boundary divisors}. Under the identification $B_{n-2}\cong \overline{M}_{0,n+1}$, these are exactly the usual boundary divisors on $\overline{M}_{0,n+1}$.

For later use, if $T\subset [n]$ with $2\le |T|\le n-2$, define the boundary divisor $D_{n+1}^{T,n+1}\subset \overline{M}_{0,n+1}$ by
\[
\pi^{-1}(D_n^T)=D_{n+1}^T + D_{n+1}^{T,n+1}.
\]
Following Keel, two subsets $S,T\subset [n]$ are called \emph{compatible}, and we write $S ** T$, if one of the following holds:
\[
S\subset T,\qquad T\subset S,\qquad S\cap T=\emptyset,\qquad \text{or}\qquad S\cup T=[n].
\]

\begin{lemma}\label{lem:boundary-divisors-facts}
With the notation above, the following statements hold.
\begin{enumerate}
    \item For every $T\subset [n]$ with $2\le |T|\le n-2$, the inclusion $D_n^T\hookrightarrow \overline{M}_{0,n}$ admits a section. Consequently, the pullback
    \[
    A^\bullet(\overline{M}_{0,n})\longrightarrow A^\bullet(D_n^T)
    \]
    is surjective.
    \item The exceptional divisors of the morphism $\pi_1\colon \overline{M}_{0,n+1}\to \overline{M}_{0,n}\times \overline{M}_{0,4}$ are precisely the divisors $D_{n+1}^{T,n+1}$.
    \item Two boundary divisors $D_n^S$ and $D_n^T$ have non-trivial intersection if and only if $S ** T$. In particular, the strict transforms of the blowup centers at each stage are pairwise disjoint.
\end{enumerate}
\end{lemma}

\begin{proof}
For (1), Keel constructs a contraction morphism
\[
\pi_T\colon \overline{M}_{0,n}\longrightarrow \overline{M}_{0,|T|+1}\times \overline{M}_{0,|T^c|+1}
\]
whose restriction to $D_n^T$ is an isomorphism \cite[Fact~2]{Keel}. Composing $\pi_T$ with the inverse of this restriction gives a section of the inclusion $D_n^T\hookrightarrow \overline{M}_{0,n}$. The surjectivity of pullback then follows from \Cref{lem:section-implies-surjective}. Statement (2) is exactly the description of the exceptional divisors in Keel's blowup model \cite[Theorem~2]{Keel}. Statement (3) is the standard combinatorial criterion for intersections of boundary divisors on $\overline{M}_{0,n}$; see \cite[\S1]{Keel}. The final claim follows from the fact that, in Keel's construction, the centers blown up at each fixed stage form a disjoint union.
\end{proof}

\subsection{Generators}

We show that the fundamental classes of boundary divisors $x_S:=[D^S]_A$ generate $A^\bullet(\Mn)$ as an $A^\bullet(\pt)$-algebra. Before that, we work out the rank of $A^\bullet(\Mn)$ as a free $A^\bullet(\pt)$-module. 
\begin{proposition}\label{prop:rank-equals-chow-rank}
Let $A^\bullet$ be an oriented cohomology theory. Then $A^\bullet(\Mn)$ is a finite free
$A^\bullet(\pt)$-module, and
\[
\operatorname{rk}_{A^\bullet(\pt)} A^\bullet(\Mn)
=
\operatorname{rk}_{\mathbb Z} CH^\bullet(\Mn).
\]
\end{proposition}

\begin{proof}
This follows by induction from the additive blowup decomposition of \Cref{thm:blowup-formula-additive}. At each step of Keel's recursive construction of $\Mn$, the free $A^\bullet(\pt)$-rank is obtained from the corresponding ranks of the ambient space and the blowup center by the same additive formula as in Chow theory. Hence the rank recursion for $A^\bullet(\Mn)$ agrees with the Chow-theoretic one, and the claim follows.
\end{proof}

In particular, Aluffi--Marcolli--Nascimento \cite[Corollaries~1.2 and~1.5]{AMN2025GrothendieckClass} computed the Betti numbers of $\Mn$ explicitly in terms of binomial coefficients and Stirling numbers of the second kind. Thus \Cref{prop:rank-equals-chow-rank} identifies the free $A^\bullet(\pt)$-rank of $A^\bullet(\Mn)$ with the total Chow rank, equivalently with the total even Betti number of $\Mn$.

\begin{proposition}\label{prop:generators-for-M0nbar}
The oriented cohomology ring $A^\bullet(\overline{M}_{0,n})$ is generated as an $A^\bullet(\pt)$-algebra by the fundamental classes $x_S \coloneqq [D^S]_A \in A^1(\overline{M}_{0,n})$ of the boundary divisors.
\end{proposition}

\begin{proof}
We argue by induction on $n$. The cases $n=3$ and $n=4$ are immediate: $\overline{M}_{0,3}$ is a point, and $\overline{M}_{0,4}\cong \PP^1$, whose oriented cohomology ring is generated by the hyperplane class, equivalently by the class of the unique boundary divisor.

Assume now that the statement holds for $\overline{M}_{0,n}$ with $n> 4$. Let
\[
B_1=\overline{M}_{0,n}\times \overline{M}_{0,4}\xleftarrow{\ \pi_2\ } B_2\xleftarrow{\ \pi_3\ }\cdots\xleftarrow{\ \pi_{n-2}\ } B_{n-2}\cong \overline{M}_{0,n+1}
\]
be Keel's iterated blowup description. By the inductive hypothesis, $A^\bullet(\overline{M}_{0,n})$ is generated by boundary divisor classes, and $A^\bullet(\overline{M}_{0,4})=A^\bullet(\PP^1)$ is generated by the class of its boundary divisor. Therefore $A^\bullet(B_1)$ is generated, as an $A^\bullet(\pt)$-algebra, by boundary divisor classes pulled back from the two factors.

Suppose inductively that $A^\bullet(B_{k-1})$ is generated by boundary divisor classes. Write
\[
\pi_k\colon B_k=\Bl_{Z_k}B_{k-1}\longrightarrow B_{k-1}
\]
for the $k$th blowup, and let $E_k\subset B_k$ be its exceptional divisor. By Keel's construction, the center $Z_k$ is a smooth boundary stratum isomorphic to  $\overline{M}_{0,i}\times \overline{M}_{0,j}$ for some $i,j< n$. Moreover, the pullback map $A^\bullet(B_{k-1})\to A^\bullet(Z_k)$ is surjective by \Cref{lem:section-implies-surjective} and \Cref{lem:boundary-divisors-facts}.

We may therefore apply the surjective Blowup Presentation of \Cref{thm:blowup-presentation-surjective}. It follows that $A^\bullet(B_k)$ is generated over $A^\bullet(B_{k-1})$ by the exceptional divisor class $[E_k]_A$. By \Cref{lem:boundary-divisors-facts} (2), this is again a boundary divisor class. Since strict transforms of boundary divisor classes on $B_{k-1}$ remain boundary divisor classes on $B_k$, we conclude that $A^\bullet(B_k)$ is generated by boundary divisor classes.

Induction on $k$ shows that $A^\bullet(B_{n-2})\cong A^\bullet(\overline{M}_{0,n+1})$ is generated by boundary divisor classes. This proves the proposition for $n+1$, and hence for all $n$.
\end{proof}

\subsection{Relations}\label{sec:relation-A(M0n)}
We work out the families of relations in $A^\bullet(\Mn)$ implied by the calculus of boundary divisors. They generalize the families of relations in Keel's original presentation of the Chow ring.  
\begin{proposition}
    For all $S\subseteq [n]$, we have $[D^S]_A = [D^{S^c}]_A$ in $A^1(\Mn)$. 
\end{proposition}

\begin{proof}
    This follows from the geometric fact that $D^S = D^{S^c}$. 
\end{proof}

\begin{proposition}[Fundamental relations, $F$-version]\label{prop:fundamental-relations-v1}
    The equality \begin{equation}\label{eq:fundamental-relations-v1}
        \Fsum_{\substack{S\subset [n] \\ i, j \in S, k, \ell \notin S}} x_S=\Fsum_{\substack{S\subset [n] \\ i, k \in S, j, \ell \notin S}} x_S=\Fsum_{\substack{S\subset [n] \\ i, \ell \in S, j, k \notin S}} x_S
    \end{equation}
    holds in $A^1(\Mn)$. 
\end{proposition}

\begin{proof}
    For $\{i,j,k,l\}\subseteq[n]$, let \[\pi_{i,j,k,l}:\Mn \to \overline{M}_{0,4}\] be the morphism forgetting all but the marked points $p_i, \dots, p_l$. Let $p_{12|34}, p_{13|24}$, and $p_{14|23}$ be the three boundary points on $\overline{M}_{0,4} \cong \PP^1$. Keel showed that there are equalities of effective Cartier divisors \[\begin{gathered}
        \pi_{i, j, k, \ell}^{-1}\left(p_{12\mid 34}\right)=\sum_{\substack{S \\ i, j \in S, k, \ell \notin S}} D^S \\
        \pi_{i, j, k, \ell}^{-1}\left(p_{13\mid 24}\right)=\sum_{\substack{S \\ i, k \in S, j, \ell \notin S}} D^S \\ 
        \pi_{i, j, k, \ell}^{-1}\left(p_{14\mid 23}\right)=\sum_{\substack{S \\ i, \ell \in S, j,k\notin S}} D^S.
    \end{gathered}\]

    We now apply the $A$-theoretic first Chern class to the corresponding line bundles for both sides of the equalities. For the left-hand sides, we have \[ [\pi_{i, j, k, \ell}^{-1}\left(p_{12\mid 34}\right)]_A= [\pi_{i, j, k, \ell}^{-1}\left(p_{13\mid 24}\right)]_A = [\pi_{i, j, k, \ell}^{-1}\left(p_{14\mid 23}\right)]_A \]
    in $A^1(\Mn)$, since all points on $\PP^1$ are linearly equivalent. For the right-hand sides by \cref{prop:divisor-sum-fgl}, we obtain 
    \[\begin{gathered}
        \bigl[\sum_{\substack{S \\ i, j \in S, k, \ell \notin S}} D^S\bigr]_A = \Fsum_{\substack{S\subset [n] \\ i, j \in S, k, \ell \notin S}} x_S \\ 
        \bigl[\sum_{\substack{S \\ i, k\in S, j, \ell \notin S}} D^S\bigr]_A = \Fsum_{\substack{S\subset [n] \\ i, k \in S, j, \ell \notin S}} x_S \\ 
        \bigl[\sum_{\substack{S \\ i, \ell \in S, j,k \notin S}} D^S\bigr]_A = \Fsum_{\substack{S\subset [n] \\ i, \ell \in S, j,k\notin S}} x_S.
    \end{gathered} \]
    Comparing the two groups of equalities yields the desired relations.
\end{proof}
Recall the definition of compatible sets from \cref{lem:boundary-divisors-facts}.

\begin{proposition}\label{prop:incompatible-divisors-relation}
For all incompatible $S$ and $T\subseteq [n]$, we have \[x_Sx_T = 0 \in A^2(\Mn).\]
    \end{proposition}

    \begin{proof}
        As mentioned in \cref{lem:boundary-divisors-facts} if $S$ and $T$ are incompatible, then $D^S\cap D^T = \varnothing$.
        Hence by \cref{prop:disjoint-divisors-zero}, we have $x_Sx_T = 0\in A^2(\Mn)$.\end{proof}




\begin{remark}
 We can simplify the fundamental relations (\Cref{prop:fundamental-relations-v1})  using the incompatibility relations. Note that the formal group law can be written as $F(x,y) = x + y + xy \cdot \Phi(x,y)$, where $\Phi$ is a power series. Therefore, if $S$ and $T$ are incompatible, $x_S x_T = 0$, which  reduces their formal sum to $x_S \Fplus x_T = x_S + x_T$.

 When fully expanding the multi-term sum $\Fsum x_S$, the result is the standard additive sum $\sum x_S$ plus a linear combination of higher-degree cross-terms of the form $x_{S_1}x_{S_2}\cdots x_{S_r}$. Such a monomial survives in $A^\bullet(\Mn)$ if and only if the sets $S_1, \dots, S_r$ are pairwise compatible. Since all sets in this specific sum contain $\{i,j\}$ and omit $\{k,\ell\}$, no two sets can be disjoint, and their union cannot be $[n]$. Thus, the only surviving higher-order cross-terms are those corresponding to strictly nested chains of subsets:
 \[ S_1 \subsetneq S_2 \subsetneq \cdots \subsetneq S_r. \]
 Consequently, the $F$-sum splits into the classical Chow sum plus a finite number of nested-chain terms.
\end{remark}

\subsection{Oriented Cohomology Ring of $\Mn$}

In this subsection, we prove that $A^\bullet(\Mn)$ admits a Keel-style presentation. We achieve this by comparing the proposed presentation, its associated graded ring with respect to filtration by powers of the augmentation ideal, and the Chow ring $CH^\bullet(\Mn)$. 

We define the free polynomial algebra
\[
S_{0,n}^A := A^\bullet(\pt)[x_S \mid S\subseteq [n],\, 2\le |S|\le n-2].
\]
Let $J\subset S_{0,n}^A$ denote the augmentation ideal generated by the variables $x_S$.

\begin{lemma}\label{lem:F-linearization}
For any $u,v\in J$ one has $u+_F v \equiv u+v \pmod{J^2}$. More generally, if $z_1,\dots,z_m\in J$, then
\[
\Fsum_{i=1}^m z_i \equiv \sum_{i=1}^m z_i \pmod{J^2}.
\]
\end{lemma}

\begin{proof}
Write $F(u,v)=u+v+\Phi(u,v)$ where every monomial appearing in $\Phi(u,v)$ has total degree at least $2$.
If $u,v\in J$, then $\Phi(u,v)\in J^2$, hence $u+_F v\equiv u+v\pmod{J^2}$.
The statement then follows by induction on $m$.
\end{proof}

Equip $S_{0,n}^A$ with the $J$-adic filtration $\mathcal F^\bullet S_{0,n}^A$ given by
\[
\mathcal F^d S_{0,n}^A := J^d.
\]
For $f\in S_{0,n}^A$, write
\[
f=\sum_{i\ge m} f_i
\]
as a sum of homogeneous components in the total degree grading, with $f_m\neq 0$.
Define the \emph{initial form} by $\operatorname{in}(f):=f_m$.
For an ideal $I\subset S_{0,n}^A$, define
\[
\operatorname{in}(I):=(\operatorname{in}(f)\mid f\in I).
\]

\begin{lemma}[Associated graded of a quotient]\label{lem:gr-quotient-initial}
Let $I\subset S_{0,n}^A$ be an ideal. Then there is a natural graded $A^\bullet(\pt)$-algebra isomorphism
\[
\operatorname{gr}_{\mathcal F}(S_{0,n}^A/I)\cong S_{0,n}^A/\operatorname{in}(I).
\]
\end{lemma}

\begin{proof}
This is a standard fact. For the induced filtration on $I$ and the quotient filtration on
$S_{0,n}^A/I$, one has
\[
\operatorname{gr}_{\mathcal F}(S_{0,n}^A/I)\cong
\operatorname{gr}_{\mathcal F}(S_{0,n}^A)/\operatorname{gr}_{\mathcal F}(I);
\]
see \cite[Tag 0120, Lemma 12.19.12]{Stacks}.
Since $\mathcal F^dS_{0,n}^A=J^d$, the associated graded ring
$\operatorname{gr}_{\mathcal F}(S_{0,n}^A)$ identifies canonically with $S_{0,n}^A$, and under
this identification $\operatorname{gr}_{\mathcal F}(I)$ is precisely $\operatorname{in}(I)$.
\end{proof}

\begin{definition}
\label{def:ring-R0n}
    Let $I_{{geom}}\subset S_{0,n}^A$ be defined by the following relations:
\begin{enumerate}
    \item (\emph{Complements}) $x_S-x_{S^c}$, for all $S$ with $2\le |S|\le n-2$,
    \item (\emph{Incompatibility}) $x_Sx_T$, whenever $S$ and $T$ are incompatible,
    \item (\emph{Nilpotency}) $x_S^{n-2}$, for all $S$ with $2\le |S|\le n-2$
\end{enumerate}
We define
\[{P_{0,n}^A}:= S_{0,n}^A/I_{geom}\]
Let us further define the ideal $I_A\subset {P_{0,n}^A} $ by the $F$-fundamental relation: for all distinct $i,j,k,\ell\in[n]$,
\[
\Fsum_{\substack{S\subset [n]\\ i,j\in S,\ k,\ell\notin S}} x_S \;-\;
\Fsum_{\substack{S\subset [n]\\ i,k\in S,\ j,\ell\notin S}} x_S,
\qquad
\Fsum_{\substack{S\subset [n]\\ i,j\in S,\ k,\ell\notin S}} x_S \;-\;
\Fsum_{\substack{S\subset [n]\\ i,\ell\in S,\ j,k\notin S}} x_S.
\]

Finally, we define
\[
R_{0,n}^A := P_{0,n}^A/I_A.
\]

\end{definition}
\begin{remark}
    Since we already imposed the nilpotency relation in $P_{0,n}^A$, the $F$-sum relations in the ideal $I_A$ are not formal power series, but rather truncations of it, i.e., polynomial relations. Hence they are well defined inside the ring.
\end{remark}
Let $I_{CH}\subset P_{0,n}^A$ be the ideal obtained by replacing each $\Fsum$ in the
$F$-fundamental relations by the ordinary Chow summation $\sum$. We define
\[
\widetilde{R}_{0,n}^{CH} := P_{0,n}^A/I_{CH}.
\] 
Note that $\widetilde{R}_{0,n}^{CH} \cong CH^\bullet(\Mn) \otimes_\ZZ A^\bullet(\pt)$ by Keel's presentation of $CH^\bullet(\Mn)$ \cite[Theorem 1]{Keel}, since we have dimension axiom for Chow ring.
\begin{proposition}\label{prop:gr-is-additive}
With respect to the $J$-adic filtration, there is a natural graded $A^\bullet(\pt)$-algebra isomorphism
\[
\operatorname{gr}_{\mathcal F}(R_{0,n}^A)\cong \widetilde{R}_{0,n}^{CH}.
\]
\end{proposition}

\begin{proof}
By \Cref{lem:gr-quotient-initial}, it suffices to identify $\operatorname{in}(I_A)$.
Complement relations are linear, hence equal to their own initial forms.
Incompatibility relations are quadratic, hence already homogeneous of degree $2$.
For an $F$-fundamental relation, \Cref{lem:F-linearization} shows that each $F$-sum
agrees with the corresponding ordinary sum modulo $J^2$. Hence the initial form of
an $F$-fundamental relation is the corresponding additive fundamental relation.
Therefore $\operatorname{in}(I_A)=I_{CH}$, which proves the claim.
\end{proof}

The following lemma is well-known.
\begin{lemma}[Lifting homogeneous bases]\label{lem:basis-lift}
Let $M$ be an $A^\bullet(\pt)$-module equipped with an exhaustive separated filtration
$\mathcal F^\bullet M$.
Assume $\operatorname{gr}_{\mathcal F}(M)$ is a finite free $A^\bullet(\pt)$-module with homogeneous basis
$\{\bar b_\lambda\}_\lambda$.
For each $\lambda$, choose a lift $b_\lambda\in \mathcal F^{d_\lambda}M$ of $\bar b_\lambda$,
where $d_\lambda$ is the degree of $\bar b_\lambda$.
Then $\{b_\lambda\}_\lambda$ is an $A^\bullet(\pt)$-basis of $M$.
In particular, $M$ is finite free over $A^\bullet(\pt)$.
\end{lemma}

\begin{proof}
Let $m\in M$. Choose $d$ minimal with $m\in\mathcal F^dM$.
The image of $m$ in $\mathcal F^dM/\mathcal F^{d+1}M$ is an $A^\bullet(\pt)$-linear combination of the
$\bar b_\lambda$ of degree $d$.
Subtracting the corresponding combination of the lifts $b_\lambda$ yields an element of
$\mathcal F^{d+1}M$.
Repeating this argument expresses $m$ as an $A^\bullet(\pt)$-linear combination of the $b_\lambda$.
Since $\operatorname{gr}_{\mathcal F}(M)$ has a finite homogeneous basis, only finitely many graded
pieces are nonzero, so this process terminates.

For linear independence, suppose that
\(
\sum_\lambda a_\lambda b_\lambda=0.
\)
Let $d$ be the minimal filtration degree among indices with $a_\lambda\ne 0$.
Reducing modulo $\mathcal F^{d+1}M$ gives a nontrivial $A^\bullet(\pt)$-linear relation among the
degree-$d$ basis elements $\bar b_\lambda$ in $\operatorname{gr}_{\mathcal F}(M)$, a contradiction.
Hence all $a_\lambda=0$.
\end{proof}
\begin{corollary}[Freeness of $R_{0,n}^A$]\label{cor:freeness-of-R0nF}
The $A^\bullet(\pt)$-algebra $R_{0,n}^A$ is finite free over $A^\bullet(\pt)$, and the same set of monomials maps to an $A^\bullet(\pt)$-basis of $R_{0,n}^A$.
Moreover,
\[
\rk_{A^\bullet(\pt)}(R_{0,n}^A)=\rk_{A^\bullet(\pt)}(\widetilde{R}_{0,n}^{CH}).
\]
\end{corollary}

\begin{proof}
The ring $\widetilde{R}_{0,n}^{CH}$ is finite free over $A^\bullet(\pt)$ with homogeneous basis given by monomials in $x_S$. By Proposition~\ref{prop:gr-is-additive}, $\operatorname{gr}_{\mathcal F}(R_{0,n}^A)\cong \widetilde{R}_{0,n}^{CH}$.
Apply \Cref{lem:basis-lift} to $M=R_{0,n}^A$, lifting the homogeneous standard monomial basis from the
associated graded to $R_{0,n}^A$.
\end{proof}

Now, we are ready to prove \Cref{thm:oriented-cohomology-of-M0nbar}.

\begin{theorem}[Oriented Cohomology of $\Mn$]\label{thm:oriented-cohomology-of-M0nbar}
    Let $A^\bullet$ be an oriented cohomology theory, characterized by the formal group law $F$, satisfying the dimension axiom. Let $R_{0,n}^A$ be as in \cref{def:ring-R0n}.
    There is an isomorphism of graded rings \begin{equation}
        A^\bullet(\Mn)\cong R_{0,n}^A
    \end{equation}
    sending the fundamental class of each boundary divisor $[D^S]_A\in A^1(\Mn)$ to the generator $x_S$. 
\end{theorem}

\begin{proof}
Let $\phi:S_{0,n}^A\to A^\bullet(\overline{M}_{0,n})$ be the $A^\bullet(\pt)$-algebra homomorphism sending $x_S$ to $[D^S]_A$.
The complement, incompatibility, $F$-fundamental relations hold for $A^\bullet(\Mn)$, as shown in \cref{sec:relation-A(M0n)} and the nilpotency relation holds by the assumption that $A^\bullet$ satisfies dimension axiom.
Hence by \cref{prop:generators-for-M0nbar}, we have that $\phi$ factors through a surjective homomorphism
\[
\bar\phi:R_{0,n}^A \twoheadrightarrow A^\bullet(\overline{M}_{0,n}).
\]

By the Blowup Presentation of $\overline{M}_{0,n}$ and the projective bundle and blowup formulas in $A^\bullet$,
the $A$-module $A^\bullet(\overline{M}_{0,n})$ is finite free, with rank equal to that of the Chow ring
$CH^\bullet(\overline{M}_{0,n})\otimes_{\ZZ}A^\bullet(\pt)$ due to \Cref{prop:rank-equals-chow-rank}.
On the other hand, \Cref{cor:freeness-of-R0nF} shows that $R_{0,n}^A$ is finite free over $A$, with
\[
\rk_{A^\bullet(\pt)}(R_{0,n}^A)=\rk_{A^\bullet(\pt)}(\widetilde{R}_{0,n}^{CH}).
\]
In the additive case, $\widetilde{R}_{0,n}^{CH}$ identifies with the Keel presentation over $\ZZ$ after base change
to $A^\bullet(\pt)$, hence $\rk_{A^\bullet(\pt)}(\widetilde{R}_{0,n}^{CH})=\rk_{A^\bullet(\pt)}(CH^\bullet(\overline{M}_{0,n})\otimes_{\ZZ}A^\bullet(\pt))$.
Therefore, $\bar\phi$ is a surjective homomorphism between finite free $A^\bullet(\pt)$-modules of the same rank, say $r$. Taking the top exterior power yields a surjective homomorphism $\bigwedge^r \bar\phi: \bigwedge^r R_{0,n}^A \twoheadrightarrow \bigwedge^r A^\bullet(\Mn)$ between free $A^\bullet(\pt)$-modules of rank $1$. Because this induced map is given by multiplication by the determinant of $\bar\phi$, surjectivity forces the determinant to be a unit in $A^\bullet(\pt)$. Consequently, $\bar\phi$ is an isomorphism of graded $A^\bullet(\pt)$-modules, and in particular bijective. Thus we get an isomorphism of graded rings.
\end{proof}

\begin{remark}
    Notice that the $F$-fundamental relations truncates to polynomial relations only if $A^\bullet$ satisfies the dimension axiom. Nevertheless, because algebraic cobordism satisfies the dimenison axiom by definition, \Cref{thm:oriented-cohomology-of-M0nbar} specializes and gives a presentation for $\Omega^\bullet(\Mn)$. We thus recover presentations of $A^\bullet(\Mn)$ for any Landweber exact theory $A^\bullet$ (\Cref{eg:landweber-exact-theories}). In particular, $\Omega^\bullet(\Mn) \otimes_\mathbb{L}\ZZ$ by sending all $a_{ij}$ to $0$ recovers Keel's presentation of the Chow ring of $\Mn$. Tensoring with $\ZZ[\beta, \beta^{-1}]$ recovers by setting $a_{11} \mapsto \beta$ and all other $a_{ij}$ to $0$ yields a presentation for $K_0(\Mn)[\beta, \beta^{-1}]$. Specializing to the degree $0$ component by setting $\beta$ to $-1$ recovers the Grothendieck group $K_0(\Mn)$ in \cite{LLPP2024WonderfulVarieties}.
\end{remark}

\begin{example}[Algebraic cobordism of $\overline{M}_{0,5}$]\label{eg:Cobordism-of-M05-bar}
Let $A^\bullet=\Omega^\bullet$ be algebraic cobordism, with universal formal group law
\[
F_{\mathrm{univ}}(u,v)=u+v+\sum_{i,j\ge 1} a_{ij}u^iv^j,
\qquad a_{ij}\in \mathbb{L}=\Omega^\bullet(\mathrm{pt}).
\]
Since $\overline{M}_{0,5}$ is a surface, the dimension axiom implies that every term of total degree at least $3$ vanishes. Hence
\[
F(u,v)=u+v+a_{11}uv
\qquad\text{in } \Omega^\bullet(\overline{M}_{0,5}).
\]

Modulo relations (1), namely complements, $\Omega^\bullet(\overline{M}_{0,5})$ is the polynomial algebra over $\mathbb{L}$ generated by
\[
x_{15}, x_{25}, x_{35}, x_{45},
\qquad
x_{125}, x_{135}, x_{145}, x_{235}, x_{245}, x_{345},
\]
where $x_S=[D^S]_\Omega\in \Omega^1(\overline{M}_{0,5})$.
The incompatibility relations are as follows.
\begin{itemize}
    \item (size 2 \& size 2) \[x_{15} x_{25}=x_{15} x_{35}=x_{15} x_{45}=x_{25} x_{35}=x_{25} x_{45}=x_{35} x_{45}=0 .\]
    \item (size 2 \& size 3) \[\begin{aligned}
& x_{15} x_{235}=x_{15} x_{245}=x_{15} x_{345}=0, \\
& x_{25} x_{135}=x_{25} x_{145}=x_{25} x_{345}=0, \\
& x_{35} x_{125}=x_{35} x_{145}=x_{35} x_{245}=0, \\
& x_{45} x_{125}=x_{45} x_{135}=x_{45} x_{235}=0 .
\end{aligned}\]
\item (size 3 \& size 3) \[\begin{aligned}
& x_{125} x_{135}=x_{125} x_{145}=x_{125} x_{235}=x_{125} x_{245}=x_{125} x_{345}=0, \\
& x_{135} x_{145}=x_{135} x_{235}=x_{135} x_{245}=x_{135} x_{345}=0, \\
& x_{145} x_{235}=x_{145} x_{245}=x_{145} x_{345}=0, \\
& x_{235} x_{245}=x_{235} x_{345}=0, \\
& x_{245} x_{345}=0 .
\end{aligned}\]
\end{itemize} 
The fundamental relations are
\begin{itemize}
    \item For $(i,j,k,\ell)=(1,2,3,4)$,
    \[
    x_{125}+x_{345}+a_{11}x_{125}x_{345}
    =
    x_{135}+x_{245}+a_{11}x_{135}x_{245}
    =
    x_{145}+x_{235}+a_{11}x_{145}x_{235}.
    \]

    \item For $(i,j,k,\ell)=(1,2,3,5)$,
    \[
    x_{145}+x_{15}+a_{11}x_{145}x_{15}
    =
    x_{245}+x_{25}+a_{11}x_{245}x_{25}
    =
    x_{345}+x_{35}+a_{11}x_{345}x_{35}.
    \]

    \item For $(i,j,k,\ell)=(1,2,4,5)$,
    \[
    x_{135}+x_{15}+a_{11}x_{135}x_{15}
    =
    x_{235}+x_{25}+a_{11}x_{235}x_{25}
    =
    x_{345}+x_{45}+a_{11}x_{345}x_{45}.
    \]

    \item For $(i,j,k,\ell)=(1,3,4,5)$,
    \[
    x_{125}+x_{15}+a_{11}x_{125}x_{15}
    =
    x_{235}+x_{35}+a_{11}x_{235}x_{35}
    =
    x_{245}+x_{45}+a_{11}x_{245}x_{45}.
    \]

    \item For $(i,j,k,\ell)=(2,3,4,5)$,
    \[
    x_{125}+x_{25}+a_{11}x_{125}x_{25}
    =
    x_{135}+x_{35}+a_{11}x_{135}x_{35}
    =
    x_{145}+x_{45}+a_{11}x_{145}x_{45}.
    \]
\end{itemize}

Thus
\[
\Omega^\bullet(\overline{M}_{0,5})
\cong
\frac{\mathbb{L}[x_{15},x_{25},\dots,x_{345}]}{\text{complements, incompatibility, and fundamental relations}}.
\]
Tensoring with $\ZZ$ by sending $a_{11}$ to $0$, we recover Keel's presentation of the Chow ring. Tensoring with $\ZZ[\beta^{\pm1}]$ by sending $a_{11}\mapsto \beta$ and then specializing to $\beta = -1$, we recover a presentation of $K_0(\overline{M}_{0,5})$ which has appeared in \cite{LLPP2024WonderfulVarieties}. 
\end{example}

\begin{example}[Algebraic cobordism of $\overline{M}_{0,6}$]\label{eg:Cobordism-of-M06-bar}
As in \Cref{eg:Cobordism-of-M05-bar}, the universal formal group law truncates on the threefold $\overline{M}_{0,6}$ to
\[
F(u,v)=u+v+a_{11}uv+a_{21}u^2v+a_{12}uv^2,
\qquad a_{12}=a_{21}.
\]
Modulo complements, the generators are
\begin{itemize}
    \item size $2$
    \[
    x_{16}, x_{26}, x_{36}, x_{46}, x_{56},
    \]

    \item size $3$
    \[
    x_{126}, x_{136}, x_{146}, x_{156}, x_{236}, x_{246}, x_{256}, x_{346}, x_{356}, x_{456},
    \]

    \item size $4$
    \[
    x_{1236}, x_{1246}, x_{1256}, x_{1346}, x_{1356}, x_{1456},
    x_{2346}, x_{2356}, x_{2456}, x_{3456}.
    \]
\end{itemize}
These satisfy the complement relations, the incompatibility relations, and the fundamental relations. For brevity, we only display only the fundamental relation for $(i,j,k,\ell)=(1,2,3,4)$.
\begin{footnotesize}
\begin{align*}
&
x_{3456}+x_{346}+x_{126}+x_{1256} \\
&\quad
+a_{11}\bigl(
x_{3456}x_{346}
+x_{3456}x_{126}
+x_{3456}x_{1256}
+x_{346}x_{1256}
+x_{126}x_{1256}
\bigr) \\
&\quad
+a_{21}\bigl(
x_{3456}^2x_{346}
+x_{3456}x_{346}^2
+x_{3456}^2x_{126}
+x_{3456}x_{126}^2
+x_{3456}^2x_{1256}
+x_{3456}x_{1256}^2 \\
&\qquad\qquad\qquad
+x_{346}^2x_{1256}
+x_{346}x_{1256}^2
+x_{126}^2x_{1256}
+x_{126}x_{1256}^2
+2x_{3456}x_{346}x_{1256}
+2x_{3456}x_{126}x_{1256}
\bigr) \\
&\quad
+a_{11}^2\bigl(
x_{3456}x_{346}x_{1256}
+x_{3456}x_{126}x_{1256}
\bigr) \\
=\;&
x_{2456}+x_{246}+x_{136}+x_{1356} \\
&\quad
+a_{11}\bigl(
x_{2456}x_{246}
+x_{2456}x_{136}
+x_{2456}x_{1356}
+x_{246}x_{1356}
+x_{136}x_{1356}
\bigr) \\
&\quad
+a_{21}\bigl(
x_{2456}^2x_{246}
+x_{2456}x_{246}^2
+x_{2456}^2x_{136}
+x_{2456}x_{136}^2
+x_{2456}^2x_{1356}
+x_{2456}x_{1356}^2 \\
&\qquad\qquad\qquad
+x_{246}^2x_{1356}
+x_{246}x_{1356}^2
+x_{136}^2x_{1356}
+x_{136}x_{1356}^2
+2x_{2456}x_{246}x_{1356}
+2x_{2456}x_{136}x_{1356}
\bigr) \\
&\quad
+a_{11}^2\bigl(
x_{2456}x_{246}x_{1356}
+x_{2456}x_{136}x_{1356}
\bigr) \\
=\;&
x_{2356}+x_{236}+x_{146}+x_{1456} \\
&\quad
+a_{11}\bigl(
x_{2356}x_{236}
+x_{2356}x_{146}
+x_{2356}x_{1456}
+x_{236}x_{1456}
+x_{146}x_{1456}
\bigr) \\
&\quad
+a_{21}\bigl(
x_{2356}^2x_{236}
+x_{2356}x_{236}^2
+x_{2356}^2x_{146}
+x_{2356}x_{146}^2
+x_{2356}^2x_{1456}
+x_{2356}x_{1456}^2 \\
&\qquad\qquad\qquad
+x_{236}^2x_{1456}
+x_{236}x_{1456}^2
+x_{146}^2x_{1456}
+x_{146}x_{1456}^2
+2x_{2356}x_{236}x_{1456}
+2x_{2356}x_{146}x_{1456}
\bigr) \\
&\quad
+a_{11}^2\bigl(
x_{2356}x_{236}x_{1456}
+x_{2356}x_{146}x_{1456}
\bigr).
\end{align*}
\end{footnotesize}
\end{example}

\subsection{Comparison between different ring descriptions}

In this final section we compare the oriented cohomology ring description of $\Mn$ which we obtained in \cref{thm:oriented-cohomology-of-M0nbar} with the various Chow and $K$-theoretic ring descriptions appearing in literature, and in particular the ring description in \cite{LLPP2024WonderfulVarieties}.
\begin{remark}
    In \cite{LLPP2024WonderfulVarieties} it is established that there is a ring isomorphism between $K_0(\Mn) \to CH^\bullet(\Mn)$. This implies that with our notation from \cref{thm:oriented-cohomology-of-M0nbar}, $R_{0,n}^{K_0}\cong R_{0,n}^{CH}$. However, in our set up this is only an abstract isomorphism. In fact, for $n \geq 6$, we show that the naive bijection between the boundary divisors do not produce an isomorphism.
\end{remark}
\begin{lemma}[The naive map is not a ring homomorphism]
\label{lem:naive-map-fails}
While $K_0(\Mn)$ and $CH^\bullet(\Mn)$ are abstractly isomorphic as rings, the naive bijection of boundary divisors 
\[ \phi: K_0(\Mn) \to CH^\bullet(\Mn), \quad \phi(x_S^K) = x_S^{CH} \]
does not extend to a ring homomorphism for $n \ge 6$, where $x_S^K$ is the $K$-theoretic fundamental class and $x_S^{CH}$ is the Chow theoretic fundamental class.
\end{lemma}

\begin{proof}
We demonstrate this by showing that $\phi$ fails to preserve the relations of the $K_0$-ring for any $n \ge 6$. We consider the fundamental relations in $\Mn$ for the 4-tuple $(i,j,k,l) = (1,2,3,4)$. 

Let \( A = \{S \subset [n] \mid 2\le |S| \le n-2, \; 1,2 \in S, \text{ and } 3,4 \notin S\}\) and \( B = \{T \subset [n] \mid 2\le |T| \le n-2, \; 1,3 \in T, \text{ and } 2,4 \notin T\} \). Note that we identify $x_S$ and $x_{S^c}$ to avoid double-counting.

In $K_0(\Mn)$, the fundamental relation states that $\Fsum_{S\in A} x_S^K = \Fsum_{T\in B} x_T^K$. Using the $K_0$-formal group law $x \Fplus y = x + y - xy$, this equality translates to the product identity:
\[ 1 - \prod_{S \in A} (1 - x_S^K) = 1 - \prod_{T \in B} (1 - x_T^K). \]
Expanding these products yields graded components. The linear terms identically match the classical Chow relation $\sum_{S \in A} x_S^K = \sum_{T \in B} x_T^K$. If $\phi$ were a well-defined ring homomorphism, it would map the degree-2 terms identically. The negative of the degree-2 terms must therefore be equal, meaning the following equation must hold in $CH^2(\Mn)$:
\[ Q_A := \sum_{\substack{S, S' \in A \\ S < S'}} x_S^{CH} x_{S'}^{CH} \;\; = \;\; \sum_{\substack{T, T' \in B \\ T < T'}} x_T^{CH} x_{T'}^{CH} =: Q_B. \]

We test this equality by multiplying both sides by $x_{34}^{CH}$. For the right-hand side, consider any index set $T \in B$. 
Because $n\ge 6$, $\{3,4\}$ is incompatible with the index set of every element in $B$, so $x_{34}^{CH} \cdot x_T^{CH} = 0$ for all $T \in B$. This forces the entire intersection product $Q_B \cdot x_{34}^{CH} = 0$.

For the left-hand side, because $n \ge 6$, the set $A$ corresponds to at least the divisors $x_{12}^{CH}$ and $x_{125}^{CH}$. Furthermore, since $n \ge 6$, the complement of $\{1,2,5\}$ has size at least 3, meaning $x_{125}^{CH}$ cannot be the same divisor as $x_{34}^{CH}$ (whose complement has size $n-2 \ge 4$). 

The subsets $\{1,2\}$ and $\{1,2,5\}$ are compatible because $\{1,2\} \subset \{1,2,5\}$.  Since both are disjoint from $\{3,4\}$, all three are mutually compatible sets. Therefore, the product $Q_A \cdot x_{34}^{CH}$ contains the monomial $x_{12}^{CH} \cdot x_{125}^{CH} \cdot x_{34}^{CH}$. Because the boundary divisors have simple normal crossings, this monomial is exactly the fundamental class of the geometric intersection $D^{12} \cap D^{125} \cap D^{34}$. For $n \ge 6$, this intersection corresponds to a non-empty codimension-$3$ boundary stratum. Since the fundamental class of a non-empty closed subvariety in a smooth projective variety is non-zero in the Chow ring, this monomial is non-zero. Because $Q_A \cdot x_{34}^{CH}$ is a sum of such fundamental classes with positive integer coefficients, it cannot evaluate to zero, so $Q_A \cdot x_{34}^{CH} \neq 0$ in $CH^3(\Mn)$.

We now proceed by contradiction. If $\phi$ were a well-defined ring homomorphism, we would have $Q_A = Q_B$ in $CH^2(\Mn)$. Multiplying both sides of this equation by the class $x_{34}^{CH}$ must preserve the equality, forcing $Q_A \cdot x_{34}^{CH} = Q_B \cdot x_{34}^{CH}$ in $CH^3(\Mn)$. However, we have just established that $Q_A \cdot x_{34}^{CH} \neq 0$ while $Q_B \cdot x_{34}^{CH} = 0$. This is a contradiction. Therefore, $Q_A \neq Q_B$ in $CH^2(\Mn)$, and $\phi$ cannot be a ring homomorphism.
\end{proof}

\begin{remark} [On exceptional isomorphisms]
      Geometrically, $\overline{M}_{0,5}$ (which Keel describes as the blowup of $\mathbb P^1 \times \mathbb P^1$ at $(0,0), (1,1) \text{ and } (\infty,\infty)$) is isomorphic to the blowup of $\mathbb{P}^2$ at $4$ points in general position. Hence from \cref{rmk: delpezzo exceptional}, we see that in this specific case, we have an exceptional isomorphism (following the terminology of \cite{LLPP2024WonderfulVarieties}) from any oriented cohomology theory to Chow theory; i.e., $A^\bullet(\overline{M}_{0,5}) \cong CH^\bullet(\overline{M}_{0,5})\otimes_{\mathbb Z}A^\bullet(\pt)$. In the following theorem, we give a change of basis for the two different presentations. We remark that this is similar to the change of bases between the Feichtner--Yuzvinsky presentation and the simplicial presentation of the Chow ring and $K_0$-ring of wonderful varieties---see for example \cite{LLPP2024WonderfulVarieties}.
\end{remark}

\begin{theorem}[Change of basis isomorphism for $A^\bullet(\overline{M}_{0,5}) \cong A^\bullet(dP_5)$]
Let $A^\bullet$ be an oriented cohomology theory satisfying the dimension axiom. The geometric identification of $\overline{M}_{0,5}$ with the degree $5$ del Pezzo surface induces an explicit change of basis isomorphism between the ring presentations of \Cref{thm:oriented-cohomology-of-M0nbar} and \Cref{prop:del-pezzo-oriented-rings}, defined on the generators by:
\begin{align*}
    x_{i5} &\mapsto e_i \quad \text{for } i \in \{1, 2, 3, 4\}, \\
    x_{ij} &\mapsto h \Fplus \chi_A(e_k) \Fplus \chi_A(e_l) \quad \text{for } 1 \le i < j \le 4,
\end{align*}
where $\{k, l\}$ is the unique complement such that $\{i, j, k, l\} = \{1, 2, 3, 4\}$, and $\chi_A$ is the formal inverse.
\end{theorem}

\begin{proof}
    We identify $\overline{M}_{0,5} \cong dP_5$. We establish the isomorphism of their oriented cohomology rings by explicitly constructing a well-defined, surjective $A^\bullet(\pt)$-algebra homomorphism $\Phi: A^\bullet(\overline{M}_{0,5}) \to A^\bullet(dP_5)$.

    Recall that $A^\bullet(dP_5)$ is generated by $h, e_1, e_2, e_3, e_4$. For $\overline{M}_{0,5}$, the $10$ boundary divisors are indexed by subsets of size $2$. We define $\Phi$ on the generators of $A^\bullet(\overline{M}_{0,5})$ as follows.
    \begin{align*}
        \Phi(x_{i5}) &= e_i \quad \text{for } i \in \{1, 2, 3, 4\}, \\
        \Phi(x_{ij}) &= h \Fplus \chi_A(e_k) \Fplus \chi_A(e_l) \quad \text{for } 1 \le i < j \le 4,
    \end{align*}
    where $\{k, l\}$ is the unique complement such that $\{i, j, k, l\} = \{1, 2, 3, 4\}$, and $\chi_A$ is the formal inverse.

    To show $\Phi$ is a well-defined ring homomorphism, we must verify that it preserves the incompatibility and fundamental relations of $\overline{M}_{0,5}$. 

    \textit{1. Incompatibility relations:} We check the three cases of disjoint index sets.
    \begin{enumerate}[(a)]
        \item $x_{i5} x_{j5} = 0$: Applying $\Phi$ yields $e_i e_j$, which is $0$ in $A^\bullet(dP_5)$ for $i \neq j$.
        \item $x_{i5} x_{ik} = 0$ (where $i, j, k \in \{1,2,3,4\}$ are distinct): Applying $\Phi$ yields 
        \[ e_i \cdot \bigl(h \Fplus \chi_A(e_j) \Fplus \chi_A(e_l)\bigr). \]
        Because $e_i h = 0$, $e_i e_j = 0$, and $e_i e_l = 0$, the entire expression vanishes.
        \item $x_{ij} x_{il} = 0$ (where $\{i,j,k,l\} = \{1,2,3,4\}$): Applying $\Phi$ yields
        \[ \bigl(h \Fplus \chi_A(e_k) \Fplus \chi_A(e_l)\bigr) \cdot \bigl(h \Fplus \chi_A(e_j) \Fplus \chi_A(e_k)\bigr). \]
        Because the relations of $A^\bullet(dP_5)$ dictate that $h^3 = 0$, $e_m h = 0$, and $e_m^2 = -h^2$, we can deduce that $e_m^3 = e_m(-h^2) = -h(e_m h) = 0$. Since $e_m^3 = 0$, the formal inverse truncates to $\chi_A(e_m) = -e_m + a_{11}e_m^2$. Furthermore, because $h \cdot e_m = 0$ and $e_m e_n = 0$, we have $\chi_A(e_k)^2 = (-e_k + a_{11}e_k^2)^2 = e_k^2 - 2a_{11}e_k^3 + a_{11}^2 e_k^4 = e_k^2$. Thus, the product simplifies to $h^2 + e_k^2$, which is zero.
    \end{enumerate}

    \textit{2. Fundamental relations:} We check the two classes of $4$-tuples. For the quadruple $(1, 2, 3, 4)$, the fundamental relation is $x_{12} \Fplus x_{34} = x_{13} \Fplus x_{24} = x_{14} \Fplus x_{23}$. Applying $\Phi$ to the first term yields
        \[ \Phi(x_{12} \Fplus x_{34}) = \bigl(h \Fplus \chi_A(e_3) \Fplus \chi_A(e_4)\bigr) \Fplus \bigl(h \Fplus \chi_A(e_1) \Fplus \chi_A(e_2)\bigr) =[2]_F h \Fplus \Fsum_{m=1}^4 \chi_A(e_m). \]
    This expression is symmetric with respect to the indices $\{1, 2, 3, 4\}$, so $\Phi(x_{13} \Fplus x_{24})$ and $\Phi(x_{14} \Fplus x_{23})$ evaluate to the same element.

    For a quadruple containing $5$, such as $(1, 2, 3, 5)$, the fundamental relation is $x_{12} \Fplus x_{35} = x_{13} \Fplus x_{25} = x_{23} \Fplus x_{15}$. Applying $\Phi$ to the first term yields
        \[ \Phi(x_{12} \Fplus x_{35}) = \bigl(h \Fplus \chi_A(e_3) \Fplus \chi_A(e_4)\bigr) \Fplus e_3 =h \Fplus \chi_A(e_4). \]
    It is easy to verify that $\Phi(x_{13} \Fplus x_{25})$ and $\Phi(x_{23} \Fplus x_{15})$ both equals $h \Fplus \chi_A(e_4)$.

    Since $\Phi$ maps generators to generators and preserves all relations, it is a well-defined $A^\bullet(\pt)$-algebra homomorphism. Its image trivially contains $e_1, \dots, e_4$. Furthermore, since $x_{12} \Fplus x_{35} \mapsto h \Fplus \chi_A(e_4)$, $h$ is also in the image, proving $\Phi$ is surjective. 
    
    Finally, both $A^\bullet(\overline{M}_{0,5})$ and $A^\bullet(dP_5)$ are finite free $A^\bullet(\pt)$-modules of rank $7$, as determined by the additive blowup decomposition. By the same commutative algebra argument used in \Cref{thm:oriented-cohomology-of-M0nbar}, any surjective module homomorphism between two finite free modules of the same rank is necessarily an isomorphism. Thus, $\Phi$ is a ring isomorphism.
\end{proof}


\printbibliography

\begin{appendices}

\addtocontents{toc}{\protect\renewcommand{\protect\cftsecpresnum}{Appendix }}
\addtocontents{toc}{\protect\setlength{\cftsecnumwidth}{6em}} 

\section{Secant Lines to the Twisted Cubic Meeting Two General Lines}
\label{apx:4-secant-lines}

In this appendix, we review the geometry of the blowup $Bl_C\PP^3$ and outline a solution to the problem of counting secant lines to $C$ meeting two general lines in $\PP^3$. Recall that $C\cong \PP^1 \xrightarrow{\nu_3}\PP^3$ is the twisted cubic in $\PP^3$, which is the image of the embedding via the complete linear system $|\O_{\PP^1}(3)|$. In coordinates, this is \[(s:t)\mapsto(s^3:s^2t:st^2:t^3).\] Let $Bl_C\PP^3 $ denote the blowup of $\PP^3$ along $C$. We briefly recap the following moduli-theoretic interpretation of $Bl_C\PP^3$, which is well-known. 

Let $\mathcal I_C\subseteq \O_{\PP^3}$ be the ideal sheaf of $C$. Since $C$ is cut out by the $2\times 2$ minors of the matrix
\[
\begin{pmatrix}
x_0 & x_1 & x_2\\
x_1 & x_2 & x_3
\end{pmatrix},
\]
its homogeneous ideal is generated by three independent quadrics. Hence, $h^0(\PP^3,\mathcal I_C(2))=3$, and 
\[
\PP\big(H^0(\PP^3,\mathcal I_C(2))\big)\ \cong\ \PP^2.
\]
The associated rational map $\varphi \colon \PP^3 \dashrightarrow \PP^2$ is given by evaluating a basis of quadrics through $C$ and is undefined exactly along $C$. The blowup $Bl_C\PP^3$ can be identified with the closure of the graph of $\varphi$ inside $\PP^3\times \PP^2$, and in particular carries a natural morphism
\[
f\colon Bl_C\PP^3 \to \PP^2 \cong |\mathcal I_C(2)|.
\]

Under the standard identification $\operatorname{Hilb}^2(C)\cong \Sym^2(C)\cong \PP^2$, a point of $\PP^2$ corresponds to a length-$2$ subscheme $\xi\subset C$. If $\xi = p\cup q$ where $p\neq q$, then it corresponds to the secant line $L_\xi :=\overline{pq}$ to $C$. If $\xi = 2p$, then it corresponds to the tangent line to $C$ at $p$. Since $C$ is the rational normal curve of $\PP^3$, no $3$ points on $C$ can be collinear, so a length-$2$ subscheme of $C$ uniquely determines a secant line, and vice versa. Therefore, there is a canonical identification \[\PP^2 \cong \operatorname{Hilb}^2(C) \cong \operatorname{Sect}(C), \quad \xi \mapsto L_\xi ,\]
where $\operatorname{Sect}(C)$ denotes the secant variety of $C\subseteq \PP^3$. Closed points in $\operatorname{Sect}(C)$ correspond to secant/tangent lines to $C$. Because the fiber of $f$ over $\xi$ is the proper transform of $L_\xi$, $Bl_C\PP^3$ is the universal family of secant lines to $C$. It parametrizes pairs $(\xi,x)$ where $\xi\in \operatorname{Hilb}^2(C)$ and $x\in L_\xi$.

We can use the above moduli interpretation of $Bl_C\PP^3$ to solve the following enumerative problem. 

\begin{quote}
   \begin{itemize}
       \item[$\dagger$]  How many secant lines to $C$ meet two general lines in $\PP^3$?
   \end{itemize}
\end{quote}
\begin{figure}[ht]
    \centering
    \centering
\begin{tikzpicture}[scale=1.24,line cap=round,line join=round]

\colorlet{mainblue}{blue!45!black}
\colorlet{secantred}{red!60!black}

\coordinate (P1) at (-5.40,0.00);
\coordinate (P2) at (-4.20,1.20);
\coordinate (P3) at (-3.60,2.15);
\coordinate (P4) at (-3.00,2.85);

\coordinate (Q1) at ( 4.45,-0.08);
\coordinate (Q2) at ( 4.95, 1.00);
\coordinate (Q3) at ( 5.05, 2.00);
\coordinate (Q4) at ( 5.15, 3.05);

\draw[secantred,line width=1.45pt] (P1) -- (Q1);
\draw[secantred,line width=1.45pt] (P2) -- (Q2);
\draw[secantred,line width=1.45pt] (P3) -- (Q3);
\draw[secantred,line width=1.45pt] (P4) -- (Q4);

\draw[
    mainblue,
    line width=2.15pt
]
plot[smooth] coordinates {
    (-5.85,-0.12)
    (-5.60,-0.05)
    (-5.40, 0.00)   
    (-4.95, 0.28)
    (-4.20, 1.20)   
    (-3.95, 1.55)
    (-3.60, 2.15)   
    (-3.30, 2.52)
    (-3.00, 2.85)   
    (-2.20, 3.30)
    (-1.10, 3.52)
    ( 0.40, 3.30)
    ( 1.70, 2.75)
    ( 2.80, 1.95)
    ( 3.70, 0.95)
    ( 4.45,-0.08)   
    ( 4.95, 1.00)   
    ( 5.05, 2.00)   
    ( 5.15, 3.05)   
    ( 5.38, 4.10)
    ( 5.72, 5.18)
};

\draw[
    mainblue,
    line width=2.15pt
]
(-4.20,-0.80) -- (-1.00,3.20);

\draw[
    mainblue,
    line width=2.15pt
]
( 2.60,-0.30) -- ( 4.60,4.20);

\fill[secantred] (P1) circle (3.5pt);
\fill[secantred] (P2) circle (3.5pt);
\fill[secantred] (P3) circle (3.5pt);
\fill[secantred] (P4) circle (3.5pt);

\fill[secantred] (Q1) circle (3.5pt);
\fill[secantred] (Q2) circle (3.5pt);
\fill[secantred] (Q3) circle (3.5pt);
\fill[secantred] (Q4) circle (3.5pt);

\fill[mainblue] (-3.572,-0.015) circle (3.2pt);
\fill[mainblue] (-2.627, 1.166) circle (3.2pt);
\fill[mainblue] (-1.864, 2.120) circle (3.2pt);
\fill[mainblue] (-1.246, 2.893) circle (3.2pt);

\fill[mainblue] ( 2.706,-0.066) circle (3.2pt);
\fill[mainblue] ( 3.208, 1.038) circle (3.2pt);
\fill[mainblue] ( 3.656, 2.024) circle (3.2pt);
\fill[mainblue] ( 4.111, 3.024) circle (3.2pt);

\end{tikzpicture}
    \caption{A schematic diagram of the $4$ secant lines (red) meeting two general lines (blue) in $\PP^3$.}
    \label{fig:4-secant-lines}
\end{figure}
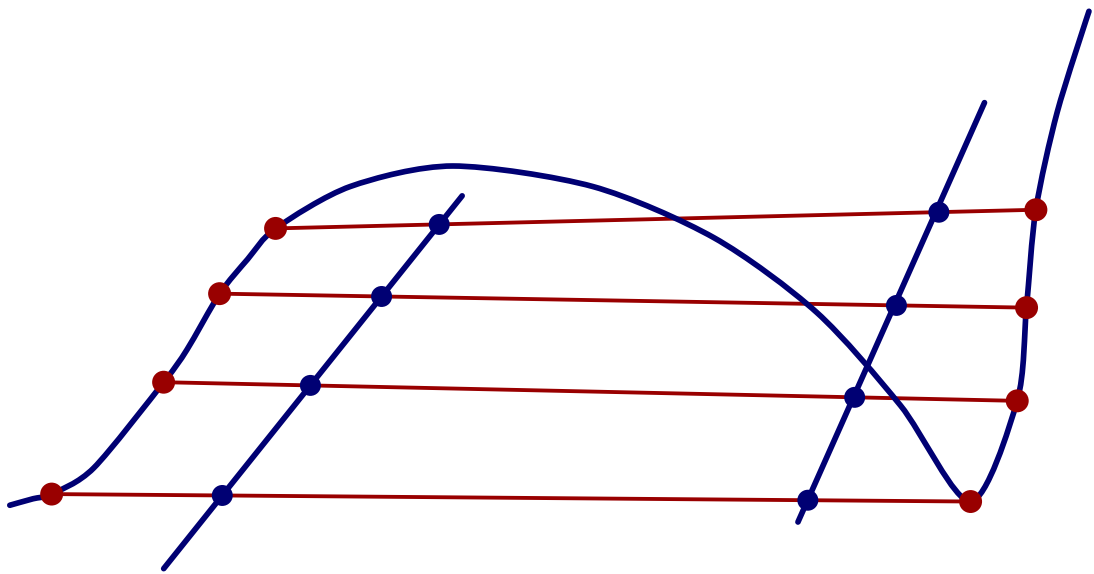

We outline the solution using classical Grassmannian Schubert calculus. The secant variety $\operatorname{Sect}(C) \cong \PP^2$ naturally embeds into the Grassmannian $\GG(1,3):= Gr(2,4)$ of lines in $\PP^3$. The image of $\GG(1,3)$ under the Plücker embedding $p$ is the Plücker quadric in $\PP^5$. Thus, we have the composite 
\[\iota: \operatorname{Sect}(C) \to \GG(1,3) \xrightarrow{p} \PP^5.\]

Fix a general line $L_0 \subseteq \PP^3$. Let $H_{L_0}\subseteq \PP^5$ be the Pl\"ucker hyperplane corresponding to the Schubert divisor
\[
\sigma_1(L_0)\ :=\ \{\ell\in \GG(1,3)\mid \ell\cap L_0\neq\varnothing\}.
\]
Equivalently, $p^*H_{L_0}=\sigma_1(L_0)$, so a hyperplane section of the Pl\"ucker quadric parametrizes lines meeting $L_0$. Restricting to $\operatorname{Sect}(C)$, the pullback $\iota^*H_{L_0}$ is the locus of secant lines to $C$ that meet $L_0$. We argue that under the identification $\operatorname{Sect}(C)\cong \Sym^2(\PP^1)\cong \PP^2$, this locus is a plane conic. 

Indeed, the twisted cubic is scheme-theoretically cut out by a $3$-dimensional space of quadrics, so $|\mathcal I_C(2)|\cong \PP^2$ and the restriction map $H^0(\PP^3,\mathcal I_C(2))\to H^0(L_0,\O_{L_0}(2))$ induces a morphism $L_0\cong\PP^1\to \PP^2$ of degree $2$. Its image is a conic, and it coincides with $\iota^*H_{L_0}$. Consequently, \(
\iota^*\O_{\PP^5}(1)\ \cong\ \O_{\PP^2}(2),
\)
so $\iota$ is given by the complete linear system $|\O_{\PP^2}(2)|$ and hence identifies $\operatorname{Sect}(C)$ with the $2$-uple Veronese surface $V$ in $\PP^5$, lying on the Pl\"ucker quadric.

We now recover the classical enumerative count. Let $L_1,L_2\subset \PP^3$ be two general lines. The condition that a line $\ell\subset \PP^3$ meet $L_i$ is again a Schubert divisor class $\sigma_1(L_i)$ on $\GG(1,3)$. Write $[S]\in CH^2(\GG(1,3))$ for the class of $S:=\iota(\operatorname{Sect}(C))$. Then, $[S] = a\sigma_{2} + b\sigma_{1,1}$ for unique constants $a, b\in \ZZ$. 

To determine $a$ and $b$, we use the geometric interpretation that $\sigma_2$ parametrizes lines in $\PP^3$ through a fixed point, and $\sigma_{1,1}$ parametrizes lines in $\PP^3$ lying on a fixed plane. Since a general point $x\in \PP^3 -C$ lies on a unique secant line to $C$ \footnote{This can also be understood via the fact that $Bl_C\PP^3 \to \PP^3$ has degree $1$ away from $C$.}, we have $a = 1$. For a general plane $\Pi \in \PP^3$, the intersection $\Pi\cap C$ consists of three distinct points, and the secant lines lying on $\Pi$ are exactly the $\binom{3}{2}$ lines joining the $3$ pairs of points. Therefore, $b = 3$, and we have
\[
[S]=\sigma_2+3\sigma_{1,1}\in CH^2(\GG(1,3)).
\]
Thus, 
\begin{equation}\label{eq:twisted-cubic-schubert-count}
    \begin{split}
           \#\{\text{secant/tangent lines to }C\text{ meeting }L_1\text{ and }L_2\}
&= \deg [S]\cdot \sigma_1(L_1)\cdot \sigma_1(L_2) \\ 
&= \deg\sigma_1^2(\sigma_2+3\sigma_{1,1}) \\
& = 4.
    \end{split}
\end{equation}

\begin{theorem}
Let $C\subset \PP^3$ be a twisted cubic and let $L_1,L_2\subset \PP^3$ be two general lines. Then there are exactly $4$ secant lines to $C$ meeting both $L_1$ and $L_2$. \hfill \qed
\end{theorem}

\section{The Moduli Space of Complete Conics} \label[appendix]{apx:moduli-of-complete-conics}
In this appendix we recall the definition of the moduli stack $\overline{\mathcal{M}}_{g,n}(X,\beta)$ of stable maps and its corresponding coarse moduli space. In particular, we consider the moduli space $\Kont$.
\subsection{The Moduli Space of Stable Maps}

\begin{definition}[Moduli stack of stable maps]
Let $X$ be a given smooth projective scheme over $\mathbb C$ and $\beta \in CH_1(X)$. The Kontsevich moduli stack of stable maps, denoted by $\overline{\mathcal{M}_{g,n}}(X,\beta)$, is defined as follows.  
\begin{enumerate}
    \item For a scheme $S$, the $S$-valued points are given by isomorphism classes of $[\pi:C\to S, \{s_i|i=1,...,n\}, \sigma:C \to X]$, where $C$ is a flat, proper family of at worst nodal singular projective, connected, reduced curves over $S$ of arithmetic genus $g$, $s_i$ are disjoint sections which map to the smooth locus of each fiber and $\sigma_*[C_s]=\beta , \forall s\in S$ where $[C_s]$ is the Chow fundamental class of the fiber curve $C_s$.
    \item (Stability condition) For every $s\in S$ and every irreducible component $E \subset C_s$, we have
    \begin{itemize}
        \item If $E \cong \mathbb P^1$ and $\sigma$ maps $E$ to a point, then $E$ must contain at least $3$ special points (i.e., a marked point or a nodal point)
        \item If $E$ has arithmetic genus $1$ and $\sigma$ maps it to a point, then $E$ must contain at least $1$ special point. 
    \end{itemize}
\end{enumerate}
    
\end{definition}
\begin{remark}
    The stability condition on the fiber is equivalent to finiteness of the automorphism group of the fiber (we only consider those automorphisms which respect the sections $s_i$ and the map $\sigma$).
\end{remark}

\begin{theorem}[\cite{fulton1997notesstablemapsquantum}]
    Let $X$ be a smooth projective convex variety over $\mathbb C$. Then the moduli stack of stable maps admits a projective coarse moduli space $\overline{M}_{g,n}(X,\beta)$. 
\end{theorem}

It turns out that these moduli spaces for $g = 0$, $X = \PP^n$, and $d = 2$ have explicit descriptions as blowups.

\begin{theorem}[\cite{kiem2008modulispacestablemaps}]
    Let us consider an $n$-tuple $(f_1,...,f_n)$ of degree $d$ homogeneous binary forms in $x,y$. This induces a rational map $ \mathbb P^1\dashrightarrow  \mathbb P^{n-1}$ which is not defined precisely when all the functions $f_i$ vanish. This induces a map $\psi : \mathbb P(Sym^d(\mathbb C^2)\otimes \mathbb C^n)//SL_2 \dashrightarrow \overline{M}_{0,0}(\mathbb P^{n-1},d) $. For $d=2$, we have that $\psi$ is the inverse of a blowup. In particular, $ \overline{M}_{0,0}(\mathbb P^{n-1},2) $ is isomorphic to the blowup of the GIT quotient  $\mathbb P(Sym^2(\mathbb C^2)\otimes \mathbb C^n)//SL_2 $ at the smooth closed subscheme $Sym^2(\mathbb C^2) \times \mathbb P^{n-1}//SL_2 \cong \mathbb P^{n-1}$.
\end{theorem}
\subsection{The Moduli Space of Complete Conics}
Let us denote the moduli stack $\overline{\M}_{0,0}(\mathbb P^2,2)$ by $\M$ and the corresponding coarse moduli space as $M$. Let $H=\mathbb P(H^0(\mathbb P^2, \mathcal{O}(2))) \cong \mathbb P^5$ be the space of plane conics. Then each point on the Veronese surface $V\subset \mathbb P^5$ corresponds to double lines. We can define a map from $\M \to H$ by sending a degree $2$ stable map $[\sigma: C \to \mathbb P^2] $ to the effective $1$-cycle $\sigma_*[C]$ of degree $2$ (i.e., a conic) in $\mathbb P^2$. By universality of coarse moduli space, this induces a map \[\phi: M \to H.\] We now analyze the fibers for the three types of plane conics - smooth, reducible reduced and double lines.

    If the image effective $1$-cycle is a smooth conic $Q \subset \mathbb P^2$, then the corresponding genus $0$ stable map $\sigma: C \to \mathbb P^2$ must be a  map $\sigma:\mathbb P^1 \to \mathbb P^2$. If not, then it is a connected tree of $\mathbb P^1$. Since the image cycle is the irreducible smooth curve $Q$, only one component of the tree can map with degree $1$ (i.e., an isomorphism) and all other components must be contracted to a point (i.e., have degree $0$). But because of stability condition, a contracted component must have at least $3$ special points, which gives a contradiction to the existence of any contracted component with only $1$ node . But this implies the tree cannot have a rational tail, i.e., must be irreducible. Hence in the coarse moduli space $M$, we see that the preimage of a smooth conic is a unique point.

    Similarly, we can argue for reducible reduced conics $L_1\cup L_2 \subset \mathbb P^2$. In this case the corresponding genus $0$ stable map must be $\sigma : \mathbb P^1 \cup \mathbb P^1 \to \mathbb P^2$, where the domain is a nodal union of $\mathbb P^1$ such that each component maps isomorphically to the corresponding $L_i$. Hence the preimage in the coarse moduli space $M$ is a unique point.

    Finally consider a double line effective cycle $2L \subset \mathbb P^2$. Again by stability arguments as above, the corresponding genus $0$ stable curve here is either a $\sigma : \mathbb P^1 \to \mathbb P^2$ such that it restricts to a degree $2$ cover of $\mathbb P^1 \to \mathbb P^1$ onto its image, or a map $\sigma : \mathbb P^1 \cup \mathbb P^1 \to \mathbb P^2$ such that both the components map isomorphically to the same line $L$. In the first case, the stable map is determined up to the branch locus which are $2$ points in $\mathbb P^1$. In the second case, the node maps also to a degree $2$ effective divisor supported at a point in $\mathbb P^1$. Hence the fiber of the double line corresponds to the space of double points in $\mathbb P^1$ given by $\operatorname{Sym}^2(\mathbb P^1) \cong \mathbb P^2$.

    In \cite{Kock2006} it is shown that $M$ is a smooth scheme. By dimension count it can be proved that the codimension of $\phi^{-1}(V)\subset M$ is $1$. But due to smoothness, it is an effective Cartier divisor. Hence by universality of blowups, this induces a map from $M \to Bl_V\mathbb P^5$, which is also bijective and birational by our description. Since both the spaces are smooth projective, Zariski's main theorem implies the result $M \cong Bl_V\mathbb P^5$.

\section{\texttt{Macaulay2} code}

\subsection{$4$ secant lines to the twisted cubic meeting two general lines in $\PP^3$ simultaneously}

\begin{footnotesize}
\label[appendix]{apx:4-secant-lines-m2}
\begin{lstlisting}[language=Macaulay2output]
`\underline{\tt i1}` : -- Truncated Lazard ring
     R = ZZ[a11, a21, a12, Degrees => {0,0,0}]
`\underline{\tt o1}` = `$R$`
`\underline{\tt o1}` : `$\texttt{PolynomialRing}$nullnullnull`
`\underline{\tt i2}` : 
     -- Truncated universal formal group law
     F = (x,y) -> x + y + a11*x*y + a21*x^2*y + a12*x*y^2
`\underline{\tt o2}` = `$\texttt{F}$`
`\underline{\tt o2}` : `$\texttt{FunctionClosure}$`
`\underline{\tt i3}` : 
     -- n-series [n]_F(t)
     nSeries = (n,t) -> (
         if n==0 then 0
         else (
             s := t;
             for k from 2 to n do (
                 s = F(s,t);
                 s = s % ideal(t^4);  -- truncate
             );
             s
         )
     );
`\underline{\tt i4}` : 
     -- [2]_F, [4]_F
     Ru = R[u]
`\underline{\tt o4}` = `$\mathit{Ru}$`
`\underline{\tt o4}` : `$\texttt{PolynomialRing}$nullnullnull`
`\underline{\tt i5}` : twoU  = nSeries(2,u)
`\underline{\tt o5}` = `$\left(\mathit{a21}+\mathit{a12}\right)u^{3}+\mathit{a11}\,u^{2}+2\,u$`
`\underline{\tt o5}` : `$\mathit{Ru}$`
`\underline{\tt i6}` : fourU = nSeries(4,u)
`\underline{\tt o6}` = `$\left(4\,\mathit{a11}^{2}+14\,\mathit{a21}+6\,\mathit{a12}\right)u^{3}+6\,\mathit{a11}\,u^{2}+4\,u$`
`\underline{\tt o6}` : `$\mathit{Ru}$`
`\underline{\tt i7}` : 
     -- formal inverse
     invF = (x) -> (
         y := -x;
         y = y - F(x, y);
         y = y - F(x, y);
         y = y - F(x, y);
         y = y - F(x, y);
         y
     )
`\underline{\tt o7}` = `$\texttt{invF}$`
`\underline{\tt o7}` : `$\texttt{FunctionClosure}$`
`\underline{\tt i8}` : 
     
     -- blowup ring
     -- alpha = pi^* c1^A(O_{P^3}(1)), e = j_*(1), x = j_*(eta), z = j_*(zeta).
     
     S = R[e,x,z,alpha, MonomialOrder => Lex]
`\underline{\tt o8}` = `$S$`
`\underline{\tt o8}` : `$\texttt{PolynomialRing}$nullnullnull`
`\underline{\tt i9}` : 
     I = ideal(
         alpha^4,
         alpha^2*e,
         alpha*e - 3*x,
         alpha*x,
         alpha*z - 3*alpha^3,
         z^2, x^2, x*z,
         e^2 + z + 10*a11*alpha^3,
         e*x + alpha^3,
         e*z - 10*alpha^3,
         3*alpha^2 - z - 10*x - 8*a11*alpha^3
     );
`\underline{\tt o9}` : `$\texttt{Ideal}$ of $S$`
`\underline{\tt i10}` : 
      A = S/I
`\underline{\tt o10}` = `$A$`
`\underline{\tt o10}` : `$\texttt{QuotientRing}$`
`\underline{\tt i11}` : 
      phiA = map(A, Ru, {alpha})
`\underline{\tt o11}` = `$\texttt{map}{}\left(A,\,\mathit{Ru},\,\left\{\mathit{alpha},\:\mathit{a11},\:\mathit{a21},\:\mathit{a12}\right\}\right)$`
`\underline{\tt o11}` : `$\texttt{RingMap}$ $A\,\longleftarrow \,\mathit{Ru}$`
`\underline{\tt i12}` : phiE = map(A, Ru, {e})
`\underline{\tt o12}` = `$\texttt{map}{}\left(A,\,\mathit{Ru},\,\left\{e,\:\mathit{a11},\:\mathit{a21},\:\mathit{a12}\right\}\right)$`
`\underline{\tt o12}` : `$\texttt{RingMap}$ $A\,\longleftarrow \,\mathit{Ru}$`
`\underline{\tt i13}` : 
      class4A = phiA(fourU)
`\underline{\tt o13}` = `$\left(4\,\mathit{a11}^{2}+14\,\mathit{a21}+6\,\mathit{a12}\right)\mathit{alpha}^{3}+6\,\mathit{a11}\,\mathit{alpha}^{2}+4\,\mathit{alpha}$`
`\underline{\tt o13}` : `$A$`
`\underline{\tt i14}` : 
      class2E = phiE(twoU)
`\underline{\tt o14}` = `$2\,e-\mathit{a11}\,z+\left(-10\,\mathit{a11}^{2}-10\,\mathit{a21}-10\,\mathit{a12}\right)\mathit{alpha}^{3}$`
`\underline{\tt o14}` : `$A$`
`\underline{\tt i15}` : 
      minus2E = invF(class2E)
`\underline{\tt o15}` = `$-2\,e-3\,\mathit{a11}\,z+\left(10\,\mathit{a11}^{2}-70\,\mathit{a21}+90\,\mathit{a12}\right)\mathit{alpha}^{3}$`
`\underline{\tt o15}` : `$A$`
`\underline{\tt i16}` : 
      
      -- (4F alpha - 2F e)^2 * alpha
      solution = (class4A + minus2E)^2 * alpha;
`\underline{\tt i17}` : 
      assert(solution == 4*alpha^3)
`\underline{\tt i18}` : 
      solution
`\underline{\tt o18}` = `$4\,\mathit{alpha}^{3}$`
`\underline{\tt o18}` : `$A$`
\end{lstlisting}
    
\end{footnotesize}

\subsection{$3264$ conics via cobordism}
\label[appendix]{apx:3264-conics}
\begin{footnotesize}
\begin{lstlisting}[language=Macaulay2output]
      R = ZZ[a11, a12, a22, a13]
`\underline{\tt o19}` = `$R$`
`\underline{\tt o19}` : `$\texttt{PolynomialRing}$nullnullnull`
`\underline{\tt i20}` : 
      S = R[e00, e10, e20, e01, e11, e21, e02, e12, e22, alpha,  MonomialOrder => Lex]
`\underline{\tt o20}` = `$S$`
`\underline{\tt o20}` : `$\texttt{PolynomialRing}$nullnullnull`
`\underline{\tt i21}` : 
      s = 57*a11^3 + 51*a11*a12 - 51*a22 + 102*a13
`\underline{\tt o21}` = `$57\,\mathit{a11}^{3}+51\,\mathit{a11}\,\mathit{a12}-51\,\mathit{a22}+102\,\mathit{a13}$`
`\underline{\tt o21}` : `$R$`
`\underline{\tt i22}` : 
      getE = (a, b) -> (
          if a == 0 and b == 0 then e00
          else if a == 1 and b == 0 then e10
          else if a == 2 and b == 0 then e20
          else if a == 0 and b == 1 then e01
          else if a == 1 and b == 1 then e11
          else if a == 2 and b == 1 then e21
          else if a == 0 and b == 2 then e02
          else if a == 1 and b == 2 then e12
          else if a == 2 and b == 2 then e22
          else 0
      )
`\underline{\tt o22}` = `$\texttt{getE}$`
`\underline{\tt o22}` : `$\texttt{FunctionClosure}$`
`\underline{\tt i23}` : 
      prodE = (a, b, c, d) -> (
          s := a + c;
          w := b + d;
          if s >= 3 or w >= 5 then 0
          else if s == 0 then (
              if w == 0 then -e01 + a11*e02 + 30*a11^2*e21 + 9*a11^2*e12 + s*e22
              else if w == 1 then -30*a11*e21 - e02 - 9*a11*e12 - 57*a11^2*e22
              else if w == 2 then 30*e21 + 9*e12 + 57*a11*e22
              else if w == 3 then -51*e22
              else 0
          )
          else if s == 1 then (
              if w == 0 then -e11 + a11*e12 + 9*a11^2*e22
              else if w == 1 then -e12 - 9*a11*e22
              else if w == 2 then 9*e22
              else 0
          )
          else if s == 2 then (
              if w == 0 then -e21 + a11*e22
              else if w == 1 then -e22
              else 0
          )
          else 0
      )
`\underline{\tt o23}` = `$\texttt{prodE}$`
`\underline{\tt o23}` : `$\texttt{FunctionClosure}$`
`\underline{\tt i24}` : 
      baseRel = {alpha^6}
`\underline{\tt o24}` = `$\left\{\mathit{alpha}^{6}\right\}$`
`\underline{\tt o24}` : `$\texttt{List}$`
`\underline{\tt i25}` : 
      mixedRels = {}
`\underline{\tt o25}` = `$\left\{\,\right\}$`
`\underline{\tt o25}` : `$\texttt{List}$`
`\underline{\tt i26}` : for a from 0 to 2 do (
          for b from 0 to 2 do (
              mixedRels = mixedRels | {alpha * getE(a,b) - (2*getE(a+1,b) + a11*getE(a+2,b))}
          )
      )
`\underline{\tt i27}` : 
      exceptRels = {}
`\underline{\tt o27}` = `$\left\{\,\right\}$`
`\underline{\tt o27}` : `$\texttt{List}$`
`\underline{\tt i28}` : for a from 0 to 2 do (
          for b from 0 to 2 do (
              for c from 0 to 2 do (
                  for d from 0 to 2 do (
                      exceptRels = exceptRels | {getE(a,b) * getE(c,d) - prodE(a,b,c,d)}
                  )
              )
          )
      )
`\underline{\tt i29}` : 
      excessRels = {
          (4*alpha^3 + 3*a11*alpha^4 + 3*a12*alpha^5) - (30*e20 + 9*e11 + e02 + 66*a11*e21 + 9*a11*e12 + 78*a11^2*e22),
          (2*alpha^4 + a11*alpha^5) - (9*e21 + e12 + 9*a11*e22),
          alpha^5 - e22
      };
`\underline{\tt i30}` : 
      I = ideal join(baseRel, mixedRels, exceptRels, excessRels);
`\underline{\tt i31}` : 
      ABlVP5 = S / I
`\underline{\tt o31}` = `$\mathit{ABlVP5}$`
`\underline{\tt o31}` : `$\texttt{QuotientRing}$`
`\underline{\tt i32}` : use ABlVP5
`\underline{\tt o32}` = `$\mathit{ABlVP5}$`
`\underline{\tt o32}` : `$\texttt{QuotientRing}$`
`\underline{\tt i33}` : 
      assert(e11 * e11 == 0)
`\underline{\tt i34}` : assert(alpha * e22 == 0)
`\underline{\tt i35}` : assert(alpha^5 == e22)
`\underline{\tt i36}` : 
      F = (x,y) -> x + y + a11*x*y + a12*(x^2*y + x*y^2) + a13*(x^3*y + x*y^3) + a22*x^2*y^2
`\underline{\tt o36}` = `$\texttt{F}$`
`\underline{\tt o36}` : `$\texttt{FunctionClosure}$`
`\underline{\tt i37}` : 
      invF = (x) -> (
          y := -x;
          y = y - F(x, y);
          y = y - F(x, y);
          y = y - F(x, y);
          y = y - F(x, y);
          y
      )
`\underline{\tt o37}` = `$\texttt{invF}$`
`\underline{\tt o37}` : `$\texttt{FunctionClosure}$`
`\underline{\tt i38}` : 
      nF = (n, x) -> (
          y := 0_ABlVP5;
          for i from 1 to n do y = F(y, x);
          y
      )
`\underline{\tt o38}` = `$\texttt{nF}$`
`\underline{\tt o38}` : `$\texttt{FunctionClosure}$`
`\underline{\tt i39}` : 
      class6A = nF(6, alpha);
`\underline{\tt i40}` : class2E = nF(2, e00);
`\underline{\tt i41}` : minus2E = invF(class2E);
`\underline{\tt i42}` : T = F(class6A, minus2E);
`\underline{\tt i43}` : solution = T^5;
`\underline{\tt i44}` : 
      assert(solution == 3264*alpha^5)
`\underline{\tt i45}` : 
      solution
`\underline{\tt o45}` = `$3\,264\,\mathit{alpha}^{5}$`
`\underline{\tt o45}` : `$\mathit{ABlVP5}$`
\end{lstlisting}
\end{footnotesize}
\end{appendices}

\textsc{Department of Mathematics, University of Washington, Seattle, WA
98195, USA}

\textit{Email address: }\texttt{arka1996@uw.edu}

\medskip

\textsc{Department of Mathematics, University of Washington, Seattle, WA
98195, USA}

\textit{Email address: }\texttt{zengrf@uw.edu}

\end{document}